\documentclass[onefignum,onetabnum]{siamart190516}

\usepackage{lipsum}
\usepackage{amsfonts,amssymb}
\usepackage{graphicx}
\usepackage{epstopdf}
\usepackage{algorithmic}
\usepackage{pdflscape}
\usepackage{afterpage}
\usepackage{xr}

\ifpdf
  \DeclareGraphicsExtensions{.eps,.pdf,.png,.jpg}
\else
  \DeclareGraphicsExtensions{.eps}
\fi

\newcommand\mscriptsize[1]{\mbox{\scriptsize\ensuremath{#1}}}
\newcommand{\bs}[1]{\boldsymbol{#1}}

\newcommand{\R}{{\mathbb R}}

\newcommand{\D}{{\mathbf D}}
\newcommand{\Z}{{\mathbf Z}}
\newcommand{\ES}{{\mathbf S}}
\newcommand{\VV}{{\mathcal V}}
\newcommand{\WW}{{\mathcal W}}
\newcommand{\Fix}{\mathrm{Fix}}
\newcommand{\NP}{\bs{O}}

\newcommand{\ONE}{{\mathbf 1}}
\newcommand{\Na}{{N_{\rm a}}}
\newcommand{\No}{{N_{\rm o}}}
\newcommand{\Xx}{\bs{X}}
\newcommand{\Oo}{\bs{O}}
\newcommand{\Zz}{\bs{Z}}

\newcommand{\Am}[4]{A_{#1#2}^{#3#4}}

\newcommand\Tstrut{\rule{0pt}{2.6ex}}         
\newcommand\Bstrut{\rule[-0.9ex]{0pt}{0pt}}   

\newsiamremark{remark}{Remark}
\newsiamremark{hypothesis}{Hypothesis}
\crefname{hypothesis}{Hypothesis}{Hypotheses}
\newsiamthm{claim}{Claim}

\headers{A model-independent theory of consensus and dissensus}{A.~Franci, M.~Golubitsky, A.~Bizyaeva, N.E.~Leonard}

\title{A model-independent theory of consensus and dissensus decision making\thanks{Submitted \today.
\funding{This research has been supported in part by NSF grant CMMI-1635056, ONR grant N00014-19-1-2556, ARO grant W911NF-18-1-0325, DGAPA-UNAM PAPIIT grant IN102420, and Conacyt grant A1-S-10610. This material is also based upon work supported by the National Science Foundation Graduate Research Fellowship under Grant No. (NSF grant number).}}}

\author{Alessio Franci\thanks{Department of Mathematics, National Autonomous University of Mexico, 04510 Mexico City, Mexico. (\email{afranci@ciencias.unam.mx}, \url{https://sites.google.com/site/francialessioac/})}
\and Martin Golubitsky\thanks{Department of Mathematics, Ohio State University, Columbus, OH, 43210-1174 USA.
  (\email{golubitsky.4@osu.edu}).}
\and Anastasia~Bizyaeva\thanks{Department
	of Mechanical and Aerospace Engineering, Princeton University, Princeton,
	NJ, 08544 USA. (\email{bizyaeva@princeton.edu})}
\and Naomi Ehrich Leonard\thanks{Department
	of Mechanical and Aerospace Engineering, Princeton University, Princeton,
	NJ, 08544 USA.(\email{naomi@princeton.edu})}
}

\usepackage{amsopn}

\begin{document}

\maketitle

\begin{abstract}
We develop a model-independent framework to study the dynamics of decision-making in opinion networks for an arbitrary number of agents and an arbitrary number of options. Model-independence means that the analysis is not performed on a specific set of equations, in contrast to classical approaches to decision making that fix a specific model and analyze it. Rather, the general features of decision making in dynamical opinion networks can be derived starting from empirically testable hypotheses about the deciding agents, the available options, and the interactions among them. After translating these empirical hypotheses into algebraic ones, we use the tools of equivariant bifurcation theory to uncover model-independent properties of dynamical opinion networks. The model-independent results are illustrated on a novel analytical model that is constructed by plugging a generic sigmoidal nonlinearity, modeling boundedness of opinions and opinion perception, into the model-independent equivariant structure. Our analysis reveals richer and more flexible opinion-formation behavior as compared to model-dependent approaches. For instance, analysis reveals the possibility of switching between consensus and various forms of dissensus by modulation of the level of agent cooperativity and without requiring any particular ad-hoc interaction topology (e.g., structural balance).
From a theoretical viewpoint, we prove new results in equivariant bifurcation theory. We recall that the equivariant branching lemma states that generically each axial subgroup leads to the existence of  a branch of equilibria with symmetries given by that axial subgroup. Here we construct an exhaustive list of axial subgroups for the action of $\ES_n \times \ES_3$ on $\R^{n-1}\otimes\R^{2}$, which provides the dissensus decision behaviors for $n$ agents deciding about three options or for three agents deciding about $n$ options. We also generalize this list to the action of $\ES_n \times \ES_k$ on $\R^{n-1}\otimes \R^{k-1}$, i.e., for $n$ agents and $k$ options, although without proving that in this case the list is exhaustive.
\end{abstract}



\section{Introduction}

Virtually any collective system of active agents must make decisions and, as a preliminary step, form opinions about a set of possible options. Modern humans collectively select a leader or form opinions about social and economic issues with different types of electoral and information-sharing systems~\cite{Sabatier2019,Leifeld2016}. Present and past hunter-gatherer communities collectively decide where to collect and hunt~\cite{Pacheco2019b,Maya2019}. Bees collectively decide on a new nest once the one they occupy is exhausted or overpopulated~\cite{Seeley2012,Seeley2010}. Animal groups collectively decide if and in which direction to move, e.g., when approaching two possible food sources and the movement direction determines where and what they will eat~\cite{Couzin2005a,Couzin2011,Conradt2009,Krause2002,Stephens2008}. Neurons in lower brain areas integrate sensory inputs to perform perceptual and motor behavior decision making, while neurons in higher brain areas integrate sensory-motor information to make higher-level decisions and allocate attention and other computational brain resources~\cite{Busemeyer2019,Deco2013,Francis2018,Gallivan2018,Latimer2015,Rossi-Pool2017,Musslick2019,Huguet2014,Bogacz2007,Diekman2012,Golubitsky2019,Tajima2019}. Bacteria and other social micro-organisms collectively decide, e.g., using quorum sensing, how and when to undergo phenotypic differentiation in response to environmental signals like nutrient scarcity or the presence of antibiotic~\cite{Weitz2008,Waters2005,Balazsi2011,Ferrell1998,Toda2018,Guzman2020,Bottagisio2019}. A variety of different models have been developed to understand the origin of consensus and dissensus opinion formation and decision making in a model-specific fashion~\cite{Castellano2009,Evans2018,Li2013,Stewart2019,Vasconcelos2019,Fortunato2005,Nedic2012,Parsegov2017,Altafini2013,Blondel2009,DeGroot1974,Friedkin1999,Hegselmann2002,Ye2019,Jia2015,Jia2019,Lorenz2007,Pan2018,Galam1995,Galam2007,Cisneros2019}. However, a general theoretical framework, suitable to identify basic shared mechanisms as well as system-specific differences, is largely missing.


Opinion changes in a multi-agent, multi-option decision-making process can be viewed as a transition in state of a continuous state-space dynamical system driven by parameters and time-dependent external inputs. The state represents the opinions of the agents about the options. The parameters and external inputs represent possibly time-varying features of the agents (e.g., resistance to forming or changing an opinion), options (e.g., absolute value or quality), context (e.g., an approaching time deadline or an approaching spatial boundary), or exchange (e.g., the attention the agents pay to each other). Qualitative transitions in the agents' opinion state determine changes in the decision outcome. For instance, a group of three agents (e.g., three people) can start a decision process over three options (e.g., three restaurants) with no marked initial preferences. Discussing and accumulating evidence (e.g., checking menus and prices on their phones), the three agents start to develop a preference for the first alternative (e.g., the Argentinian grill). They are on the verge of making a decision (e.g., reserving a table), when the second agent brings new evidence against the chosen option (e.g., she remembers that a friend who is joining for dinner is vegetarian). The first and third agents remain in favor of the first option, whereas the second agent becomes conflicted between the remaining two options. Depending on individual characteristics and factors associated with the setting, the agents may continue discussing to reach a consensus on one of the options or they may reach a dissensus by choosing different options (e.g., deciding to split into vegetarian and non-vegetarian groups).

The prediction of state transitions as a function of parameter variations and inputs is the subject of bifurcation theory: a bifurcation diagram expresses the set of solutions as a bifurcation parameter (or input) varies. Recent works have used bifurcation theory to model decision making in a variety of biological and social systems~\cite{Diekman2012,Gray2018,Lee2014,Leonard2012,Lucas2013,Ozcimder2019,Pais2013,Reina2017,Pinkoviezky2018,You2013,Gekle2005,Deco2013,Nabet2009,Fontan2017,Galam2007}. Figure~\ref{FIG: bif decision sketch} summarizes the general idea behind those works. The distinguished bifurcation parameter $\lambda$ rules the transition from an unopinionated to an opinionated state. The remaining auxiliary (or, unfolding) parameters determine the qualitative properties of this transition. The distinction between the two types of parameters, or inputs, calls for the use of the bifurcation theory methods developed in~\cite{Golubitsky1985,Golubitsky1985-2}.

\begin{figure}
	\centering
	\includegraphics[width=\textwidth]{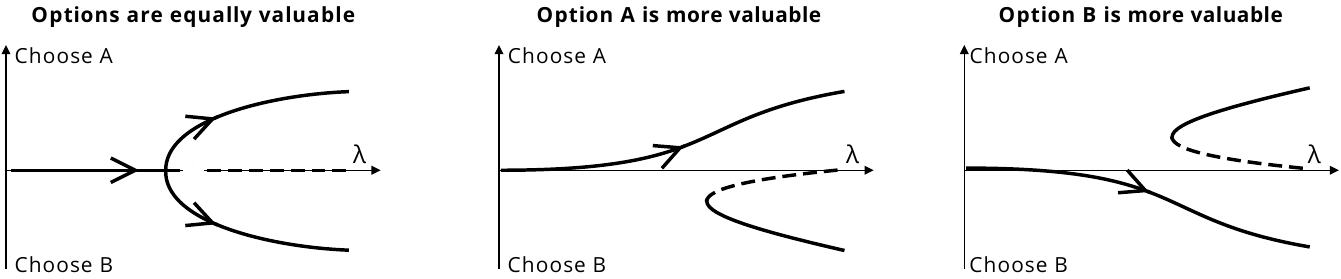}
	\caption{\footnotesize Decision making as a bifurcation phenomenon. Three bifurcation diagrams are shown with a collective opinion variable on the vertical axis and bifurcation parameter $\lambda$ on the horizontal axis. Solid (dashed) lines are stable (unstable) opinion state. Arrows show flow of the stable quasi steady-state as $\lambda$ slowly increases.  $\lambda$ models a feature of the agents, options, context, or exchange, which can change over time. The idealization that options are equally valuable leads to symmetric bifurcation problems and pitchfork bifurcations. In a pitchfork bifurcation, spontaneous symmetry breaking leads to symmetric (equally valuable) solution branches. However, when the idealizations are only approximately valid, generic forced symmetry breaking leads to asymmetric solution branches by opening (or unfolding) the pitchfork in the direction of the more valuable option.
	} \label{FIG: bif decision sketch}
\end{figure}

A bifurcation corresponds to a point in the state and parameter space where the multiplicity, the stability, or both, of solutions change.  In the idealized symmetric case of two equally valuable options, the indecisive state can lose stability through {\it spontaneous symmetry breaking} and bifurcate into two symmetric opinion-formation branches, each favoring one of the options. As illustrated in the left diagram, the symmetry of the equations does not change, but the symmetry of the bifurcating solutions do.  Which branch is followed depends on initial conditions that might, for example, be caused by random perturbations, a coin-flipping process. When options differ slightly in value, {\it forced symmetry breaking} changes the symmetry of the equations and can change the structure of the bifurcation diagram, as illustrated by the right two diagrams. This symmetry-breaking perturbation leads to asymmetric opinion-formation branches.  The symmetry of solutions is broken in favor of the more valuable option.

Near a symmetry-breaking singularity the system has a heightened sensitivity, which explains how a decision-making system can be highly responsive to changes in parameters and external inputs, such as individual biases and option values. Away from the singularity, the bifurcation diagram is robust, which explains how a decision-making system can reject disturbances. In this way, bifurcation theory reveals how decision-making systems can tune the balance between flexibility and robustness~\cite{Bizyaeva2020,Gray2018}.

We take a step toward the development of a general theoretical framework for decision making in dynamical opinion networks by leveraging the model-independent tools of equivariant bifurcation and singularity theory~\cite{Golubitsky1985-2,Golubitsky2002book}. We develop our approach for a group of equal agents deciding over a set of {\it a priori} equally valuable options, a case that is not only mathematically tractable but also important for modeling in its own right and as a means to examine messier real-world settings. The most challenging decision-making problems correspond indeed to those in which the options have near-equal value, the agents have no marked biases, hierarchies, or clusters, but where making no decision is costly. These problems are also compelling because of the flexibility they afford: an agent without bias can more quickly be won over to another decision as compared to an agent who starts with a bias.

The case of equal agents and equally valuable options is mathematically tractable because of the symmetries it possesses, namely, agent permutations and option permutations, which admit the use of equivariant bifurcation theory for analysis. From a singularity theory perspective, the highly symmetric case constitutes an {\it organizing center} that determines or {\it organizes} the model behavior, {\it even when the symmetry assumptions are violated}, i.e., when agents are not equal and options are not equally valuable. Equivariant bifurcation theory applied at a decision-making organizing center leads to a list of generic, i.e., likely, opinion-formation patterns emerging through spontaneous symmetry-breaking and identifies the key parameters modulating the deciding group behavior. Further, it is a general principle underlying the use of organizing centers in mathematical modeling that the model-independent results are robust to generic perturbations, which makes them especially practical.

The main model-independent results we prove are the following:
\begin{itemize}
	\item {\bf Consensus or dissensus}. Generically, opinion formation from the unopinionated state will only lead either to a consensus state, where all the agents share the same opinion, or to a dissensus state, where the agents disagree in such a way that on average the group remains (close to) unopinionated.
	The appearance of either consensus or dissensus is simply determined by the balance between agent cooperation (how much an agent follows other agents' opinions) and agent competition  (how much an agent rejects other agents' opinions), even in a fully homogeneous, all-to-all coupling topology. This is in sharp contrast to existing model-dependent results that state the necessity of specific network structures (e.g., structural balance)~\cite{Altafini2013,Pan2018,Cisneros2019} or specific initial conditions~\cite{Fortunato2005,Nedic2012,Parsegov2017,Friedkin1999,DeGroot1974,Hegselmann2002,Ye2019,Lorenz2007}, for stable dissensus to be possible\footnote{Notable exceptions are~\cite{Dandekar2013,Xia2020}, where the proposed models get closer to the richness of the model-independent theory by predicting dissensus opinion formation for all-to-all homogeneous coupling.}.
	\item {\bf Uniform and moderate/extremist dissensus}. Dissensus opinion formation happens in only two generic ways: 1) uniform dissensus or 2) moderate/extremist dissensus. In uniform dissensus, the agent group splits into equally sized clusters, each favoring a different option with the same intensity. In moderate/extremist dissensus, the agent group splits into two clusters of different sizes. Agents in the larger cluster, the moderates, have weak opinions, while agents in the smaller cluster, the extremists, have strong opinions in opposition to the moderates.
	\item {\bf Switch-like or continuous opinion formation}. The way in which a deciding group develops opinions can either be continuous, i.e., the magnitude of the agents' opinion state varies slowly and continuously as a function of time and other ruling parameters, or switch-like, i.e., the magnitude of agents' opinion state varies abruptly and almost discontinuously when certain thresholds in time and/or in the parameter space are reached. A prediction of our theory is that the continuous or switch-like nature of opinion formation is mainly determined by the number of agents and the number of option. Consensus opinion formation is always switch-like when the number of options is larger than or equal to three. In the case of two options, consensus opinion formation can either be continuous or switch-like, depending on the nature of the organizing singularity.  Dissensus opinion formation is always switch-like in the presence of three or more options and three or more agents. When either the number of agents or the number of options is equal to two, dissensus opinion formation can be continuous or switch-like, again depending on the nature of the organizing singularity.
\end{itemize}
A fundamental contribution of the companion paper~\cite{Bizyaeva2020}, which we borrow in the present paper, is the construction of a new general opinion-formation analytical model. This model generalizes a number of published models, in the sense that it recovers them for specific parameter choices and/or when linearized. Here, we use the general analytical model to numerically illustrate the model-independent results listed above. We further explain when and how the model-independent results depend on alternative assumptions that can be made about the model symmetries. Alternative symmetries allow modelling distinct and specific ways in which agents and their opinions can interact, e.g., the presence of structures, like cycles and clusters, in the opinion network and its topology.

From an equivariant bifurcation theory perspective, consensus and dissensus are two different {\it irreducible representations} of the symmetry group of the opinion-formation dynamics.
The possible ways in which agents agree in a consensus state is determined by the {\it axial subgroups} of the consensus irreducible representation. Generically, at a consensus bifurcation from an unopinionated state, the agents select a cluster of $p$ favored options that they favor over the remaining $\No-p$ disfavored options, where $\No$ is the total number of options. The smaller is $p$ the larger is the preference that the agents assign to the favored options. We briefly discuss the possibility of {\it secondary bifurcations} through which the agents could sequentially converge from $p$ options to a single favored option, for example, as a sequence of binary choices as has been studied for spatial decision making in the plane~\cite{Pinkoviezky2018}.
The possible ways in which agents disagree in a dissensus state is determined by the {\it axial subgroups} of the dissensus irreducible representation. Uniform and moderate/extremist dissensus correspond in particular to two different types of axial subgroups. More specifically, we prove that there are only three conjugacy classes of dissensus axial subgroups, two of which correspond to uniform dissensus and one to moderate/extremist dissensus.


\subsubsection*{Paper organization}
The paper is organized as follows. In Section~\ref{sec:model} we introduce a general setting to study opinion networks as smooth dynamical systems and present all the model-independent results in a non-rigorous way, with only a minimum of mathematical jargon. Section~\ref{ssec: applications} discusses the relevance of the presented theoretical framework for many important real-world applications. Section~\ref{SEC: preliminaries} introduces the basic concepts and results from equivariant bifurcation theory, underlying all the model-independent results, and present the first model-independent result, namely, the generic occurrence of consensus or dissensus opinion formation under suitable symmetry assumptions. In Section~\ref{SEC: examples} we revisit known results in equivariant bifurcation theory in the opinion formation setting and describe a list of low-dimensional opinion formation behaviors. Section~\ref{SEC:main results} presents the main theoretical results of the paper and their interpretation in terms of opinion dynamics. In particular, it proves the genericity of uniform and moderate/extremist dissensus and the switchy vs continuous nature of opinion formation depending on the number of agents and the number of options. Some conclusions and future directions are discussed in Section~\ref{sec:conclusions}.

\section{A dynamical systems theory of dynamical opinion networks}
\label{sec:model}

\subsection{Setting up the state space}
\label{ssec:state space}

We consider a network of $\Na\geq 1$ agents that form an opinion about $\No\geq 2$ options. The state of agent $i$ is denoted by the vector $\Xx_i = (x_{i1},\ldots,x_{i\No})\in\R^\No_{\geq 0}$. The positive number $x_{ij}$ represents the {\em opinion} of agent $i$ about option $j$. The larger $x_{ij}$ is, the greater the preference that agent $i$ has for option $j$. The entries of the vector $\Xx_i$ satisfy the $\No-1$ dimensional simplex conditions $x_{ij} \ge 0$ and
\begin{equation}  \label{e:total_vote}
x_{i1} +  \cdots + x_{i\No} = 1.
\end{equation}
Constraint \eqref{e:total_vote} is based on the assumption that each agent has the same voting capacity to distribute among the various options and this capacity is normalized to one.  Hence, each agent's state space is the $(\No-1)$ dimensional simplex $\Delta$, and the state space of the dynamical opinion network is 
\begin{equation}\label{EQ: model state space}
\VV = \underbrace{\Delta\times\cdots\times \Delta}_{\Na\text{-times}},
\end{equation}
which has dimension $\Na(\No-1)$. A point in the state space $\VV$ is denoted by the vector $\Xx=(\Xx_1,\ldots,\Xx_\Na)$. The {\em neutral point} $\Oo\in\VV$ is the unique state where each agent assigns the same proportional vote to each option; that is, $\Oo=(\NP_1,\ldots,\NP_\Na)$ where
\begin{equation}\label{e:NP}
\NP_i = \left(\frac{1}{\No},\ldots,\frac{1}{\No}\right)\in \Delta.
\end{equation}
At the neutral point, all options are equally preferred by all agents.

Mathematical analysis will be performed in the translated variables defined by state variations around the neutral point
\begin{equation}\label{EQ:linearized vars}
\Zz=\Xx-\Oo.
\end{equation}
Let
\begin{equation}\label{EQ: zero sum subspace}
V_n=\{y\in\R^n:\, y_1+\cdots+y_n=0\}\cong\R^{n-1}.
\end{equation}
Then $\Zz_i\in V_{\No}$ for all $i=1,\ldots,\Na$. Observe, moreover, that $T_{\Oo_i}\Delta=V_\No$, that is, working with the shifted variable $\Zz$ is equivalent to working in $V=T_{\Oo}\VV$, i.e., the tangent space to $\VV$ at $\Oo$. The linearized state space $V$ decomposes as 
\begin{equation}\label{EQ: tensor product structure}
V=\underbrace{V_\No\times\cdots\times V_\No}_{\Na\text{ times}}\cong(\R^{\No-1)^\Na}\cong\R^{\No-1}\otimes\R^\Na\,.
\end{equation}

\begin{remark}
Developing the analysis in the linearized variables is justified by the fact that only an agent's relative opinions about the various options matter because its total voting capacity is fixed. For interpretation and visualization purposes, variables will be mapped back to the simplex.
\end{remark}

In our model, agent $i$ is {\em unopinionated} when its opinion state lies in a $\vartheta$-neighborhood of the neutral point, i.e. $\|\Xx_i-\Oo_i\|\leq \vartheta$, where $\vartheta\geq 0$ is a given threshold. Agent $i$ is {\em opinionated} otherwise, i.e., when $\|\Xx_i-\Oo_i\|> \vartheta$. The continuous nature of the state space allows to quantify the strength of an opinion. In particular, agent $i$ is {\em weakly opinionated} if $\|\Xx_i-\Oo_i{\|\approx} \vartheta$ and {\em strongly opinionated} if $\|\Xx_i-\Oo_i\|\gg \vartheta$. Agent $i$ {\em favors} an option $j$ if it is opinionated and $x_{ij}\geq x_{il}-\vartheta$ for all $l\neq j$. Agent $i$ {\em disfavors} an option $j$ if it is opinionated and $x_{ij}< x_{il}-\vartheta$ for some $l\neq j$. An agent is {\em conflicted} among a set of options if it has near-equal and favorable opinions about all of them relative to the options not in the set. Note that with these definitions the state space of each agent is partitioned into qualitatively different, non-overlapping, opinion states. \Cref{FIG:un-decided}a illustrates the definitions for $\No=2$ and $\Na=3$.

\begin{figure}[h!]
	\centering
	\includegraphics[width=0.85\textwidth]{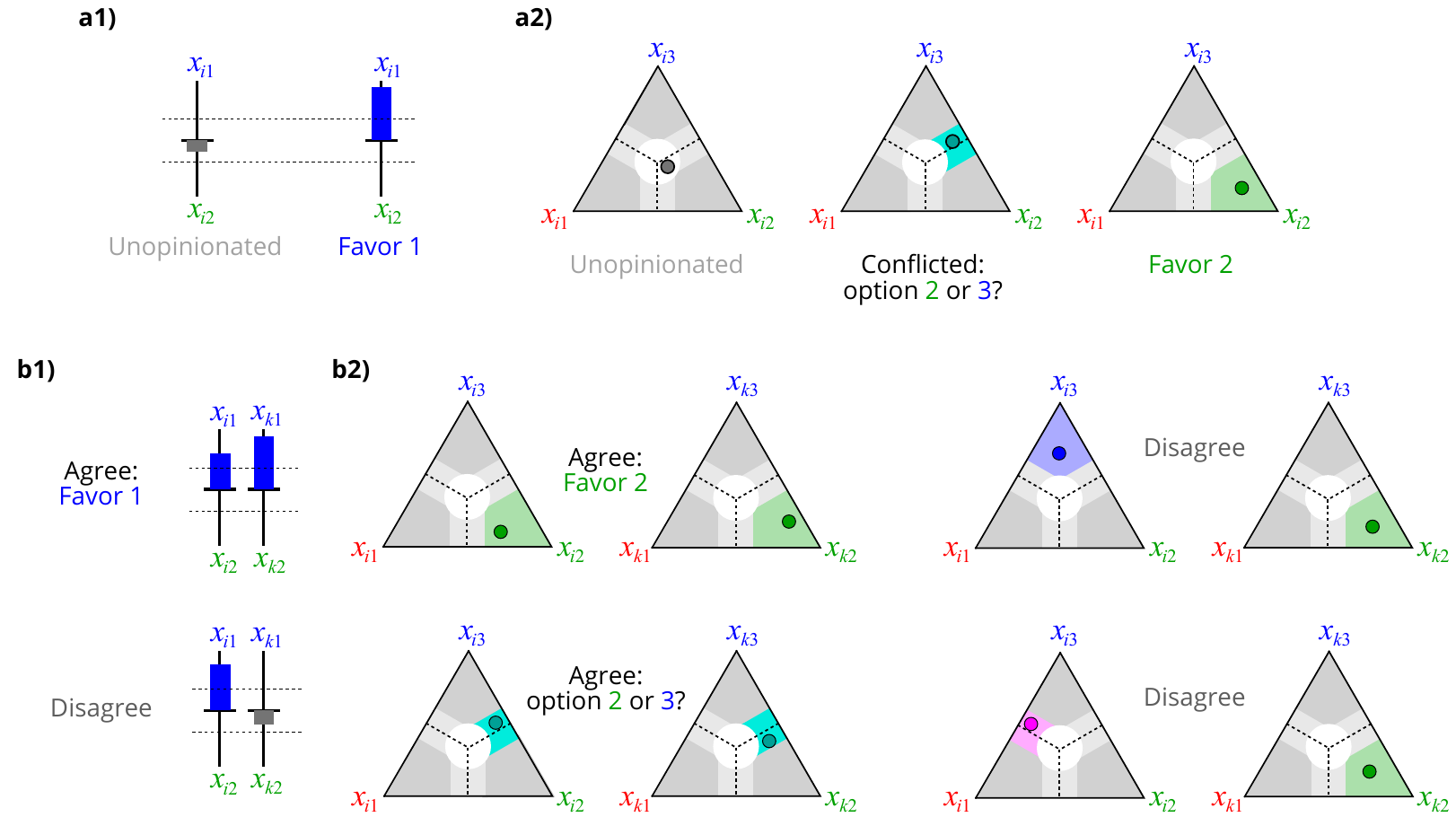}
	\caption{ {\bf a)} State space and opinion states of one agent $i$ for $\No=2$ ({\bf a1}) and $\No=3$ ({\bf a2)}). In {\bf a1)}, the 1-simplex is the segment $\Delta\subset\R^2$ with vertices $(1,0),\,(0,1)$. The mid-point of the segment is the neutral point $\NP_i$. If the agent's state lies in a neighborhood of $\NP_i$, the agent is unopinionated (region between the dashed lines). In {\bf a2)}, the 2-simplex is the triangle $\Delta\subset\R^3$ with vertices $(1,0,0),\,(0,1,0),\,(0,0,1)$. The center of the triangle is the neutral point $\NP_i$. If the agent's state lies in a neighborhood of $\NP_i$, the agent is unopinionated. The 
	dashed lines in the triangle indicate points at which two options are exactly equally favored. If the agent's state belongs to a tubular neighborhood of these lines, the agent is conflicted between the corresponding two options (light gray or color-coded depending on the conflicting options). In the remaining three regions (gray or color-coded depending on the option), the agent favors one of the three options. {\bf b} Agreement and disagreement states of two agents $i$ and $k$ for $\No=2$ ({\bf b1}) and $\No=3$ ({\bf b2}).}\label{FIG:un-decided}
\end{figure}

 Based on their opinions, multiple agents can either {\em agree} or {\em disagree}. We illustrate this concept in \Cref{FIG:un-decided}{b1}, for the case $\No=2$, and \Cref{FIG:un-decided}{b2}, for the case $\No=3$. Two agents agree when they share the same qualitative opinion state and disagree when they exhibit different qualitative opinion states. Observe that agents can agree on favoring one of the options or they can agree by remaining conflicted between a set of options. The group reaches an {\em agreement} state when all agents are opinionated and agree. If the group is in agreement and, moreover, $\|\Zz_i-\Zz_j\|\leq\vartheta$ for all $i,j\in\{1,\ldots,\Na\}$, then the group has reached a {\it consensus} state. If at least one pair of agents disagree, then the group is in a {\em disagreement} state. If the group is in disagreement and the average agent is unopinionated, i.e., $\left\|\frac{1}{\Na}\sum_{i=1}^{\Na}\Zz_i\right\|\leq\vartheta$, then the group is in a {\em dissensus} state. Finally, the group is {\it unopinionated} if all agents are unopinionated.
 
 \begin{remark}
Although motivated by collective decision-making problems, we avoid providing a formal definition of ``decision", ``deadlock" and other related terms. Those concepts are indeed context-dependent. For instance, in a ``democratic" decision-making context, one could define ``decision" as the state where at least a majority of agents agree and ``deadlock" otherwise. In a task allocation problem, however, any opinion state could correspond to a valid decision. For instance, if all agents are unopinionated, then the group has decided not to develop any task, whereas if certain agents favor certain options they have decided to address the corresponding tasks while the rest of the group does nothing. We will thus restrict the use of the two concepts of decision and deadlock only to specific examples.
 \end{remark}

\subsection{Collective opinion dynamics and symmetry}
\label{ssec: collective decision and symmetry}

Given the definitions introduced in Section~\ref{ssec:state space}, we want to answer the following questions. How can a group of equal agents become opinionated about a set of {\it a priori} equally valuable options? Are there prototypical opinion-forming behaviors that are generically observed in such a dynamical opinion network? How does consensus or dissensus arise? How do the opinion dynamics depend on system parameters and inputs, such as the number of agents and options, the interconnection topology, and the presence of external signals, such as agents' biases toward certain options? To answer these questions, we use equivariant bifurcation theory~\cite{Golubitsky1985-2,Golubitsky2002book} and model the time evolution of each agent's opinion as a smooth dynamical system
\begin{equation}\label{e:symm bifu dyna}
\dot{\Xx} = G({\Xx},\lambda),
\end{equation}
where $\lambda\in\R$ is the {\it bifurcation parameter}.
\begin{remark}\label{rmk:discrete time}
	In some applications it might be convenient to model the evolution of opinions in time as a discrete process~\cite{Hegselmann2002,Krause2002,Friedkin1999,DeGroot1974}, i.e., according to a smooth map
	$\Xx_{k+1}=G_d(\Xx_k,\lambda)$.
	Our results generalize to this setting by applying the equivariant bifurcation theory to the discrete-time steady-state equation $G_d(\Xx,\lambda)-\Xx=0$ instead of the continuous-time steady-state equation $ G({\Xx},\lambda)=0$. One must however be careful in applying equivariant bifurcation theory to discrete-time dynamics as intrinsically discrete dynamical phenomena, like the appearance of discrete periodic orbits, are not trivially captured by the theory. We will not address discrete-time dynamics in the remainder of the paper.
\end{remark}

The bifurcation parameter $\lambda$ can model a variety of features and parameters depending on the context, for instance, the strength of information exchange between agents (e.g., the attention paid to each other), the evolution of time toward a decision deadline (like an upcoming political election), the changing relative spatial position of a group of moving agents and an object in the environment, other influences exerted by environmental signals (e.g., the scarcity of food), or a combination of them. When these factors are sufficiently strong, we expect the neutral point and close-by unopinionated states to become unstable. Having no clear opinion becomes unsustainable or too costly and at least some of the agents transition to an opinionated state. The system {\em bifurcates} from an unopinionated state to a qualitatively new opinion state. State-dependent (feedback) dynamics for $\lambda$ are derived and analyzed in~\cite{Bizyaeva2020}.

Equivariant bifurcation theory is a {\it model-independent} theory, in the sense that precise predictions about a modeled dynamical behavior can be formulated solely based on empirical assumptions, without declaring any specific model. The first assumption we have already implicitly made is that collective opinion dynamics can be modeled as a smooth dynamical system in which opinions are valued in the simplex. The other assumptions we make concern the model {\it symmetries}. Mathematically, two agents are equal if interchanging them does not change the behavior of the opinion dynamics. In other words, two agents are equal if the opinion dynamics are symmetric with respect to interchanging the two agents. Similarly, two options are {\it a priori} equally valuable if the opinion dynamics are symmetric with respect to interchanging the two options. The most symmetric case is when all agents are equal (for instance, there are no hierarchies as in a ``democratic" decision process) and all options are {\it a priori} equally valuable. In this case, model~\eqref{e:symm bifu dyna} has symmetry group
\begin{equation}\label{e:symmetry group}
\Gamma= \ES_{\Na}\times \ES_{\No},
\end{equation}
where the permutation group $\ES_{\Na}$ interchanges the agents and $\ES_{\No}$ interchanges the options. The action of $\Gamma$ on the opinion dynamics is rigorously defined in Section~\ref{SSEC: conse disse bifu}.

The opinion dynamics~\eqref{e:symm bifu dyna} have symmetry $\Gamma=\ES_\Na\times \ES_\No$ in the sense that they are {\it $\Gamma$-equivariant}:
\begin{equation}\label{EQ:decision gamma equivariance}
\gamma G(\Xx,\lambda)=G(\gamma\Xx,\lambda),\quad\forall \gamma\in\Gamma,\ \forall\lambda\in\R.
\end{equation}
$\Gamma$-equivariance is equivalent to requiring that the symmetry group $\Gamma$ sends solutions of~\eqref{e:symm bifu dyna} into solutions of~\eqref{e:symm bifu dyna}. The fully symmetric situation is interesting because it provides a mathematically tractable organizing center to uncover generic properties of opinion-forming behaviors. As we shall verify with numerical simulations, the predictions derived at an organizing center remain valid when the symmetry assumption is violated, for instance, by introducing random heterogeneities across the agents, the options, and their interactions.

\begin{remark}\label{RMK: coopetitive 1}
Qualitative differences between the agents and between the options, for instance, due to existence of clusters with sharply distinct prior biases or the presence of structured {\it coopetitive} (i.e., both positive and negative) interactions~\cite{Altafini2013}, can also be modeled in the equivariant setting by requiring that the symmetry group of the opinion dynamics is a subgroup of $\ES_\Na\times\ES_\No$. The same techniques and results from equivariant bifurcation theory used here for the fully symmetric case also apply to the case in which part of the symmetry is lost due to biases or structured intra- and inter-agent interactions.
\end{remark}


Observe that because at $\Oo$ all agents have the same state and all options are equally favored, $\gamma\Oo=\Oo$ for all $\gamma\in\Gamma$, that is, the neutral point is {\it fixed} by the symmetry group of the model. Observe also that no other point in the state space is fixed by {\it all} the group's elements. In other words, the neutral point is the most symmetric point in the state space, where all agents share exactly the same opinion and no options are favored or disfavored over the others. We will show that a dynamical opinion network can leave this highly symmetric, completely unopinionated, state by means of {\it spontaneous symmetry breaking}~\cite{Golubitsky2002book}.

\subsection{A vector field realization}
\label{ssec: state space real}
For numerical illustration and exploration purposes, we will rely on the analytical model introduced in the companion paper~\cite{Bizyaeva2020}. For simplicity, this realization was set up on the tangent space $V$. The resulting vector field and dynamical behaviors can always be mapped to the product of simplexes $\VV$ by an affine change of coordinates, provided suitable boundedness conditions on the dynamics on $V$. The proposed opinion formation dynamics is\footnote{In the proposed model, either or both of the sums over the agents could be swapped with the sigmoidal functions. This modification does not change neither the interpretation of the model nor the validity of model-independent predictions. However, this and other types of modifications could change model-independent properties of the dynamics, like the stability of some bifurcation branches.}
\begin{subequations}\label{EQ:generic decision dynamics}
	\begin{align}
	\dot z_{ij}&=F_{ij}(\Zz)-\frac{1}{\No}\sum_{l=1}^{\No}F_{il}(\Zz)\\
	F_{ij}(\Zz)&=-z_{ij}+
	\lambda\left(S_1\left(\sum_{k=1}^\Na \Am{i}{k}{j}{j} z_{kj}\right)+\sum_{\substack{l\neq j\\l=1}}^\No S_2\left( \sum_{k=1}^\Na \Am{i}{k}{j}{l} z_{kl}\right)\right)+b_{ij}
	\end{align}
\end{subequations}
where $S_{1,2}:\R\to\R$ satisfies $S_{1,2}(0) = 0$ and $S_{1,2}'(0) = 1$ and is otherwise a generic\footnote{In the sense that all the coefficients of its Taylor expansion at the origin are non-zero, or at least up to a given degree.} bounded sigmoidal nonlinear function. It follows by the results in \cite[Section~III.B]{Bizyaeva2020} that~\eqref{EQ:generic decision dynamics} are a well-defined, bounded dynamics on $V$. Boundedness also implies that the resulting vector field on $V$ can be mapped to the full state space $\VV$ using a simple affine change of coordinates \cite[Corollary III.5.1]{Bizyaeva2020}.

The {\it adjacency tensor} $\Am{}{}{}{}$ in~\eqref{EQ:generic decision dynamics} defines that structure of the opinion-influence network. If $\Am{i}{k}{j}{l}$ is not zero, then the opinion of agent $i$ about option $j$ is influenced by the opinion of agent $k$ about option $l$. The sign and magnitude of $\Am{i}{k}{j}{l}$ defines the sign and strength of these opinion interactions. The numbers $b_{ij}$ define the prior bias of agent $i$ for option $j$. Model~\eqref{EQ:generic decision dynamics} recover as special cases a number of nonlinear and linear opinion formation models. See~\cite{Bizyaeva2020} for details.

Observe that the adjacency tensor $\Am{}{}{}{}$ specifies distinct intra- and inter-agent option-influence coupling topologies. 1) The intra-agent opinion coupling topology of agent $i$ is specified by the adjacency matrix $\Am{i}{i}{j}{l}$. It describes the influence network between the opinions of agent $i$. 2) The inter-agent opinion-influence coupling topology related to options $j$ and $l$ is specified by the adjacency matrix $\Am{i}{k}{j}{l}$. It describes the inter-agent network of influence of the agents' opinions about option $l$ over the agents' opinions about option $j$.

Prior biases and the opinion interaction topology determine the symmetry of model~\eqref{EQ:generic decision dynamics}.
\begin{theorem}\label{THM:realization symmetries}
Model~(\ref{EQ:generic decision dynamics}) is $\Gamma$-equivariant if and only if $b_{ij}=b\in\R$, $\Am{i}{i}{j}{j}=\alpha\in\R$, $\Am{i}{i}{j}{l}=\beta\in\R$, $\Am{i}{k}{j}{j}=\gamma\in\R$,  and  $\Am{i}{k}{j}{l}=\delta\in\R$ for all $i,k=1,\ldots,\Na$, $k\neq i$, and all $j,l=1,\ldots,\No$, $l\neq j$.
\end{theorem}
\begin{proof}
 Sufficiency simply follows by verifying \eqref{EQ:decision gamma equivariance}, which leads to lengthy but straightforward computations that we omit. Necessity will be proved in Section~\ref{SSEC: balance determines}.
\end{proof}

\begin{remark}
	Under the assumption of $\Gamma$-equivariance, the drift term $F_{ij}(\Zz)$ in~\eqref{EQ:generic decision dynamics} reduces to
		\begin{align}\label{EQ:generic decision dynamics all}
		F_{ij}(\Zz)&=-z_{ij}+
		\lambda\left(S_1\left(\alpha z_{ij}+\gamma\sum_{\substack{k=1\\k\neq i}}^\Na z_{kj}\right)+\sum_{\substack{l=1\\l\neq j}}^\No S_2\left(\beta z_{il}+\delta \sum_{\substack{k=1\\k\neq i}}^\Na z_{kl}\right)\right)+b_{ij}
		\end{align}
\end{remark}

\Cref{THM:realization symmetries} translates symmetry properties into readily verifiable and experimentally controllable properties, like network topology and agents' biases. In particular, $\Gamma$ equivariance is equivalent to asking that the coupling topology is {\it all-to-all homogeneous} in the sense that all the option-influence coupling topologies specified by the adjacency tensor $\Am{}{}{}{}$ are all-to-all and homogeneous. In practice, these assumptions do not need to be verified exactly for the predictions of our theory to remain valid.

\begin{remark}\label{RMK: coopetitive 2}
	In general, the symmetry group of dynamics~(\ref{EQ:generic decision dynamics}) is determined by the automorphism group of the multi-graph associated to the tensor adjacency matrix $A$. For instance, if the inter-agent topology is of dihedral type, e.g., an undirected ring, then $\Gamma=\D_{\Na}\times\Gamma_o$, with $\Gamma_o\subset\ES_{\No}$. See \cite[Proposition IV.2]{Bizyaeva2020}.
\end{remark}

\begin{remark}
	In the lowest dimensional case of two agents deciding about two options, model~\eqref{EQ:generic decision dynamics} is a universal unfolding of $\ES_2\times\ES_2$-equivariant singularities. This observation originally motivated the functional form of model~\eqref{EQ:generic decision dynamics}.
\end{remark}

In the remainder of the paper, we use model~(\ref{EQ:generic decision dynamics all}), with $S_1(x)=\tanh(x+k_{\rm h.t.o.}\tanh(x^2))$ and $S_2(x)=0.5\tanh(2x+2k_{\rm h.t.o.}\tanh(x^2))$, $k_{\rm h.t.o.}\neq 0$, to numerically illustrate our theoretical predictions. The non-degeneracy condition $k_{\rm h.t.o.}\neq 0$ ensures that all the derivatives of $S_1$ and $S_2$ at the origin are non-zero. To highlight the robustness of our predictions and their practical validity, all the simulations are performed in perturbed symmetric conditions. That is, we do not use an exact homogeneous all-to-all topologies and we allow the virtual agents to have small heterogeneous biases.

\subsection{Consensus and dissensus arise from symmetry}
\label{ssec: conse pola symmetry}

Given the symmetry assumptions, namely that the opinion-forming dynamics are $\Gamma$-equivariant, equivariant bifurcation theory predicts that opinion forming follows two generic types of paths, or {\it bifurcation branches}, as sketched in Figure~\ref{FIG:consensus_polar}. One type is tangent at $\Oo$ to the {\em consensus space}
\begin{equation}\label{e:agreement space}
W_c=\Bigg\{\bigg(\underbrace{v,\ldots,v}_{\Na\text{-times}}\bigg),\, v\in V_\No\Bigg\}
\end{equation}
and the other to the {\em dissensus space}
\begin{equation}\label{e:disagreement space}
W_d=\{({\Zz}_1,\ldots,{\Zz}_\Na):\, {\Zz}_1+\cdots+{\Zz}_\Na=0,\ \Zz_i\in V_\No\}.
\end{equation}
Observe that $W_c\cong\R^{\No-1}$, $W_d=V_{\No}\otimes V_{\Na}\cong\R^{\No-1}\otimes\R^{\Na-1}$, and $V=W_c\oplus W_d$.

Points on the consensus space correspond the situations in which all the agents have exactly the same opinion. Points on dissensus space correspond to the situation in which the average ag ent is neutral. Observe that, according to our definition, a consensus state is any state in a $\vartheta$-neighborhood of the consensus space outside a $\vartheta$-neighborhood of the origin and, similarly, a dissensus state is any state in a $\vartheta$-neighborhood of the dissensus space outside a $\vartheta$-neighborhood of the origin. The set of consensus states reduces exactly to the consensus space for $\vartheta=0$ and the set of dissensus states reduces exactly to the dissensus space for $\vartheta=0$. Mathematical analysis will be developed inside the consensus and dissensus spaces, that is, for the idealized fully symmetric case with $\vartheta=0$. The predictions of the theory will however be practically valid (as numerically illustrated) under weakly violated symmetry and for $\vartheta>0$.

\begin{remark}
	Some consensus states correspond to conflicted opinions, e.g., as in case c2 in Figure~\ref{FIG:consensus_polar}a, where the agents agree to exclude Option~3 but remain conflicted between Options~1 and~2. In these cases, the group can still transition to consensus on one of the options through {\it secondary symmetry-breaking bifurcations at which the conflicted equilibrium becomes unstable}. The analysis of secondary equivariant bifurcations is likely to be intractable in general, as it requires high-dimensional singularity theory. It has however been worked out in a handful cases~\cite[Chapter~X]{Golubitsky1985},\cite[Chapter~XV]{Golubitsky1985-2}. We discuss one of these examples for consensus opinion formation over three options in Section~\ref{SSEC: two three}.
\end{remark}

\begin{figure}
	\centering
	\includegraphics[width=\textwidth]{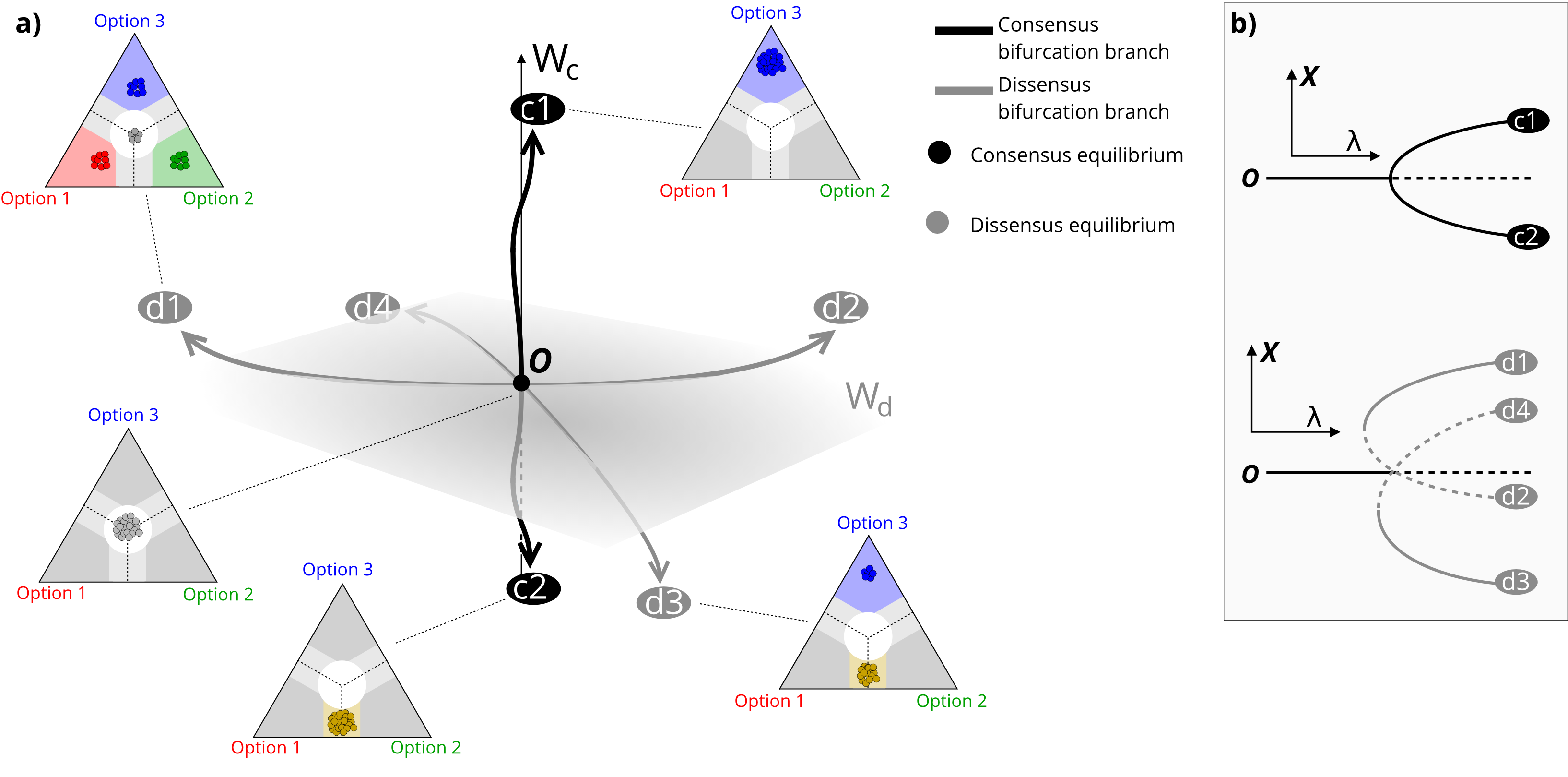}
	\caption{{The geometry of collective opinion formation. \bf a)}The consensus space $W_c$ and dissensus space $W_d$ define complementary orthogonal subspaces at the neutral point $\Oo$. Equilibria of the opinion dynamics~\eqref{e:symm bifu dyna} vary according to the bifurcation parameter $\lambda$. The neutral point $\Oo$ is always an equilibrium of~\eqref{e:symm bifu dyna}, but at a bifurcation (see {\bf b}), new equilibria generically appear along the consensus space or along the dissensus space. Along the consensus space, the group is in a consensus state: all agents have nearly the same opinion (c1,c2). In c1, the group reaches consensus on one of the options, whereas in c2 it remains conflicted between the other two options. Along the dissensus space, the group is in a dissensus state: agents remain unopinionated on average, in the sense that the average opinion is close to neutral. The are two generic types of dissensus states: uniform (d1) and moderate/extremist (d3). {\bf b)} Bifurcation diagram representation of consensus (top) and dissensus (bottom) opinion forming. At an opinion-forming bifurcation, the neutral state  $\Oo$ loses stability and new equilibria either appear along the consensus branch or along the dissensus branch. Consensus branches are tangent to the consensus space $W_c$ (see {\bf a}), whereas dissensus branches are tangent to the dissensus space $W_d$. Solid branches signify stable opinion states. Dashed branches signify unstable opinion states. The bifurcation diagrams are sketched for illustration purposes.}\label{FIG:consensus_polar}
\end{figure}

Consensus and dissensus spaces appear as important objects in the opinion dynamics~\eqref{e:symm bifu dyna} because they are the two {\it irreducible representations} \cite[Section~XII.2]{Golubitsky1985-2},\cite[Page 14]{Golubitsky2002book} of the symmetry group $\Gamma$. Certain ``genericity results" (see \cite[Theorem~1.27]{Golubitsky2002book} and \cite[p.82]{Golubitsky1985-2}) ensure that, generically, bifurcations from the neutral, unopinionated, equilibrium $\Oo$ happen either along the consensus space {\it or} along the dissensus space through spontaneous symmetry breaking. We make these statements rigorous in Theorem~\ref{THM: consensus or dissensus}, Section~\ref{SEC: preliminaries}. We can interpret these results from a modeling perspective as follows. Consensus and dissensus are the only two likely and mutually-exclusive opinion-forming outcomes of any initially unopinionated dynamical opinion network that approximately satisfies our symmetry assumptions. These assumptions can be weakly violated in practice while letting our predictions remain practically valid, a statement that we will verify with numerical simulations in the next sections.

\begin{remark}\label{RMK: consensus or dissensus transi}
	The splitting of opinion formation into consensus and dissensus bifurcations generalizes to the case in which $\Gamma=\Gamma_a\times\Gamma_o$, with $\Gamma_a\subset \ES_{\Na}$ and $\Gamma_o\subset \ES_{\No}$, such that both $\Gamma_a$ acts transitively on $\{1,\ldots,\Na\}$. See Remark~\ref{RMK: conse disse general} for details. As an example consider $\Gamma_a=\D_{\Na}$, the dihedral group of order $\Na$, and $\Gamma_o=\Z_{\No}$, the cyclic group of order $\No$. As observed in Remark~\ref{RMK: coopetitive 2}, the modeling relevance of considering different symmetry groups relies in the possibility of capturing different topologies of the opinion network defined by~\eqref{e:symm bifu dyna}.
\end{remark}

Observe that the dimensions of the consensus and dissensus spaces are different in general, as are the ways in which the symmetry group $\Gamma$ acts on the two spaces. The consensus space is generally lower dimensional and, because all agents have the same state, the permutation group $\ES_\Na$ acts trivially on this space. Opinion forming along the consensus space is thus governed by $\ES_\No$-equivariant bifurcations acting on $\R^{\No-1}$, a type of symmetry breaking that has already been studied \cite{Elmhirst2004},\cite[Sections~1.5,~2.6,~2.7]{Golubitsky2002book}. The dissensus space is generally higher dimensional and the full symmetry group acts nontrivially on this space. Opinion forming along this space is governed by $\ES_\Na\times\ES_\No$-equivariant bifurcations acting on $\R^{\Na-1}\otimes\R^{\No-1}$. These types of bifurcations have been studied only for $\Na=2$ or $\No=2$ cases~\cite{Aronson1991},\cite[Section~XIII.5]{Golubitsky1985-2},\cite[Chapter X]{Golubitsky1985-2}. We rigorously extend these results in Section~\ref{SEC:main results}.

%
%
%

\subsection{The balance between agent cooperation and agent competition selects between consensus and dissensus}
\label{SSEC: balance determines}

A question that naturally arises is what determines the occurrence of consensus bifurcations, leading to group consensus, versus dissensus bifurcations, leading to group dissensus?

Let $G$ be a $\Gamma$-equivariant vector field on $\VV$, as in~\eqref{e:symm bifu dyna}. Then the components of $G$ can be written as follows (see \cite{Stewart2003,Golubitsky2005} for details),
\begin{subequations}\label{EQ:equivariant dynamics form}
\begin{align}
G_{ij}&=\tilde G\left( x_{ij} ,\, \overline{ \{x_{i,l\neq j}\} } ,\, \overline{ \{x_{k\neq i,j} \}}  ,\, \overline{ \{x_{k\neq i,l\neq j}\} } \,,\lambda \right)-\langle\tilde G\rangle^i(\Xx,\lambda)\\
\langle \tilde G\rangle^i(\Xx,\lambda)&=\frac{1}{\No}\sum_{\bar l=1}^\No \tilde G\left( x^i_{\bar l} ,\, \overline{ \{x_{i,l\neq \bar l}\} } ,\, \overline{ \{x_{k\neq i,\bar l} \}}  ,\, \overline{ \{x_{k\neq i,l\neq \bar l}\} } \,,\lambda \right)
\end{align}
\end{subequations}
where the notation $\overline{ \{\,\cdot\,\}}$ signifies that $\tilde G$ is invariant with respect to permutations of the elements appearing in the arguments. Given the invariance properties, subtracting the average terms $\langle\tilde G\rangle^i(\Xx)$ is necessary to impose that $G_{i1}+\cdots + G_{i\No}=0$. This, in turn, ensures that each agent's voting capacity remains constant as required by the simplex condition~\eqref{e:total_vote}. The invariance properties of $\tilde G$ imply that the smooth functions
\begin{equation}\label{EQ: balance terms}
\bar\alpha(\lambda)=\left.\frac{\partial \tilde G_{ij}}{\partial x_{ij}}\right|_{\Oo}+1,\  \bar\beta(\lambda)=\left.\frac{\partial \tilde G_{ij}}{\partial x_{il}}\right|_{\Oo},\  \bar\gamma(\lambda)=\left.\frac{\partial \tilde G_{ij}}{\partial x_{kj}}\right|_{\Oo},\  \bar\delta(\lambda)=\left.\frac{\partial \tilde G_{ij}}{\partial x_{kl}}\right|_{\Oo}\,,
\end{equation}
where $k\neq i$ and $l\neq j$, are well-defined, in the sense that they do not depend on the index choices.

\begin{proof}[Proof of~\Cref{THM:realization symmetries} (continued)]
Specializing the invariance properties of a generic $\Gamma$-equivariant vector field to the state-space realization~\eqref{EQ:generic decision dynamics}, we can now prove the necessity part of Theorem~\ref{THM:realization symmetries}. Because possible affine terms in~(\ref{EQ:equivariant dynamics form}a) cannot depend on the indexes, $\Gamma$ equivariance imposes that biases are homogeneous across agents and options. Moreover, well-definiteness of the functions in~\eqref{EQ: balance terms} implies that
\begin{align*}
\lambda S_2'(0)\Am{i}{k}{j}{j}&=\left.\frac{\partial F_{ij}}{\partial z_{kj}}\right|_0=\left.\frac{\partial F_{\bar i\bar j}}{\partial z_{\bar k\bar j}}\right|_0=\lambda S_2'(0)\Am{\bar i}{\bar k}{\bar j}{\bar j},\quad\Rightarrow\quad  \Am{i}{k}{j}{j}=\Am{\bar i}{\bar k}{\bar j}{\bar j}=\gamma,\\
\lambda S_1'(0)\Am{i}{i}{j}{j}&=\left.\frac{\partial F_{ij}}{\partial z_{ij}}\right|_0=\left.\frac{\partial F_{\bar i\bar j}}{\partial z_{\bar i\bar j}}\right|_0=\lambda S_1'(0)\Am{\bar i}{\bar i}{\bar j}{\bar j},\quad\Rightarrow\quad  \Am{i}{i}{j}{l}=\Am{\bar i}{\bar i}{\bar j}{\bar j}=\alpha,\\
\lambda S_1'(0)\Am{i}{i}{j}{l}&=\left.\frac{\partial F_{ij}}{\partial z_{il}}\right|_0=\left.\frac{\partial F_{\bar i\bar j}}{\partial z_{\bar i\bar l}}\right|_0=\lambda S_1'(0)\Am{\bar i}{\bar i}{\bar j}{\bar l},\quad\Rightarrow\quad  \Am{i}{i}{j}{l}=\Am{\bar i}{\bar i}{\bar j}{\bar l}=\beta,\\
\lambda S_2'(0)\Am{i}{k}{j}{l}&=\left.\frac{\partial F_{ij}}{\partial z_{kl}}\right|_0=\left.\frac{\partial F_{\bar i\bar j}}{\partial z_{\bar k\bar l}}\right|_0=\lambda S_2'(0)\Am{\bar i}{\bar k}{\bar j}{\bar l},\quad\Rightarrow\quad  \Am{i}{k}{j}{l}=\Am{\bar i}{\bar k}{\bar j}{\bar l}=\delta,
\end{align*}
for all $i,k,\bar i,\bar k\in\{1,\ldots,\Na\}$, $i\neq k$, $\bar i\neq\bar k$ and all $j,\bar j,l,\bar l\in\{1,\ldots,\No\}$, $l\neq j,\bar l\neq\bar j$.  Thus $\Gamma$ equivariance implies that the agent coupling topology is homogenous and all-to-all.
\end{proof}

The following theorem shows that the sign and magnitude of the various numbers introduced in~\eqref{EQ: balance terms} determine whether bifurcations are likely to appear along the consensus subspace or along the dissensus subspace.
\begin{theorem}\label{THM:cooperation competition}
Let~\eqref{e:symm bifu dyna} be $\Gamma$-equivariant and let $\bar\alpha(\lambda),\bar\beta(\lambda),\bar\gamma(\lambda),\bar\delta(\lambda)$ be defined as in~\eqref{EQ: balance terms}. Suppose that there exists $\lambda_0\in\R$ such that
\begin{subequations}\label{EQ:stability conditions}
\begin{align}
\bar\alpha(\lambda_0)-\bar\beta(\lambda_0)+(\Na-1)(\bar\gamma(\lambda_0)-\bar\delta(\lambda_0))&<1\\
\bar\alpha(\lambda_0)-\bar\beta(\lambda_0)-\bar\gamma(\lambda_0)+\bar\delta(\lambda_0)&<1
\end{align}
\end{subequations}
Suppose, moreover, that there exists $\varepsilon>0$ such that $(\bar\alpha(\lambda)-\bar\beta(\lambda))'\geq \varepsilon$, for all $\lambda$. Then the following hold true.
\begin{enumerate}
	\item The neutral equilibrium is stable for $\lambda=\lambda_0$.
	\item If $\bar\gamma(\lambda)>\bar\delta(\lambda)$ and $(\bar\gamma(\lambda)-\bar\delta(\lambda))'\geq 0$, then the neutral equilibrium loses stability for $\lambda=\lambda_c^a>\lambda_0$ in a $\Gamma$-equivariant bifurcation along the consensus space, where $\lambda_c^a$ satisfies
	\[
	\bar\alpha(\lambda_c^a)-\bar\beta(\lambda_c^a)+(\Na-1)(\bar\gamma(\lambda_c^a)-\bar\delta(\lambda_c^a))=1.
	\]
	\item If $\bar\gamma(\lambda)<\bar\delta(\lambda)$ and $(\bar\gamma(\lambda)-\bar\delta(\lambda))'\leq 0$, then the neutral equilibrium loses stability for $\lambda=\lambda_c^d>\lambda_0$ in a $\Gamma$-equivariant bifurcation along the dissensus space, where $\lambda_c^d$ satisfies
	\[
	\bar\alpha(\lambda_c^d)-\bar\beta(\lambda_c^d)-\bar\gamma(\lambda_c^d)+\bar\delta(\lambda_c^d)=1.
	\]
\end{enumerate}
\end{theorem}
\begin{remark}
When $\bar\gamma\equiv\bar\delta$, bifurcations along the consensus and dissensus space happen simultaneously, a phenomenon called {\it mode interaction}. A rigorous analysis of mode interaction was worked out only in the lowest dimensional case $\Na=\No=2$~\cite[Chapter~X]{Golubitsky1985}. We do not address mode interaction further in the present work. See~\cite{Bizyaeva2020} for an analysis of opinion formation between two agents and two options.
\end{remark}
\begin{proof}[Proof]
Let $J:=\frac{\partial G}{\partial{\bf X}}(\Oo)$. A simple computation shows that
\begin{equation*}
J:=
\mscriptsize{
	\begin{bmatrix}
	B_0 & B_1 & \cdots & B_1\\
	B_1& \ddots           & \ddots  & \vdots \\
	\vdots      &\ddots  & \ddots & B_1 \\
	B_1& \cdots & B_1 & B_0
	\end{bmatrix}},\ {\rm where}\ B_0=
	\mscriptsize{
		\begin{bmatrix}
		a & b & \cdots & b\\
		b & \ddots                 & \ddots  & \vdots \\
		\vdots      &\ddots  & \ddots & b \\
		b & \cdots & b & a
		\end{bmatrix}},\ B_1=
	\mscriptsize{
	\begin{bmatrix}
	c & d & \cdots & d\\
	d & \ddots                 & \ddots  & \vdots \\
	\vdots      &\ddots  & \ddots & d\\
	d & \cdots & d & c
	\end{bmatrix}},
\end{equation*}
and $
a=\frac{\No-1}{\No}(-1+\bar\alpha-\bar\beta), \quad b=-\frac{a}{\No-1}, \quad c=\frac{\No-1}{\No}(\bar\gamma-\bar\delta), \quad d=-\frac{c}{\No-1}$.
Because $W_c$ and $W_d$ are non-isomorphic irreducible representations of $\Gamma$, it follows by~\cite[Theorem~2.12]{Golubitsky2002book} that $J$ can be block-diagonalized with respect to $W_c$ and $W_d$. Moreover, generically~\cite[Theorem~1.27]{Golubitsky2002book}, $J|_{W_i}=c_i(\lambda)I$, $i=c,d$. To find $c_i(\lambda)$, $i=1,2$, let ${\bf X}_{agree}\in W_c$ and ${\bf X}_{disagree}\in W_d$. Then, computing, we get
\begin{equation*}\label{EQ:agreement bif condition J}
J\,{\bf X}_{agree}=(\No-1)(-1+\bar\alpha-\bar\beta+(\Na-1)(\bar\gamma-\bar\delta)){\bf X}_{agree},
\end{equation*}
which gives $c_1=(\No-1)(-1+\bar\alpha-\bar\beta+(\Na-1)(\bar\gamma-\bar\delta))$, 
and
\begin{equation*}\label{EQ:disagreement bif condition J}
J\,{\bf X}_{disagree}=(\No-1)(\Na-1)(-1+\bar\alpha-\bar\beta-\bar\gamma+\bar\delta){\bf X}_{disagree},
\end{equation*}
which gives $c_2=(\No-1)(\Na-1)(-1+\bar\alpha-\bar\beta-\bar\gamma+\bar\delta)$.
If~\eqref{EQ:stability conditions} is satisfied, then the eigenvalue along the consensus and the eigenvalue along the dissensus space are both negative and thus the neutral state is stable. If $\bar\gamma(\lambda)>\bar\delta(\lambda)$, then $c_1>c_2$ and the monotonicity assumptions imposed on $\bar\alpha(\lambda)-\bar\beta(\lambda)$ and $\bar\gamma(\lambda)-\bar\delta(\lambda)$ ensure that there must exist $\lambda_c^a$ such that $c_1(\lambda_c^a)=0$ with $c_1'(\lambda_c^a)>0$. This implies a non-degenerate equivariant bifurcation along the consensus space. If $\bar\gamma(\lambda)<\bar\delta(\lambda)$, then $c_1<c_2$ and the imposed monotonicity assumptions ensure that there must exists $\lambda_c^d$ such that $c_2(\lambda_c^d)=0$ with $c_2'(\lambda_c^d)>0$. This implies a non-degenerate equivariant bifurcation along the dissensus space.
\end{proof}

We now interpret~\Cref{THM:cooperation competition} from the perspective of agents' {\it cooperativity/competitivity}, as defined below, and then illustrate the prediction in our numerical realization~\eqref{EQ:generic decision dynamics}. The key parameter to determine whether bifurcations happen along the consensus space or along the dissensus space is the difference $\bar\gamma-\bar\delta$. Recall that $\bar\gamma=\left.\frac{\partial \tilde G_{ij}}{\partial x_{kj}}\right|_{\Oo}$ quantifies how much the opinion of agent $i$ for option $j$ is influenced by other agents' opinions for the same option, whereas $\bar\delta=\left.\frac{\partial \tilde G_{ij}}{\partial x_{kl}}\right|_{\Oo}$ quantifies how much the opinion of agent $i$ for option $j$ is influenced by other agents' opinions for the other options. If $\bar\gamma>\bar\delta$, the opinion of agent $i$ for option $j$ will tend to align to other agents' opinions for the same option. The agents behave cooperatively, thus leading to consensus. If $\bar\gamma<\bar\delta$, the opinion of agent $i$ for option $j$ will instead tend to align with the average opinion of the other agents for all the other options. Because options are mutually exclusive~\eqref{e:total_vote}, this implies that if on average the other agents' opinion for option $j$ increases then the opinion of agent $i$ for option $j$ will decrease. The agents behave competitively, thus leading to dissensus.


\Cref{FIG: coop compe} illustrates our prediction in the analytical model~\eqref{EQ:generic decision dynamics all} for $\Na=17$ and $\No=3$. In this model, the value of $\lambda_c^a$ and $\lambda_c^d$, as defined in the statement of Theorem~\ref{THM: consensus or dissensus}, can easily be computed by observing that in model~\eqref{EQ:generic decision dynamics all}
\[
\bar\alpha=\lambda\alpha, \quad \bar\beta=\lambda\beta, \quad \bar \gamma=\lambda\gamma, \quad \bar\delta=\lambda\delta.
\]
Thus, for $\gamma>\delta$ and $\alpha-\beta>0$, consensus bifurcations happen for \[
\lambda=\lambda^a_c=(\alpha-\beta+(\Na-1)(\gamma-\delta))^{-1}. 
\]
For $\gamma<\delta$ and $\alpha-\beta>0$, dissensus bifurcations happen for \[
\lambda=\lambda^d_c=(\alpha-\beta-\gamma+\delta))^{-1}. 
\]
We stress that in all the simulations presented in \Cref{FIG: coop compe} and in subsequent figures small random agents' biases and coupling heterogeneities were added to weakly violate the $\Gamma$-equivariance assumption. Our theoretical predictions remain robustly true.

\begin{remark}
	With similar computations, Theorem~\ref{THM:cooperation competition} and its specialization to model~\eqref{EQ:generic decision dynamics} can be extended to the case in which $\Gamma=\Gamma_a\times\Gamma_o$, with $\Gamma_a\subset\ES_{\Na}$ acting transitively on $\{1,\ldots,\Na\}$. The resulting critical values, at which consensus or dissensus bifurcations happen, will of course depend on the specific symmetry groups.
\end{remark}

\begin{figure}
\centering
\includegraphics[width=0.75\textwidth]{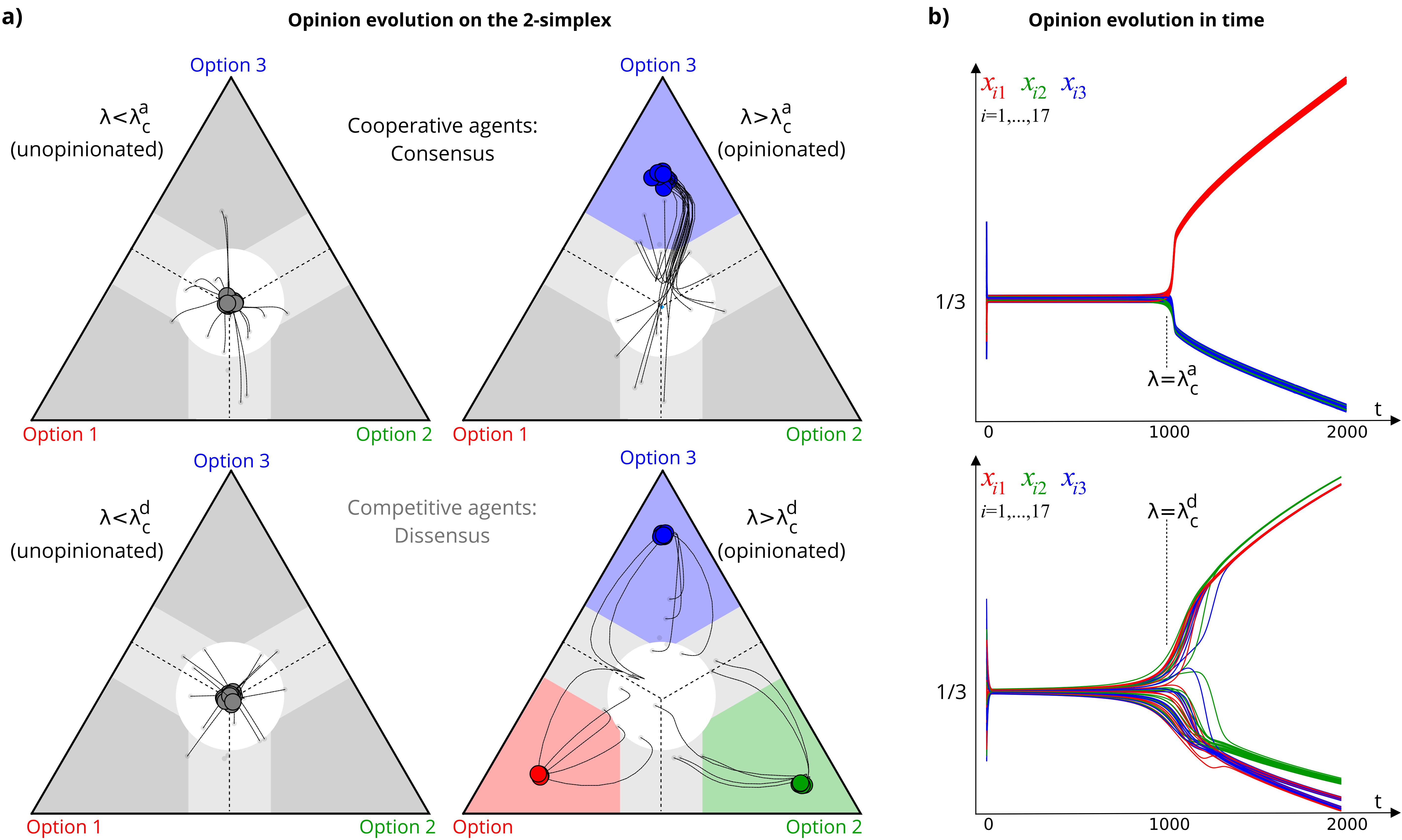}
\caption{{\bf a)} Bifurcation to consensus (top) or to dissensus (bottom) in model~\eqref{EQ:generic decision dynamics all}. When $\gamma>\delta$ (top), a consensus bifurcation, leading to a group consensus state, happens at the critical value $\lambda=\lambda_c^a$. The neutral state is stable for $\lambda<\lambda_c^a$ and unstable for $\lambda>\lambda_c^a$. When $\gamma<\delta$ (bottom), a dissensus bifurcation, leading to a group dissensus state, happens at the critical value $\lambda=\lambda_c^d$. The neutral state is stable for $\lambda<\lambda_c^d$ and unstable for $\lambda>\lambda_c^d$. In each plot, the agents' opinion states are color-coded to reflect the option favored at steady state. Initial conditions are randomly picked and are depicted by the light gray circles. The agents' opinion trajectories are depicted by the thin black lines. {\bf b)} Time evolution of agents' opinion $x_{i1},x_{i2},x_{i3}$, $i=1,\ldots,17$, in model~\eqref{EQ:generic decision dynamics all} for the same parameters as in {\bf a)} but with $\lambda$ smoothly increasing from below to above the bifurcation critical value. Parameters:  for consensus, $\alpha=0$, $\beta=-1.5$, $\gamma=0.2$, $\delta=0.1$, $\lambda=\lambda_c^a\pm 0.05$; for dissensus, $\alpha=0$, $\beta=-0.5$, $\gamma=0.1$, $\delta=0.2$, $\lambda=\lambda_c^d\pm 0.05$; in {\bf b)}, $\lambda(t)=\lambda_c^{a,d}-0.2+0.4t/2000$.}\label{FIG: coop compe}
\end{figure}

\subsection{Moderate and extremist opinions arise from symmetry}
\label{ssec: modarate and extremists}

Another model-independent prediction of our theory is that, generically, dissensus can appear in only two forms: {\it uniform} dissensus and {\it moderate/extremist} dissensus. \Cref{FIG: moderate extremist} illustrates this model-independent fact in the model~\eqref{EQ:generic decision dynamics all}. In a uniform dissensus state (\Cref{FIG: moderate extremist}a), the agents' opinions are uniformly spread across well-defined clusters, each cluster rejecting a different option with the same strength. On the contrary, in a moderate/extremist dissensus state (\Cref{FIG: moderate extremist}b) the agents' opinions are spread across the options in an asymmetric fashion. A small group of agents, the {\it extremists}, develop a strong opinion rejecting one of the options. A larger group of agents, the {\it moderates}, develop a diametrically opposed opinion as compared to the extremists but with a weaker strength. In~\Cref{FIG: moderate extremist}b, the extremists' opinion in rejection for Option~2 is so strong that consensus is not possible although a larger group of agents agree in favoring that option. The numerical explorations we performed on our analytical model suggest that moderate/extremist dissensus happens as a stable or meta-stable state\footnote{We have not been able to find parameter combinations for which moderate/extremist dissensus is stable in the proposed model. This observation is however likely to be model-dependent, in the sense that in other realizations of the equivariant opinion formation dynamics this dissensus state might be stable for suitable parameter choices.} at transition between uniform dissensus and consensus, i.e., for $\gamma<\delta$, $\gamma\approx\delta$, whereas for $\gamma$ sufficiently smaller than $\delta$ only uniform dissensus is observed. To rigorously analyze stability properties of this dissensus state in a model-independent fashion one needs to develop equivariant singularity theory for $\ES_n\times\ES_k$-equivariant bifurcations, $n,k\geq 3$, which is a hard problem, likely to be intractable in general\footnote{Equivariant singularity theory has been rigorously developed only in a handful of cases, i.e., the ones analyzed in~\cite{Golubitsky1985-2}.}.

\begin{figure}
	\centering
	\includegraphics[width=0.95\textwidth]{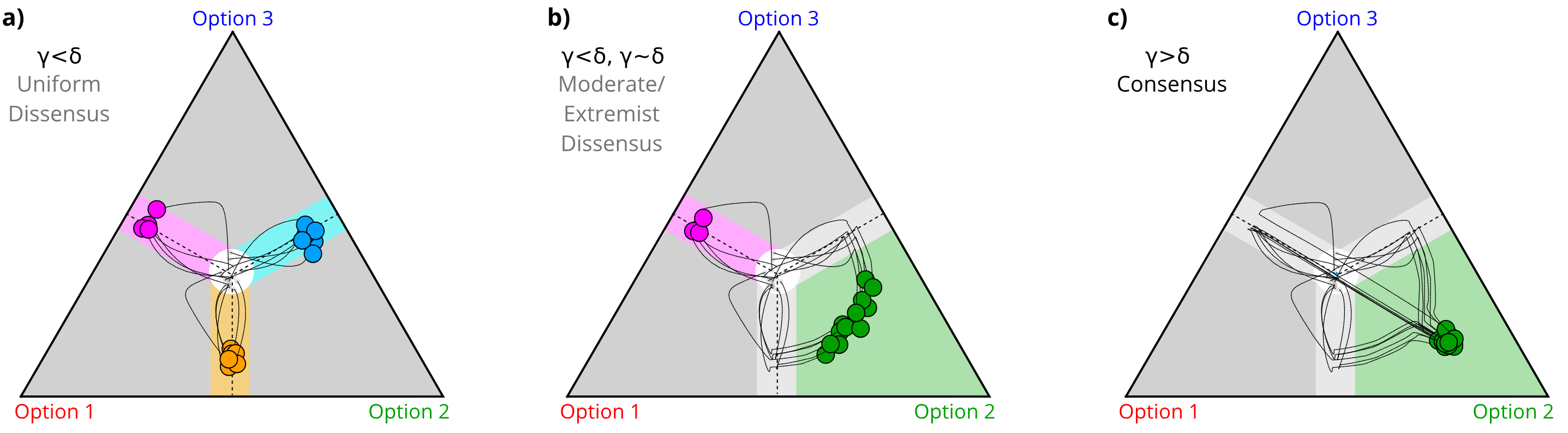}
	\caption{Uniform dissensus (left), moderate/extremist dissensus (center), and consensus (right). In the uniform dissensus state, agents' opinions diverge from the neutral state to form three sharply separated clusters, one for each option, each with roughly the same number of agents and the same opinion strength. In the moderate/extremist dissensus state, agents' opinions diverge from the neutral state in an asymmetric fashion. One small cluster, the extremists (purple), takes a strong stance in disfavor of one of the options. A larger cluster, the moderates (green), maintain a more moderate stance in opposition to the opinion of the extremists. In consensus, all agents favor the same option. Parameters: $\alpha=1.1$, $\beta=-1.0$, $\gamma=0.05$, $\delta(t)=5\gamma/4-\gamma t/(8000)$, $\lambda(t)=0.01+(\alpha-\beta-\gamma+\delta(t))^{-1}$, if $\gamma<\delta(t)$, $\lambda(t)=0.01+(\alpha-\beta+(\Na-1)(\gamma-\delta(t)))^{-1}$, if $\gamma>\delta(t)$. }\label{FIG: moderate extremist}
\end{figure}

Uniform and moderate/extremist dissensus arise in our theory as the {\it axial subgroups} of the action of the opinion-formation dynamics symmetry group inside the dissensus space. Equivariant bifurcation theory predicts that, generically, the symmetry group of the equilibria appearing at a bifurcation from the neutral state must be an axial subgroup. We also conveyed this genericity result in \Cref{FIG:consensus_polar}, where the two sketched branches of dissensus states (d1-d2 and d3-d4) correspond to uniform dissensus and to moderate/extremist dissensus, respectively. In Sections~\ref{SEC: examples} and~\ref{SEC:main results}, we provide a rigorous discussion and proofs of this generic fact.


\begin{remark}
	The results presented in this section do {\it not} generalize to the case in which the opinion dynamics has symmetry $\Gamma=\Gamma_a\times\Gamma_o$, with transitive $\Gamma_a\subset \ES_{\Na}$ and $\Gamma_o\subset \ES_{\No}$. The genericity of uniform and moderate/extremist dissensus depends indeed on the specific dissensus axial subgroup structure and, thus, might no generalize to other symmetry groups.
\end{remark}

\subsection{The ``switchiness" of opinion formation depends on the number of agents and the number of options}
\label{ssec: switch-likeness}

Opinion formation (either toward consensus or toward dissensus) can happen in two qualitatively distinct manners: 1) continuous formation and 2) switch-like formation. In the continuous case, the agents' opinions change continuously as a function of the bifurcation parameter, including at the passage through the bifurcation point at which the neutral state loses stability. In the switch-like case, at the passage through the bifurcation point, the agents' opinions change abruptly as a function of the bifurcation parameter: when the neutral state loses stability, the agents' opinions {\it switch} to a new opinion state almost discontinuously.

\Cref{FIG: switch-like} illustrates these two possible behaviors for both consensus and dissensus bifurcations for two and three options in a group of $\Na=17$ agents. In the simulations, the bifurcation parameter is slowly increased through the critical value at which the neutral state loses stability. In the case of a consensus bifurcation (\Cref{FIG: switch-like} a1 and b1) and a dissensus bifurcation (\Cref{FIG: switch-like} a2 and b2), there is a clear difference between the two-option scenario (a1,a2) and the three-option scenario (b1,b2). In the two-option scenario, opinion formation is continuous. As the bifurcation parameter increases, the agents' opinions change slowly and continuously in time in favor of Option~1 (a1) or to a dissensus state (a2). In the three-option scenario, when the bifurcation parameter crosses the critical value, the agents' opinions exhibit a jump in favor of Option~2 (b1) or toward a dissensus state (b2).

\begin{figure}
	\centering
	\includegraphics[width=\textwidth]{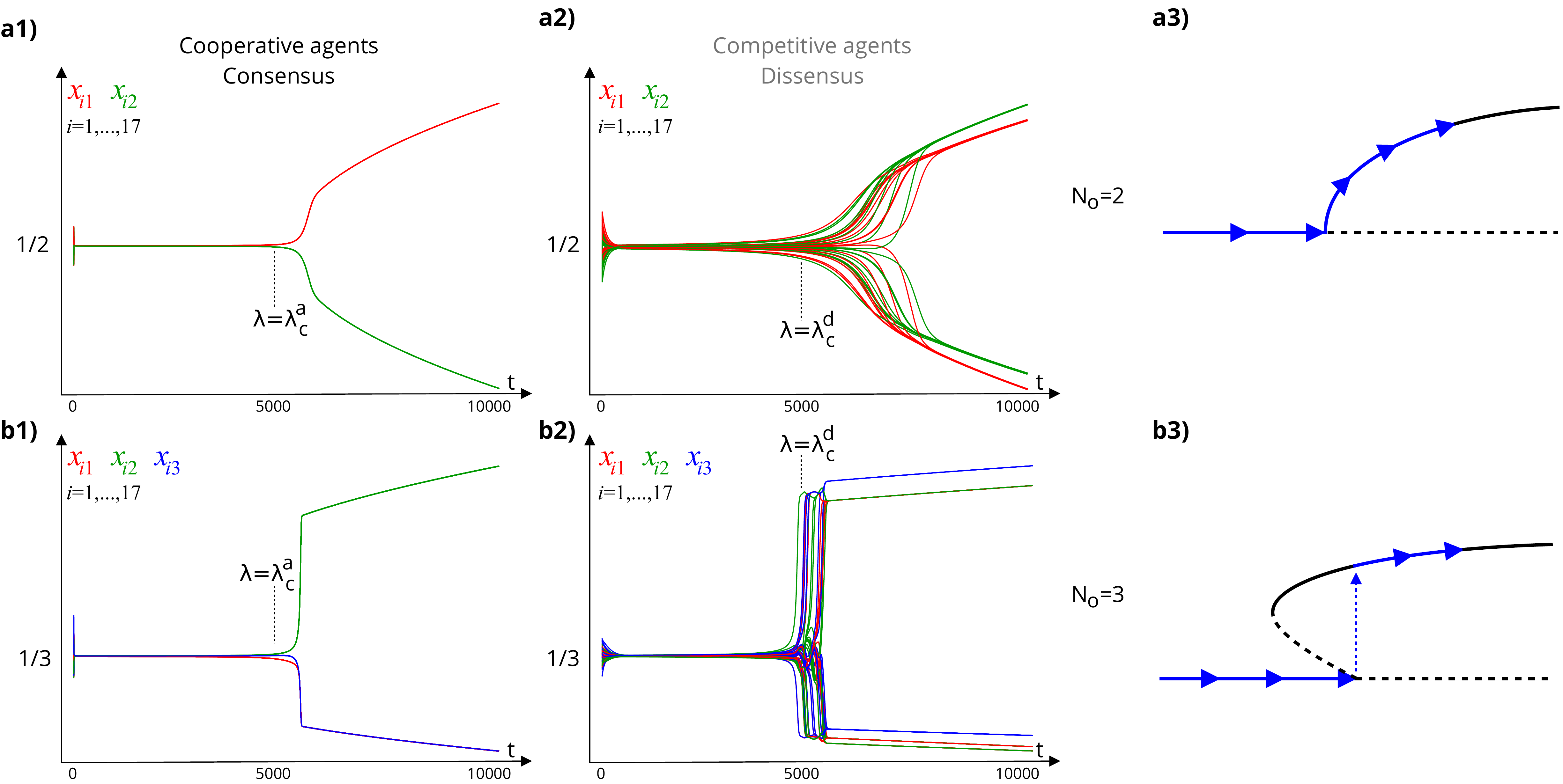}
	\caption{Switch-like ({\bf b1},{\bf b2}) versus continuous ({\bf a1},{\bf a2}) opinion formation. When $\No=2$ ({\bf a1},{\bf a2}), both consensus and dissensus opinion formation are continuous. As the bifurcation parameter is slowly increased through the critical value at which the neutral state becomes unstable, the agents' opinion state continuously follows a bifurcation branch of stable equilibria. When $\No=3$ ({\bf b1},{\bf b2}), both consensus and dissensus opinion formation are switch-like. As the bifurcation parameter is slowly increased through the critical value at which the neutral state becomes unstable, the agents' opinion state jumps away from it and switches to a novel, distinct opinion state because no stable steady state branches appear at the bifurcation. Bifurcation diagram interpretation of continuous ({\bf a3}) and switch-like ({\bf b3}) opinion formation. Solid bifurcation branches are stable and dashed bifurcation branches are unstable. Parameters: $\alpha=0$, $\beta=-1.5$, $\gamma=0.2$, $\delta=0.1$, $\lambda(t)=(\alpha-\beta+(\Na-1)(\gamma-\delta))-0.2+0.4t/10000$, for consensus; $\alpha=0$, $\beta=-0.5$, $\gamma=0.1$, $\delta=0.2$, $\lambda(t)=(\alpha-\beta-\gamma+\delta)-0.2+0.4t/10000$, for dissensus.}\label{FIG: switch-like}
\end{figure}

Of course, as for the solution of any smooth dynamical system, none of the dynamical behaviors described above is actually discontinuous. What engenders either a switch-like or a continuous opinion-formation behavior is how the {\it stable} opinion equilibria change as a function of the bifurcation parameter (\Cref{FIG: switch-like} a3 and b3). In continuous opinion formation, some of the opinion branches bifurcating from the neutral equilibrium are made of stable equilibria. In this case, the agents' opinion state slides along these stable equilibrium branches continuously away from the neutral equilibrium. In switch-like opinion formation, all the opinion branches bifurcating from the neutral equilibrium are made of unstable equilibria. The agents' opinion state is thus obliged to jump to a new stable opinion state far from the neutral equilibrium, which engenders the switch-like opinion-formation behavior. Our theory predicts that consensus opinion formation is always switch-like for $\No\geq 3$. We also predict that dissensus opinion formation is always switch-like when $\No\Na\geq 9$ (See Section~\ref{SSEC:quadratic equivariant} for details). Consensus opinion formation with $\No=2$ is continuous when the organizing bifurcation is a supercritical pitchfork and switch-like when the organizing bifurcation is a subcritical pitchfork~\cite{Gray2018}. Similarly, for $\No\Na\leq 6$, dissensus opinion formation can be either continuous or switch-like depending on whether the organizing pitchfork bifurcations are supercritical or subcritical.

\begin{remark}
	The results presented in this section also do {\it not} generalize to the case in which the opinion dynamics has symmetry $\Gamma=\Gamma_a\times\Gamma_o$, with transitive $\Gamma_a\subset \ES_{\Na}$ and $\Gamma_o\subset \ES_{\No}$. Stability of bifurcation branches is indeed dependent on the properties of the symmetry group action inside the consensus and dissensus space, in particular, the existence of consensus or dissensus quadratic equivariants.
\end{remark}


\section{Some applications of the theory}
\label{ssec: applications}

\subsection{Animal groups group decision making and motion}

Animal groups that are social often make collective decisions when they forage, reproduce, avoid threats, and, eventually, survive. Honeybee colonies have been studied in fine behavioral detail during the critical period in which a decision is made on the location of a new nest~\cite{Seeley2010}. Studies have revealed a rich repertoire of signals through which honeybees express preferences, convince others, and eventually, come to a democratic consensus on one of the alternatives. Some animal groups must also make decisions rapidly while in motion. Examples include groups of strongly schooling fish that must collectively decide in which direction to swim without splitting the group (for instance, in the presence of two spatially separated sources of food or when a predator approaches them)~\cite{Couzin2005a}.

In various works~\cite{Leonard2012,Couzin2011,Pinkoviezky2018,Seeley2012,Reina2017,Gray2018,Pais2013,Nabet2009}, animal group collective decision making has been mathematically described as a bifurcation phenomenon, where the state variables describe the opinion state of the group and the bifurcation parameter capture example-specific properties of the decision-making process. For instance, in the case of honeybee decision making, the bifurcation parameter is mostly related to the quality of the alternatives~\cite{Pais2013}, whereas for animal groups in motion, the bifurcation parameter is mainly related to the geometry of the alternatives (e.g., their distance to the group or their location relative to the group motion direction)~\cite{Couzin2011,Couzin2005a,Leonard2012,Pinkoviezky2018}.

The model-independent theoretical framework and the analytical realization~\eqref{EQ:generic decision dynamics} developed here capture and generalize all the aforementioned modeling efforts. The possibility of making rigorous predictions about decision-making behaviors in animal groups at the single agent level, for an arbitrary number of agents, an arbitrary number of options, and for both consensus and dissensus outcomes, is new. Previous analytical models of animal group decision making either make mean-field approximations (i.e., average out the agent level), focus on consensus only, or restrict to the two alternative case. Overcoming these limitations might open novel interaction paths between experimental and theoretical studies of decision making in animal groups as well as providing the basis to formally translate animal group behavior into artificial multi-agent systems.

\subsection{Cognitive neuroscience}

Low-level perceptual decision making refers to the cognitive phenomenon of distinguishing between two (or more), usually incompatible, perceptual alternatives. Low-level perceptual decision making can be exemplified by the classical random dot experiment~\cite{Williams1984}. This type of decision making happens fast ($<1s$) and is implemented in sensory and motor cortices without requiring either short-term (working) or long-term memory, or any form of ``abstract thinking". 

Behavioral data obtained from the random-dot experiment is usually explained using phenomenological models, like accumulator, counter, or race models of decision-making~\cite{Bogacz2006}. In these models, noisy evidence for perceptual alternatives accumulates by integration until some threshold is reached. Although some experimental evidence of perceptual integration has been found in mean-field (coarse) measures of neuronal activity~\cite{Hanks2017}, phenomenological models fail to capture the neuronal mechanisms underlying perceptual decision making at the cellular and circuit level.

Attractor network models constitute a more biologically-grounded alternative to understand the mechanisms behind low-level perceptual decision making, in particular, when basic properties of biological neuronal networks (such as, recurrent connections and intrinsic and synaptic short-term plasticity) are taken into account. In these models, decision making happens as bifurcations between basal/indecisive states and excited decision states (see~\cite{Deco2013} for a review of attractor network models of perceptual decision making). Interestingly, recent work on cognitive control~\cite{Musslick2019} and binocular rivalry~\cite{Diekman2012} suggests that similar models also describe higher-level decision making (for instance, allocation of cognitive efforts in a multi-task setting or interpretation of ambiguous binocular visual stimuli). This observation suggests a seemingly overlooked connection between low- and high-level decision making in the brain.

The model-independent theory we are proposing, together with the analytical realization~\eqref{EQ:generic decision dynamics}, provide the means to start the rigorous exploration of multi-alternative (more than two) perceptual decision making. There are indeed considerable qualitative differences between two and three option decision making in terms of the continuous or swtich-like nature of opinion formation. Because of this fundamental difference, we believe that moving from two to three options in perceptual decision making could provide a novel source of knowledge in understanding brain dynamics.

\subsection{Socio-political opinion networks}

%
%
%


A growing body of literature is studying opinion formation in socio-political networks from a mathematical modeling perspective~\cite{Castellano2009,Evans2018,Li2013,Stewart2019,Vasconcelos2019,Fortunato2005,Nedic2012,Parsegov2017,Altafini2013,Blondel2009,DeGroot1974,Friedkin1999,Hegselmann2002,Ye2019,Jia2015,Jia2019,Lorenz2007,Pan2018,Galam1995,Galam2007,Cisneros2019,Deffuant2000}. One of the reasons for this growing interest in socio-political opinion networks is the increasing occurrence, during the last decades, of polarized political debate, both at the elite and voter level~\cite{McCarty2019}. For the sake of this paper, what is meant by {\it polarization} is probably best illustrated with an example. During decades of political debate about, for example, abortion, the electorate opinion in the United States about the subject (strongly disfavor, disfavor, weakly disfavor, neutral, weakly favor, favor, strongly favor) changed from most of the Democrats (Republicans) having moderate opinions in favor (disfavor) of abortion to most of the Democrats (Republicans) having extreme opinions in favor (disfavor) of abortion~\cite{Mouw2001,Jelen2003}.

Various modeling approaches bring different mechanistic accounts for the possible origin of polarized political discourse. In bounded confidence models~\cite{Deffuant2000,Hegselmann2002,Hegselmann2002,Nedic2012,Lorenz2007}, for instance, the appearance of polarization is mainly determined by the model initial conditions, representing the agents' priori beliefs. In linear consensus-like models~\cite{Altafini2013,Pan2018,Cisneros2019}, a specific form of polarization, called bipartite consensus, emerge from specific signed network topologies, called structurally balanced networks.

Our model-independent approach reconciles these modeling efforts. For instance, our model becomes sensitive to initial conditions if the bifurcation parameter is already beyond the singularity. In this regime, the model is multistable, with each attractor corresponding to a different (consensus or dissensus) opinion configuration. In the dissensus regime, moderate prior beliefs (i.e., initial conditions close to neutral) lead to polarized (i.e., much further away from neutral) final opinions in one of the dissensus states predicted by the equivariant theory. At the same time, as explained in Remarks~\ref{RMK: coopetitive 1} and~\ref{RMK: coopetitive 2}, our equivariant approach naturally accommodates structured interaction topologies (for instance, structurally balanced ones) by suitably changing the model symmetry group. Our modeling approach is thus able to unify bounded-confidence, linear/structurally balanced, and maybe other, opinion network models in a single theoretical framework. This might allow for a more rigorous comparison between different hypothesis and theories about the origin of polarized political debates, as well as other questions in socio-political opinion network dynamics.

\subsection{Phenotypic decision making}

The universal biological phenomenon of phenotypic differentiation provides a wide-spread example of dissensus decision making and opinion formation. A change in phenotype at the single cell level can be seen as a cellular decision-making process realized through the creation of, disappearance of, and transition between, different attractors of the cell molecular regulatory network~\cite{Balazsi2011,Xiong2003,Ferrell1998}. In social micro-organisms, group fitness is increased when cells undergo phenotypic differentiation and make cellular decisions in a coordinated fashion. Typical examples are bacterial groups reacting with phenotypic differentiation in response to environmental stressors, like the presence of antibiotics~\cite{Bottagisio2019} or the scarcity of nutrients~\cite{Rivera2019}, via quorum sensing~\cite{Miller2001,Schrom2020}.

The collective decision-making process associated with phenotypic differentiation is of the dissensus type because different cells make different decisions, i.e., they transition to different phenotypes. The theory developed in the present paper provides a novel theoretical framework to support experimental research in phenotypic differentiation and, more generally, collective decision making in social microorganisms.  Over long (multiple generation) timescales and when averaging out the cellular level, phenotypic differentiation can lead to speciation via natural selection. Equivariant bifurcation theory was successfully applied in modeling speciation~\cite{Stewart2003a}. It may be possible to generalize this idea to the temporal and spatial scales of single cells reacting in real time to the environment.

%

\subsection{Distributed dynamic task-allocation}

In collective robotic systems, each agent represents a robot and the $\No$ options represent a set of $\No$ possible tasks to be accomplished by the robot swarm.
Each robot must dynamically (i.e., in real time) decide which task to accomplish in a distributed fashion, that is, solely relying on the information collected through its sensors, the information received by other robots in the swarm, and without a centralized controller \cite{Bayindir2016,Krieger2000,Brambilla2013}.
Agents' opinions represent each robot's ``motivation" to develop a given task.  To achieve an efficient distribution of labor, different robots could be motivated to accomplish different tasks. In terms of our opinion formation theory, dissensus could thus be a natural state for a distributed dynamic task-allocation problem.

The idea of using nonlinear opinion formation dynamics  to realize dynamic task allocation in robotic systems was originally proposed in~\cite{Reverdy2018} in the case of a single robot deciding over two possible tasks. The key ideas in~\cite{Reverdy2018} are the following. The robot actions happen in the physical space, which evolves dynamically in continuous time. Representing the robot's motivational and decision state in continuously evolving dynamical variables creates novel design possibilities with respect to more classical logic- and optimization-based approaches. Namely, it creates the possibility of suitably coupling the physical world dynamics with the robot's motivational dynamics in real time without requiring any hybrid dynamic/logic or continuous/discrete interface. Our theoretical framework and vector field realization provide  the means to extend these ideas to multi-robot, multi-task problems.

%
%

\section{An introduction to equivariant bifurcation theory and its application to consensus and dissensus}
\label{SEC: preliminaries}

This section is divided into three parts. First, we introduce the background needed to formalize the symmetry-based analysis of opinion formation.  Then we discuss the equivariant branching lemma, the basic result for finding symmetry-breaking solutions in equivariant bifurcation problems.  Finally, we use symmetry-breaking to set up the analysis of consensus and dissensus solutions.

\subsection{Fixed-point subspaces and isotropy subgroups}

Symmetries of a system of differential equations are invertible linear maps that take solutions to solutions. More specifically, let 
\begin{equation} \label{e:G}
\dot{x} = G(x)
\end{equation}
where $x\in\R^m$ and $G:\R^m\to\R^m$.  Then the invertible linear map $\gamma:\R^m\to\R^m$ takes solutions to solutions if and only if $G$ is $\gamma$-equivariant; that is,
\[
G(\gamma x) = \gamma G(x) \quad\forall x\in\R^m.
\]
Let $\Gamma$ be the group of symmetries of $G$; we say that $G$ is $\Gamma$-equivariant.

The key idea in finding equilibria of equivariant systems \eqref{e:G} is the symmetry forced existence of flow-invariant subspaces.  

\begin{definition}
	Let $\Gamma$ be a group acting on $\R^m$ and let $\Sigma\in \Gamma$ be a subgroup. The 
	{\it fixed-point subspace} of $\Sigma$ is
	\[
	\Fix(\Sigma)=\{x\in\R^m \ :\ \gamma x=x\quad \forall\gamma\in\Gamma \}
	\]
\end{definition}
Observe that $\Fix(\Sigma)$ is a vector subspace of $\R^m$.
We claim that $G:\Fix(\Sigma)\to\Fix(\Sigma)$. The proof is short.  Suppose that $x\in\Fix(\Sigma)$.  Then $G(x) = G(\sigma x) = \sigma G(x)$ for all $\sigma\in\Sigma$. It follows that $G(x)\in \Fix(\Sigma)$. 

It follows that we can find solutions to $G=0$ by searching for solutions to $G|\Fix(\Sigma) = 0$ for subgroups $\Sigma$ with a certain special symmetry property; and this is the approach that we take.  

A point $x\in\R^n$ has symmetry group $\Sigma_x\subset\Gamma$ where $\sigma\in\Sigma_x$ if and only if $\sigma x = x$.  Technically, $\Sigma_x$ is called the {\em isotropy subgroup} of $x$. As we next explain, the important idea from equivariant bifurcation theory is that certain isotropy subgroups (called axial subgroups) generically lead to families of equilibria with symmetry $\Sigma_x$.  Thus, the analytic problem of finding equilibria to systems of differential equations is reduced to the algebraic problem of finding axial subgroups.

\subsection{Equivariant bifurcation theory}

Bifurcation theory assumes that the vector field $G$ in \eqref{e:G} depends explicitly on a parameter $\lambda$.   That is, we assume
\begin{equation} \label{s:Glambda}
\dot{x} = G(x,\lambda)
\end{equation}
where $x\in\R^m$, $\lambda\in\R$, and $G:\R^m\times\R\to\R^m$.  In equivariant bifurcation theory, we assume $G$ is $\Gamma$-equivariant; that is, $G(\gamma x, \lambda) = \gamma G(x,\lambda)$ for all $\gamma\in\Gamma$. 

Moreover, suppose $\Fix(\Gamma) = \{0\}$ (which happens when $\Gamma$ acts irreducibly on $\R^m$).  Then the subspace $\{0\}$ is flow-invariant and $G(0,\lambda) = 0$ for all $\lambda$.  We say that $x=0$ is a {\em trivial equilibrium} and we will look for branches of solutions that bifurcate from the trivial equilibrium.  

Bifurcation from the trivial solution occurs at $\lambda_0$ if the Jacobian $J(\lambda_0)  =  dG(0,\lambda_0)$ is singular.  Let $V$ be the kernel of $J(\lambda_0)$.  The subspace $V$ is $\Gamma$-invariant and generically, $V$ is an absolutely irreducible representation of $\Gamma$.   That is:

\begin{definition}\label{DEF: abs irre} \rm
	A group $\Gamma$ acts {\em absolutely irreducibly} on the subspace $V$ if the only linear maps on $V$ that commute with $\Gamma$ are multiples of the identity.
\end{definition}

Without loss of generality we can assume that $\lambda_0 = 0$.  The chain rule implies that $J(\lambda)$ commutes with $\Gamma$ and absolute irreducibility implies that $J(\lambda) = c(\lambda)I_m$.  It follows that steady-state bifurcation occurs at $\lambda = 0$ if and only if c(0) = 0.

The equivariant branching lemma provides the basic tool for finding branches of bifurcating solutions in 
$\Gamma$-equivariant bifurcation problems.  Before stating this lemma we present the central definition in the theory.

\begin{definition} \rm
	Let $x\in\R^m$ and $\Gamma$ act on $\R^m$. The isotropy subgroup $\Sigma_x$
	is {\em axial} if $\dim\,\Fix(\Sigma_x) = 1$.
\end{definition}

\begin{lemma}[\bf Equivariant Branching Lemma]\label{LEM: equiv br lemma}
	Assume $\Gamma$ acts absolutely irreducibly on $\R^m$ and $G:\R^m\times\R\to\R^m$ is $\Gamma$-equivariant.  Then
	\[
	\begin{array}{rcl}
	G(0,\lambda)&\equiv&0\\
	(dG)_{(0,\lambda)}&=&c(\lambda) I_m
	\end{array}
	\]
	Assume the bifurcation condition  $c(0) = 0$ and the non-degeneracy condition $c'0)\neq 0$.  
	Then, for every axial subgroup $\Sigma\subset\Gamma$ there exists a unique branch of solutions to $G(x, \lambda) = 0$ emanating from $(0,0)$, where the symmetry group of the bifurcating solutions is $\Sigma$.
\end{lemma}

For each absolutely irreducible representation of $\Gamma$, the Equivariant Branching Lemma determines branches of solutions that correspond to axial subgroups of $\Gamma$.  There may be solutions with non-axial symmetry that bifurcate from the origin for an open set of bifurcation problems~\cite{Field1989}, but the axial solutions are the ones that occur on an open dense set of bifurcation problems.  Moreover, the axial subgroups and their fixed-point subspaces provide the easiest way to determine equilibria.   

The following results~\cite[Lemma~1.31]{Golubitsky2002book} help explain why the equivariant branching lemma is so useful.

\begin{theorem}\label{thm:generic facts}
	Let $\Gamma$ be a compact Lie group. Suppose the equivariant bifurcation problem $G$ has a steady-state bifurcation from $x=0$ at $\lambda_0 = 0$. Then $A_0 = (dG)_{0,0}$ has a $0$ eigenvalue and the following results hold generically.
	\begin{enumerate}
		\item[(a)] $0$ is the only eigenvalue of $A_0$ on the imaginary axis.
		\item[(b)] The generalized eigenspace corresponding to $0$ is $\ker\,A_0$.
		\item[(c)] $\Gamma$ acts absolutely irreducibly on $\ker\,A_0$.
	\end{enumerate}
\end{theorem}

The Lyapunov-Schmidt reduction \cite[Section~1.3]{Golubitsky2002book} ensures that we can reduce the $\Gamma$-equivariant bifurcation problem to $\ker\;A_0$ preserving $\Gamma$ equivariance and Theorem~\ref{thm:generic facts} ensures that, in doing so, the action of $\Gamma$ on $\ker\; A_0$ is absolutely irreducible.


\subsection{Consensus and dissensus bifurcations}
\label{SSEC: conse disse bifu}

Given this background we will prove that consensus and dissensus are the only generic symmetry-breaking opinion formation behaviors.

Recall that agent permutations $\ES_\Na$ acts on $\VV$ by permuting the agent axes, that is, if $\sigma\in\ES_\Na$,
\[
\sigma\Xx=(\Xx_{\sigma^{-1}(1)},\ldots,\Xx_{\sigma^{-1}(\Na)})
\]
Options permutations $\ES_\No$ acts diagonally on $\VV$, that is, if $\tau\in\ES_\No$,
\[
\tau\Xx=(\tau\Xx_1,\ldots,\tau\Xx_\Na),\quad \tau\Xx_i=(x_{i\tau^{-1}(1)},\ldots,x_{i\tau^{-1}(\No)}).
\]
Representing this action in the shifted (linearized) variables $\Zz$,  defined by~\eqref{EQ:linearized vars}, we obtain a representation of $\Gamma = \ES_\Na\times\ES_\No$ on the linear space $V=T_{\Oo}\VV\cong\R^{\No-1}\otimes\R^{\Na}$. The equivariant analysis will be developed in this linear representation.

The given action of $\ES_\No$ on each $V_i=T_{\NP_i}\Delta$ is isomorphic to the standard action of $\ES_\No$ on $\R^{\No-1}$. In particular, $\Fix_{V_i}(\ES_\No)=\{0\}$.

In the original coordinates on the simplex, this means that the fixed point set of the option symmetries in the option space $\Delta$ of each agent $i$ is exactly the neutral point \eqref{e:NP}. Because $\ES_\No$ acts diagonally, the fixed-point subset of $\ONE\times S_\No\subset\Gamma$ in $V$ is $\Fix_V(\ONE\times S_\No)= \{0\}$. It follows that
\begin{equation} \label{e:Gamma fix}
\Fix_V(\Gamma)= \{0\}.
\end{equation}

Our model has symmetry $\Gamma$ in the following sense.
The opinion formation dynamics~\eqref{e:symm bifu dyna} maps in an affine way to a opinion formation dynamics on $V$ defined by 
\begin{equation}\label{EQ:tangent dynamics}
\dot\Zz=G(\Zz,\lambda),
\end{equation}
where $G:V\times\R\to V$\footnote{Note that, with a small abuse of notation, we are denoting the vector field on the tangent space $V$ with the same letter as the vector filed on the base space $\VV$.}. Our symmetry assumption is formalized by requiring that $G$ is $\Gamma$-equivariant.  We study the associated equivariant bifurcation problem
\begin{equation}\label{EQ:symm bifu prob}
G(\Zz,\lambda)=0.
\end{equation}
Invoking~\eqref{e:Gamma fix} and \cite [Theorem~1.17]{Golubitsky2002book}, it follows that the origin is a solution of the bifurcation problem~\eqref{EQ:symm bifu prob} for all~$\lambda$, that is,
\[
G(0,\lambda) \equiv 0.
\]
We call the origin the {\it trivial equilibrium}. Note that in the original opinion formation coordinates $\Xx$ the trivial equilibrium is exactly the neutral point $\Oo$. Our goal is to study symmetry-breaking from the trivial equilibrium. In doing so, we mainly rely on the Equivariant Branching Lemma~\cite[Lemma~1.31]{Golubitsky2002book},\cite[Theorem~3.3]{Golubitsky1985-2}.

The first step in applying the Equivariant Branching Lemma is to decompose the state space into the direct sum of {\it irreducible representations}. A group representation on a given vector space is irreducible if the only invariant subspace with respect to the group action is the origin (see~\cite[Section~XII.2]{Golubitsky1985-2},\cite[Definition~1.21]{Golubitsky2002book}). In our case, we can write
\[
V = W_c\oplus W_d
\]
where $W_c$ and $W_d$ are the consensus and dissensus space defined in~\eqref{e:agreement space} and~\eqref{e:disagreement space}, respectively. It is easy to verify that both $W_c$ and $W_d$ are irreducible representations of $\Gamma$. Moreover $\ES_\Na$ acts trivially on $W_c$ but faithfully on $W_d$, whereas $\ES_\No$ acts faithfully on both spaces. It follows that the actions of $\Gamma$ on $W_c$ and $W_d$ are not isomorphic. Because $W_c\oplus W_d=V$, \cite[Corollary XII.2.6(a)]{Golubitsky1985-2} implies that no other $\Gamma$-irreducible representations exist.  Absolute irreducibility implies that there is a trivial equilibrium for every equivariant bifurcation problem.

Going back to symmetry breaking in decision making, we can conclude that, generically, there are two types of symmetry breaking bifurcations of~\eqref{e:symm bifu dyna}.
\begin{theorem}\label{THM: consensus or dissensus}
	Suppose that~\eqref{e:symm bifu dyna} is $\Gamma$-equivariant. Then generically symmetry breaking either happens along the consensus space (which corresponds to $\ker A_0=W_c$) or along the dissensus space (which corresponds to $\ker A_0=W_d$).
\end{theorem}
Note that, generically, the two types do not occur together, that is, either one type of bifurcation is observed or the other. The analysis in Section~\ref{SSEC: balance determines} determines which type occurs. To determine the structure of the bifurcating branches appearing inside either mode, we can apply the Equivariant Branching Lemma.

\begin{remark}\label{RMK: conse disse general}
	The result of Theorem~\ref{THM: consensus or dissensus} generalizes to the case in which  $\Gamma=\Gamma_a\times\Gamma_o$, with $\Gamma_a\subset\ES_{\Na}$ and $\Gamma_o\subset\ES_{\No}$ such that $\Gamma_a$ acts transitively on $\{1,\ldots,\Na\}$.  Transitivity of the agent factor of the symmetry group $\Gamma_a$ means that all agents can be sent into all other agents by the symmetry group action. In other words, transitivity generalizes the condition that all agents are equivalent by asking that any agent can be sent into any other agent by the symmetry group action, although it might not be possible to do so by simply swapping arbitrary pairs of agents as in the fully symmetric case. This means that there are no apriori defined clusters of distinguishable agents. As a relevant example, consider the case in which $\Gamma_a=\D_\Na$, the dihedral group of order $\Na$. Geometrically, this means that the agent and their interactions are organized as a regular polygon with $\Na$ vertices. $\D_\Na$ is transitive and indeed it is possible to send any agent into any other agent, for instance, by cycling them forward. Agents are equal. However, if $\Na>3$, $\D_\Na$ does not contain all simple permutations $(i\ j)$, $i,j\in\{1,\ldots,\Na\}$, $i\neq j$, so not every pair of agents can be swapped while leaving the opinion formation dynamics unchanged. An example of a non-transitive $\Gamma_a$ is $\Gamma_a=\ES_{n_1}\times\ES_{n_2}$, $n_1+n_2=\Na$, corresponding to two clusters of agents, with the $n_1$ agents in the first cluster distinguished from the $n_2$ agents in the second cluster. Agents in the first clusters cannot be sent to agents in the second cluster by the symmetry group action, reflecting the distinguished nature of the two groups. A concrete situation in which such symmetry might arise is when the first cluster of agents is influenced by the opinion of the agents in the second cluster, but not vice-versa, a lack of symmetry in the interaction topology which makes the two groups distinguished.

	We now prove the statement of this remark. By~\cite[Proposition~2.4]{Dionne1996}, absolutely irreducible representations of the action of $\Gamma$ on $V$ are of the form $W_i^{\rm a}\otimes W_j^{\rm o}$, with $W_i^{\rm a}\subset \R^{\Na}$, $W_1^{\rm a}\oplus\cdots\oplus W_s^{\rm a}=\R^\Na$, absolutely irreducible representations of $\Gamma_a$ on $\R^\Na$, and $W_i^{\rm o}\subset V_{\No}$, $W_1^{\rm a}\oplus\cdots\oplus W_r^{\rm a}=V_\No$, absolutely irreducible representations of $\Gamma_o$ on $V_\No$. The action of $\Gamma_a$ on $\R^\Na$ always admits a trivial irreducible representation on $W^{\rm a}_1=\R\{(1,\ldots,1)\}$. Because $\Gamma_a$ is transitive, no other trivial representations of $\Gamma_a$ on $\R^{\Na}$ exist and, in particular, $W^{\rm a}_j$ is not $\Gamma_a$-isomorphic to $W^{\rm a}_1$ for any $j>1$. Furthermore, $V_\Na$ is $\Gamma_a$-invariant, so all other isotypic components (and thus all other irreducible representations) of $\Gamma_a$ are contained in $V_\Na$. It follows that when $\Gamma_a$ is transitive on $\{1,\ldots,\Na\}$ the absolutely irreducible representations of $\Gamma$ appear in two types: {\it consensus irreducible representations}, $W^{\rm o}_1,\ldots,W^{\rm o}_r\subset W_c$, and {\it dissensus irreducible representations} $W^{\rm a}_i\otimes W^{\rm o}_1,\ldots, W^{\rm a}_i\otimes W^{\rm o}_\No\subset W_d$, $i=2,\ldots,s$. The generalization of Theorem~\ref{THM: consensus or dissensus} follows easily from isotypic components techniques~\cite[Section~2.2]{Golubitsky2002book},\cite[Section~XII.2]{Golubitsky1985-2}.
\end{remark}

In the following, we only develop the equivariant analysis for the fully symmetric case $\Gamma=\ES_{\Na}\times\ES_{\No}$. Axial subgroups and, thus, the structure of consensus and dissensus bifurcation branches for lower-symmetry cases can be similarly constructed.

\section{Opinion formation for $\Na=2$ or $\No=2$}
\label{SEC: examples}

We now apply the Equivariant Branching Lemma to various opinion-forming problems with either $\Na=2$ or $\No=2$. All of the relevant equivariant bifurcation results used in this section were previously derived~\cite{Aronson1991},\cite[Section~XIII.5]{Golubitsky1985-2},\cite[Chapter X]{Golubitsky1985-2}. We interpret them in terms of opinion formation. In the remainder of the paper, we set the opinion threshold $\vartheta=0$. With this choice, the only unopinionated state of an agent $i$ is exactly its neutral point $\Oo_i$ and an agent is conflicted between two or more options when her opinion about them is exactly the same. This choice allows a precise mapping between the theoretical results and their interpretation in terms of opinion formation. A positive $\vartheta$ can always be plugged back to make the interpretation robust to heterogeneities and other small perturbations to the exact symmetry assumption, as it was done for the numerical illustrations in Figures~\ref{FIG: coop compe},\ref{FIG: moderate extremist},\ref{FIG: switch-like}. Tables~\ref{TAB: consensus data} and~\ref{TAB: dissensus data} summarize the algebraic data underlying the analysis of consensus and dissensus opinion formation developed in this section. The case $\Na=\No=2$ is analyzed in in~\cite{Bizyaeva2020} in terms of two clusters of agents deciding over two options. It won't be considered here.

\subsection{Three agents and two options}
\label{SSEC: three two}

When $\Na=3$ and $\No=2$ the state space is isomorphic to $\R^3$. Each agent state-space is isomorphic to the real line with positive and negative values associated to preferring either of the options. The consensus space is thus one-dimensional and inside this space the symmetry group action reduces to the standard action of $\Z_2$ in $\R$, as defined in Table~\ref{TAB: consensus data}. Consensus opinion formation is therefore organized by the generic $\Z_2$-equivariant bifurcation, the pitchfork (Figure~\ref{FIG: 3 agents 2 options}a1). The two bifurcating branches correspond to the three agent agreeing on one of the two options (Figure~\ref{FIG: 3 agents 2 options}a2).

\begin{figure}[t!]
	\centering
	\includegraphics[width=0.75\textwidth]{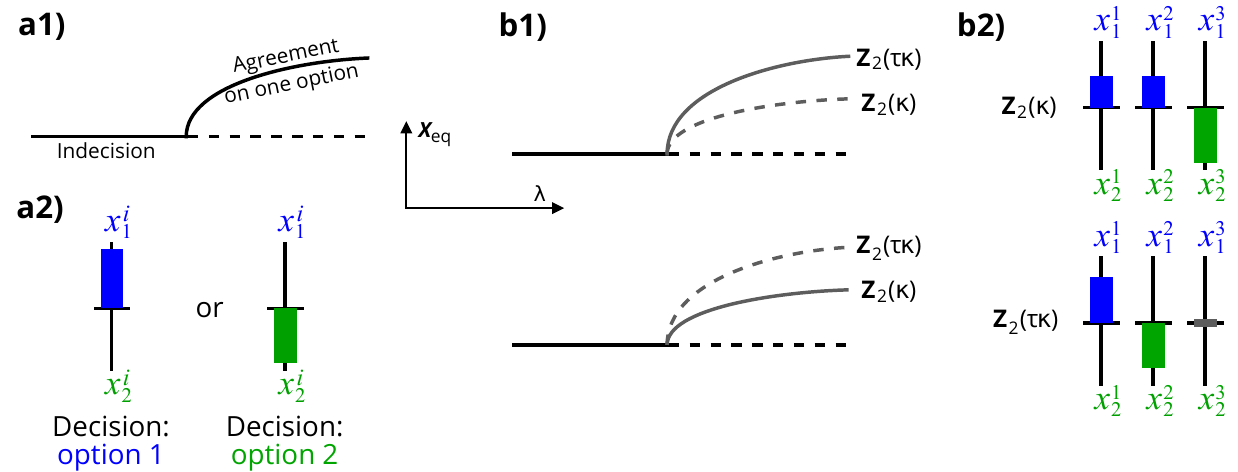}
	\caption{Consensus (a) and dissensus (b) opinion formation between $3$ agents and $2$ options. {\bf a1,b1)} Equivariant bifurcation diagrams. Dashed branches are unstable, full branches are stable. The symmetry group of each branch (i.e., its associated axial subgroup) is indicated, unless trivial. {\bf a2,b2)} Opinion formation interpretation of the predicted non-trivial bifurcation branches.}\label{FIG: 3 agents 2 options}
\end{figure}

The dissensus space is isomorphic to the subspace $V_3\subset\R^3$ where the coordinate sum is zero, as defined in~\eqref{EQ: zero sum subspace}, and it is isomorphic to $\R^2$. The action of the symmetry group in this subspace is the standard action of $\ES_3\times \Z_2\cong \D_6$ on $\R^2$, as defined in Table~\ref{TAB: dissensus data}.
In this representation, there are two conjugacy classes of axial subgroups~\cite[Section~XIII.5]{Golubitsky1985-2} and associated conjugacy class of bifurcation branches (Figure~\ref{FIG: 3 agents 2 options}b1): $\Z_2(\kappa)$ and $\Z_2(\tau\kappa)$. The associated conjugacy classes of one-dimensional fixed point subspaces are
\begin{subequations}\label{EQ: 3a 2o dis fixed}
	\begin{align}
	\Fix(\Z_2(\kappa))&=\{\Zz_1=\Zz_2=(c,-c),\ \Zz_3=-2\Zz_1,\ c\in\R\}\\
	\Fix(\Z_2(\tau\kappa))&=\{\Zz_3=0,\ \Zz_1=-\Zz_2=(c,-c),\ c\in\R \}\,.
	\end{align}
\end{subequations}
The  class $\Z_2(\kappa)$ describes the situation in which two agents agree on one of the two options and the third agent strongly disagree by preferring the other option (Figure~\ref{FIG: 3 agents 2 options}b2, top). It is a moderate/extremist situation. The class $\Z_2(\tau\kappa)$ describes the situation in which one agent is neutral, while the other two agents disagree by preferring different options (Figure~\ref{FIG: 3 agents 2 options}b2, bottom). It is a homogeneous dissensus situation. Depending on second-order terms, the bifurcating branch of either type of conjugacy classes can be stable, but not both at the same time, as illustrated in Figure~\ref{FIG: 3 agents 2 options}b1.

\subsection{Two agents and three options}
\label{SSEC: two three}

When $\Na=2$ and $\No=3$ the state space is isomorphic to $\R^4$. Each agent state space is the two-simplex in $\R^3$, which is locally isomorphic to its tangent space $V_3\cong\R^2$. The consensus subspace is thus isomorphic to $\R^2$ and inside this space the symmetry group action reduces to the standard action of $\ES_3\cong \D_3$ on $\R^2$, generated in our representation by option permutations as defined in Table~\ref{TAB: consensus data}. Figure~\ref{FIG: 2 agents 3 options}a1 reproduces a generic $\D_3$-equivariant diagram from \cite[Figure~XV.4.2]{Golubitsky1985-2} and its two interpretations in terms of consensus opinion formation. A similar interpretation can be worked out for the other generic $\D_3$-equivariant diagram in \cite[Figure~XV.4.2]{Golubitsky1985-2}. There is one conjugacy class of axial subgroups, that is, $\Z_2(\kappa)$. The associated conjugacy class of fixed point subspace is
\begin{equation}\label{EQ: 2a 3o cons fixed}
\Fix(\Z_2(\kappa))=\{\ z_{11}=z_{21}=z_{12}=z_{22}=c,\ z_{13}=z_{23}=-2c\ c\in\R\}.
\end{equation}
The bifurcating $\Z_2(\kappa)$ branch is transcritical, i.e., it is made of one subcritical half-branch before the bifurcation point and one supercritical half-branch after the symmetric bifurcation point. Both branches are unstable at the bifurcation point. Thus, as discussed in Section~\ref{ssec: switch-likeness}, consensus opinion formation over three options is switch-like. Equivariant singularity theory~\cite[Section~XV.4]{Golubitsky1985-2} predicts that the subcritical branch bends in a fold bifurcation away from the symmetric bifurcation point at which it becomes stable. We expect the opinion state to jump toward this stable branch at crossing the $\ES_3$-equivariant singularity. The two bifurcation branches describe two different opinion configurations: favoring one option (Figure~\ref{FIG: 2 agents 3 options}a2, right) or to being conflicted between two options (Figure~\ref{FIG: 2 agents 3 options}a2, left).
The stable opinion configuration to which the opinion state jumps at crossing the $\ES_3$-equivariant singularity can correspond to both of them, which leads to two possible opinion formation interpretations of this axial.


\begin{remark}\label{RMK: S3 secondary}
$\ES_3$-equivariant singularity theory further predicts the existence of {\it secondary bifurcations} at which the stable opinion configuration loses stability. {\it In terms of opinion formation, these secondary bifurcations are the mechanisms through which an opinionated conflicted state can transition to a state where only one option is favored.} In the three-option consensus case considered here, the stable conflicted state in Figure~\ref{FIG: 2 agents 3 options}a2, left, can loose stability in a secondary pitchfork bifurcation that can either be supercritical or subcritical. In the former (latter) case, the opinion state continuously changes (switches) to favoring one of the two options between which the group was conflicted.
\end{remark}


\begin{figure}
	\centering
	\includegraphics[width=\textwidth]{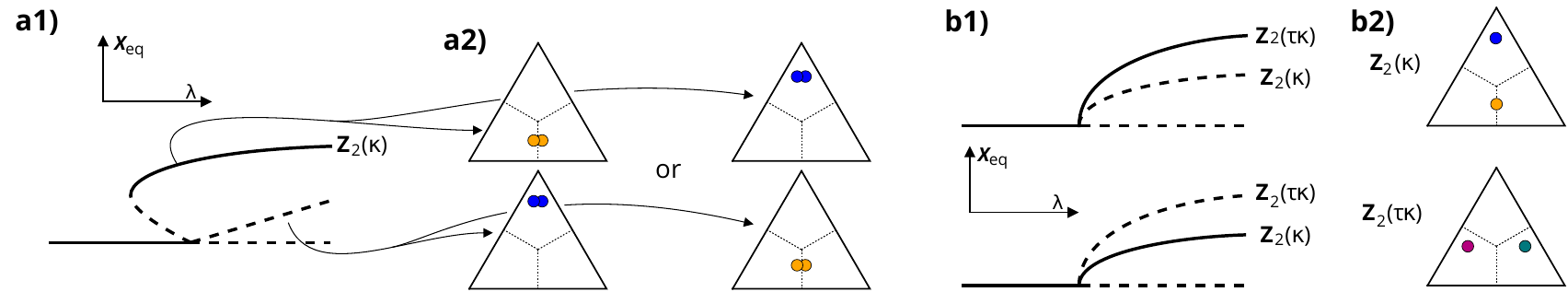}
	\caption{Consensus (a) and dissensus (b) opinion formation between $2$ agents and $3$ options. {\bf a1,b1)} Equivariant bifurcation diagrams. Dashed branches are unstable, full branches are stable. The symmetry group of each branch (i.e., its associated axial subgroup) is indicated, unless trivial. {\bf a2,b2)} Opinion formation interpretation of the predicted non-trivial bifurcation branches.}\label{FIG: 2 agents 3 options}
\end{figure}

The dissensus subspace $W_d=V_2\otimes V_3\cong\R\otimes V_3=V_3\cong\R^2$. Indeed, each agent state space is $V_3\cong\R^2$ and $\Zz_1+\Zz_2=0$ inside the dissensus space. The symmetry group action in this subspace is the standard action of $\Z_2\times S_3 \cong \D_6$ on $\R^2$, as defined in Table~\ref{TAB: dissensus data}.
Note that the symmetry group action is isomorphic to the case $\Na=3$, $\No=2$. The conjugacy classes of axial subgroups and associated bifurcation branches are therefore the same (Figure~\ref{FIG: 2 agents 3 options}b1). However, the conjugacy classes of fixed point subspaces and their interpretation are different (Figure~\ref{FIG: 2 agents 3 options}b2). They are now given by
\begin{subequations}\label{EQ: 2a 3o dis fixed}
	\begin{align}
	\Fix(\Z_2(\kappa))&=\{z_{11}=z_{12}=-z_{21}=-z_{22}=c,\ z_{13}=-z_{23}=-2c,\ c\in\R\}\\
	\Fix(\Z_2(\tau\kappa))&=\{z_{13}=z_{23}=0,\ z_{11}=-z_{12}=-z_{21}=z_{22}=c,\ c\in\R\}.
	\end{align}
\end{subequations}
The conjugacy class $\Fix(\Z_2(\kappa))$ describes the situation in which one agent favors one of the three options and the other agent remains conflicted between the other two options. The conjugacy class $\Fix(\Z_2(\tau\kappa))$ describes the situation in which both agents remains neutral about one of the three options and each of them favors a different one between the remaining two options. Depending on second-order terms, both types of branches can be stable, but not both of them at the same time.

\subsection{Consensus opinion formation for an arbitrary number of agents and arbitrary number of options}

For generic $\Na=n$ and $\No=k$ the consensus subspace is isomorphic to $\R^{k-1}$ and the action of $\Gamma$ on this subspace is the standard action of $\ES_k$ on $V_k\cong\R^{k-1}$, as defined in Table~\ref{TAB: consensus data}. There are $\lfloor\frac{k}{2}\rfloor$ conjugacy classes of axial subgroups
\begin{equation}\label{EQ: generic consensus axial}
\Sigma_p=\ES_p\times \ES_{k-p},\quad p\leq \left\lfloor\frac{k}{2}\right\rfloor,
\end{equation}
where the first factor permutes the first $p$ options and the second factor permutes the last $k-p$ options. The associated conjugacy classes of fixed points subspaces are
\begin{equation}\label{EQ: na ko cons fixed}
\Fix(\Sigma_p)=\left\{c(\underbrace{v,\ldots,v}_{n\text{ times}}),\ v=\left(\overbrace{{\textstyle\frac{k-p}{p}},\ldots,{\textstyle\frac{k-p}{p}}}^{p\text{ times}},\overbrace{-1,\ldots,-1}^{k-p\text{ times}}\right),\ c\in\R\right\}.
\end{equation}
For $z>0$, this conjugacy class $\Fix(\Sigma_p)$ corresponds to the situations in which the agents favor the first $p$ options over the last $k-p$ options. Furthermore, the preference $\frac{k-p}{p}$ that the agent assign to the favored options increases as $p$ decreases, i.e., the smaller is the number of favored options the sharper is the way in which the agents favor them.

\begin{remark}
As in the three option consensus case discussed above (see Remark~\ref{RMK: S3 secondary}), when $p\geq 2$, we expect secondary bifurcations to sequentially lead to smaller and smaller sets of favorite options, eventually reaching consensus on a single option.	
\end{remark}

\subsection{Dissensus opinion formation for $\Na=2$ or $\No=2$}

When either $\Na=2$ and $\No=n\geq 2$, or $\No=2$ and $\Na=n\geq 2$, the action of $\Gamma$ on the dissensus space $W_d=V_n\cong\R^{n-1}$ is the action of $\ES_n\times\Z_2$ on $V_n\cong\R^{n-1}$, as defined in Table~\ref{TAB: dissensus data}. In both cases, the non-trivial element of $\Z_2$, corresponding to agent swapping if $\Na=2$ or to option swapping if $\No=2$, can be represented as minus the identity. The abstract equivariant analysis is thus the same for the two cases and can be found in \cite[Corollary~3.2]{Aronson1991}. We summarize here the relevant results. Of course, their interpretation will depend on whether $\Na=2$ or $\No=2$.

Partition $n$ into $3$ blocks, in such a way that the first two blocks posses $1\leq l\leq \frac{n}{2}$ elements each and the last block possesses $n-2l$ elements. Let $\tilde\rho_l\in\ES_n$ swap the first two blocks of coordinates and define $\rho_l=(\tilde\rho_l,-I)\in \ES_n\times\Z_2$.

\begin{theorem}\label{THM: Aronson}
	The conjugacy classes of axial subgroups of  $S_n\times\Z_2$ acting on $V_n$ and associated conjugacy classes of one-dimensional fixed point subspaces are:
	\begin{itemize}
		\item[\it I)]
		\begin{equation}\label{EQ: n 2 diss modext axial}
		\Sigma_k=\ES_k\times \ES_{n-k},
		\end{equation} 
		$1\leq k<\frac{n}{2}$, where the first factor permutes the first $k$ coordinates and the last factor permutes the last $n-k$ coordinates, with conjugacy class of fixed point subspaces
		\begin{equation}\label{EQ: n 2 diss modext axial fix}
		\Fix(\Sigma_k)=\R\Bigg\{\bigg(\overbrace{1,\ldots,1}^{k\text{ times}},\overbrace{{\textstyle\frac{k}{k-n}},\ldots,{\textstyle\frac{k}{k-n}}}^{n-k\text{ times}}\bigg)\Bigg\}.
		\end{equation}
		\item[\it II)]
		\begin{equation}\label{EQ: n 2 diss uniform axial}
		T_l=\ES_l\times \ES_l\times \ES_{n-2l}\times\Z_2(\rho_l)
		\end{equation}
		where $1\leq l\leq \frac{n}{2}$, with conjugacy class of fixed point subspaces
		\begin{equation}\label{EQ: n 2 diss uniform axial fix}
		\Fix(T_l)=\R\Bigg\{\bigg(\underbrace{1,\ldots,1}_{l\text{ times}},\underbrace{-1,\ldots,-1}_{l\text{ times}},\underbrace{0,\ldots,0}_{n-2l\text{ times}}\bigg)\Bigg\}
		\end{equation}
	\end{itemize}
	Generically, only the bifurcation branches corresponding to the axial $\Sigma_k$, with $n/3<k<n/2$ are stable at bifurcation.
\end{theorem}

Let us interpret Theorem~\ref{THM: Aronson} in terms of opinion formation. When $\Na=2$ and $\No=n$, the fixed point subspace $\Fix(\Sigma_k)$ describes the situation in which one agent has opinion $\frac{1}{\No}+c_1$, $c_1\in\R$ about the first $k$ options and opinion $1/\No-c_2$, $c_2=\frac{c_1k}{n-k}$ about the last $\No-k$ options, while the other agent has opinion $\frac{1}{\No}-c_1$ about first $k$ options and opinion $1/\No+c_2$ about the last $\No-k$ options (Figure~\ref{FIG: 2 options disagreement}). Suppose for the definiteness that $c_1>0$.  The reader can easily work out the interpretation for $c_1<0$. Observe that, for $1\leq k<\frac{n}{2}$, $c_1>c_2$. Then the first agent has a strong preference for the first $k$ options while the second agent has a weak preference for the last $\No-k$ options.
The first agent behaves with sureness in situations in which agent decisions are mutually exclusive by giving all its vote to a small number of options and strongly securing them. The second agent behaves more insecurely, by securing all the remaining options, but with weaker preference. When $\No=2$ and $\Na=n$, the fixed point subspace $\Fix(\Sigma_k)$ describes the situation in which a small group of $k$ agents (with opinion $1/2+c_1$) strongly favor the first option and a larger group of $\Na-k$ agents (with opinion $1/2-c_2$) weakly favor the second option. The elements of the small group can be considered as the extremists, who manage to avoid consensus for the other option, favored by a larger number of agents, by developing a strong preference for their favorite option. The elements of the large group can be considered as the moderates, who do not develop a strong preference but rely on their large number to try to achieve consensus. This decision behavior is reminiscent of that observed in schooling fishes~\cite{Couzin2011}.
\begin{figure}[h!]
	\centering
	\includegraphics[width=\textwidth]{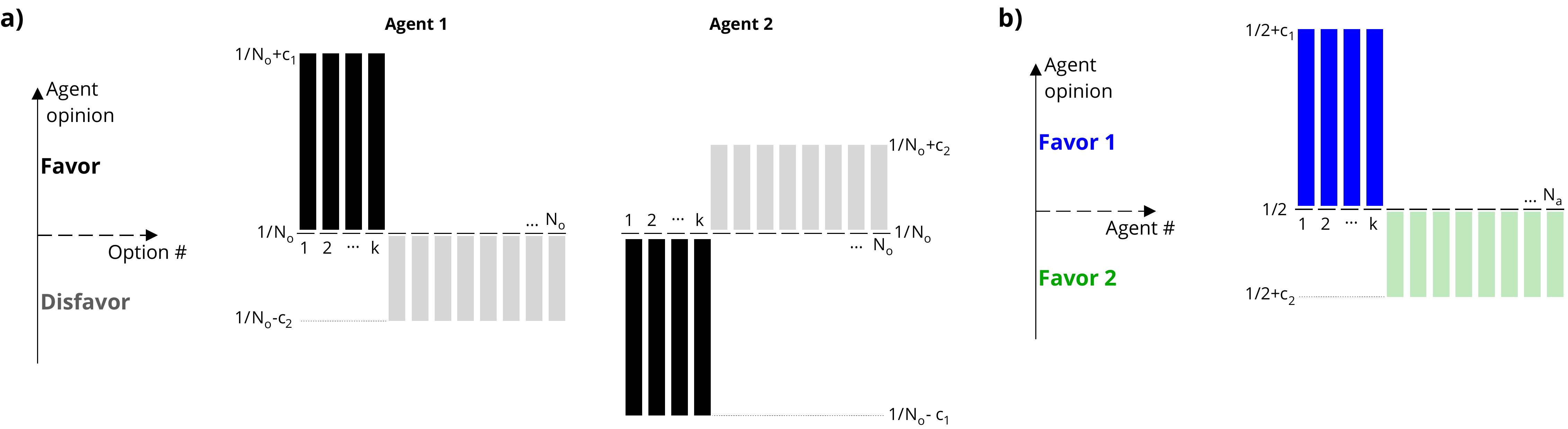}
	\caption{Interpretations of the $S_n\times\Z_2$ axial $\Sigma_k$ in terms of dissensus opinion formation. {\bf a)} $\Na=2$. {\bf b)} $\No=2$.}\label{FIG: 2 options disagreement}
\end{figure}

A similar interpretation can be worked out for the other axial. When $\Na=2$ and $\No=n$, the fixed point subspace $\Fix(T_l)$ describes the situation in which the two agents are neutral about the last $\No-2l$ options and conflicted, with exactly opposite opinions, about the first $2l$ options. In particular, the first agent favors the first $l$ options with opinion $1/\No+c$ and disfavors the second $l$ options with opinion $1/\No-c$, and viceversa for the second agent. When $\No=2$ and $\Na=n$, the first $l$ agents equally favor the first option with opinion $1/2+c$, the second $l$ agents equally favor the second option with opinion $1/2-c$, and the last $\Na-2l$ agents form an unopinionated group.
\begin{figure}[h!]
	\centering
	\includegraphics[width=\textwidth]{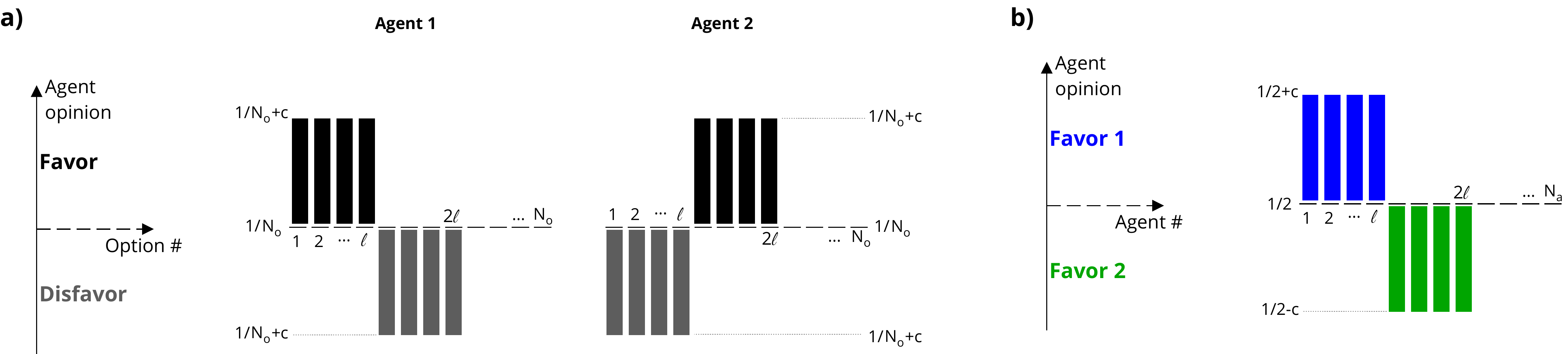}
	\caption{Interpretations of the $S_n\times\Z_2$ axial $T_l$ in terms of dissensus opinion formation. {\bf a)} $\Na=2$. {\bf b)} $\No=2$.}\label{FIG: 2 options disagreement 2}
\end{figure}

\section{Main theoretical results and their interpretation}
\label{SEC:main results}

We now state the main theoretical results of the paper and interpret them as needed. The algebraic data supporting the analysis in Sections~\ref{SSEC: 3-3 diss},~\ref{SSEC: n-3 diss}, and~\ref{SSEC: n-k diss} are summarized in Table~\ref{TAB: dissensus data NEW}.

\subsection{Existence of a dissensus quadratic equivariant for $\bs{\Na\geq 3}$ and $\bs{\No\geq 3}$}
\label{SSEC:quadratic equivariant}

Let $F_i:V_i\to V_i$ be the quadratic map defined as
$$F_i(\Zz_i)=
\begin{bmatrix}
\No(z_{i1})^2-\left((z_{i1})^2+\cdots+(z_{i\No})^2\right)\\
\vdots\\
\No(z_{i\No})^2-\left((z_{i1})^2+\cdots+(z_{i\No})^2\right)
\end{bmatrix}.
$$
If $\No\geq 3$, $F_i$ is not identically zero and $\ES_\No$-equivariant.

Let $F:V\to V$ be the quadratic map defined as
$$F(\Zz)=
\begin{bmatrix}
\Na F_1(\Zz_1)-(F_1(\Zz_1)+\cdots+F_\Na(\Zz_\Na))\\
\vdots\\
\Na F_\Na(\Zz_\Na)-(F_1(\Zz_1)+\cdots+F_\Na(\Zz_\Na))
\end{bmatrix}
$$
$F$ is $\Gamma$-equivariant. Moreover, if $\Na\geq 3$, $F|_{W_d}:W_d\to W_d$ is not identically zero.

It follows that, if $\Na\geq 3$ and $\No\geq 3$, there exists a non-zero quadratic equivariant in the dissensus space. Invoking~\cite[Theorem~2.14]{Golubitsky2002book} (see also~\cite[page~90]{Golubitsky1985-2}), we thus expect all dissensus branches predicted by the Equivariant Branching Lemma to be unstable. In turns this implies that dissensus opinion formation is expected to be switch-like whenever $\Na\geq 3$ and $\No\geq 3$.

\subsection{Dissensus axials for $\bs{\Na=\No=3}$}
\label{SSEC: 3-3 diss}
	
The main difference between the  $\Na=\No=3$ case considered here and the  $\Na\cdot\No=6$ considered in Sections~\ref{SSEC: three two} and~\ref{SSEC: two three} is that the group $\ES_3\times\ES_3$ is not isomorphic to any $\ES_k$ or $\D_k$ for any $k$. Novel types of axial conjugacy classes appear. Moreover, the existence of a quadratic equivariant, proved in Section~\ref{SSEC:quadratic equivariant}, implies that all dissensus bifurcation branches are unstable.

The following theorem classifies the axial subgroups of $\Gamma=\ES_3\times\ES_3$ acting on the dissensus space $W_d=V_3\otimes V_3\cong\R^2\otimes\R^2$, where one factor of $\Gamma$ permutes the agent index and the other factor permutes the option index, as defined in Table~\ref{TAB: dissensus data NEW}. It is a straightforward corollary of the more general Theorem~\ref{THM: D3 axials}. Let $\kappa^{\rm a}\in\ES_3$ be the order two element that swaps the first two agents and  $\kappa^{\rm o}\in\ES_3$ be the order two element that swaps the first two options. Let $\theta^{\rm a}\in\ES_3$ be the order three element that cycles forward the agents and  $\theta^{\rm o}\in\ES_3$ be the order three element that cycles forward the options. Let $\rho\in\ES_3\times\ES_3$ be the order two element defined by $\rho=(\kappa^{\rm a},\kappa^{\rm o})$. Let $\nu\in\ES_3\times\ES_3$ be the order three element defined by $\nu=(\theta^{\rm a},\theta^{\rm o})$.

\begin{theorem}\label{THM: 3-3 dissensus}
	There are two conjugacy classes of axial subgroups for the action of $\Gamma=\ES_3\times\ES_3$  on $W_d\cong\R^2\otimes\R^2$ as defined in Table~\ref{TAB: dissensus data NEW}. They are
	\begin{itemize}
		\item[\it I)]
		\begin{equation}\label{EQ: 3-3 product}
		\Z_2(\kappa^{\rm a})\times\Z_2(\kappa^{\rm o})
		\end{equation}
		with fixed point subspace
		\begin{equation}\label{EQ: 3-3 product fix}
		\Fix(\Z_2(\kappa^{\rm a})\times\Z_2(\kappa^{\rm o}))=\R\left\{(v_3,v_3,-2v_3)\right\}
		\end{equation}
		where $\kappa^{\rm o}v_3=v_3$, that is, $v_3=\in\R\left\{\left(-\frac{1}{2},-\frac{1}{2},1\right)\right\}$.
		\item[\it II)]
		\begin{equation}\label{EQ: 3-3 S3}
		\Z_3(\nu)\times\Z_2(\rho)
		\end{equation}
		with fixed point subspace
		\begin{equation}\label{EQ: 3-3 S3 fix}
		\Fix(\Z_3(\nu)\times\Z_2(\rho))=\R\left\{(v_1,v_2,v_3)\right\}
		\end{equation}
		where $\kappa^{\rm o}v_1=v_2$, $\theta^{\rm o}v_1=v_2$,  $(\theta^{\rm o})^2v_1=v_3$, that is, $v_1\in\R\left\{\left(1,-\frac{1}{2},-\frac{1}{2}\right)\right\}$, $v_2\in\R\left\{\left(-\frac{1}{2},1,-\frac{1}{2}\right)\right\}$, and $v_3\in\R\left\{\left(-\frac{1}{2},-\frac{1}{2},1\right)\right\}$.
	\end{itemize}
	Moreover, all the bifurcation branches are unstable.
\end{theorem}

The interpretation of the two subgroups is provided in Figure~\ref{FIG: 3-3}. The axial $\Z_2(\kappa^{\rm a})\times\Z_2(\kappa^{\rm o})$ describes the situation in which the disagreeing agents form two clusters (Figure~\ref{FIG: 3-3}a). One is made of two agents, the moderate, that  either develop a weak opinion toward one of the options (right) or remain conflicted between two options (left). The second cluster is made of one agent, the extremist, that develops a stronger opinion exactly opposing the moderate cluster. The axial $\Z_3(\nu)\times\Z_2(\rho)$ describes the situation in which the disagreeing agents develop symmetrically opposed, but homogeneous in strength, opinions (Figure~\ref{FIG: 3-3}b). They either favor different options (right), or remain conflicted between different pairs of options (left).

\begin{figure}
	\centering
	\includegraphics[width=\textwidth]{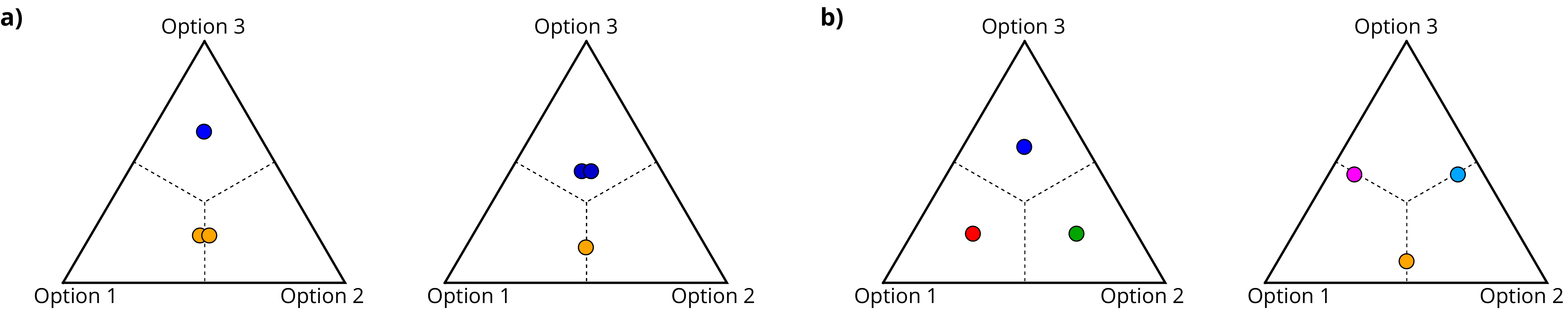}
	\caption{Dissensus decision making between three agents and three options. See text for details.}\label{FIG: 3-3}
\end{figure}

We stress that because all bifurcation branches are unstable the transition from neutrality to dissensus is discontinuous, as in Figure~\ref{FIG: switch-like}b2. In general, other types of branches, not predicted by the Equivariant Branching Lemma, might appear at and away from the singularity. However, such a situation is not generic~\cite{Field1989}, as opposed to the Equivariant Branching Lemma, which holds generically. We can therefore expect that, similarly to the three-option consensus opinion formation in Figure~\ref{FIG: switch-like}b1 and Figure~\ref{FIG: 2 agents 3 options}a, the dissensus branches predicted here bend at a fold singularity and become stable away from the singularity, thus attracting the group decision state. Simulations reveal that this is the case in the numerical model~\eqref{EQ:generic decision dynamics}.

\subsection{Dissensus axials for $\bs{\Na=3}$ or $\bs{\No=3}$}
\label{SSEC: n-3 diss}


When either $\Na=3$ and $\No=n\geq 3$, or $\No=3$ and $\Na=n\geq 3$, the action of $\Gamma$ on the dissensus space is the action of $\ES_n\times\ES_3$ on $V_n\otimes V_3\cong\R^{n-1}\otimes\R^{2}$. Each factor of $\Gamma$ permutes the indices of the corresponding tensor product factor in the natural way, as defined in Table~\ref{TAB: dissensus data NEW}. Note that the two actions specified in Table~\ref{TAB: dissensus data NEW} for $\Na=n,\No=3$ and for $\Na=3,\No=n$ are the same action. They differ in their interpretation, though, as in one case the factor $\ES_n$ permutes the agent indices and the factor $\ES_3$ permutes the option indices, and vice-versa for the other case.

Let $\sigma_m\in \ES_n$ be the order two element defined as follows. Partition $n$ into $r\geq 2$ blocks such that the first and second blocks have each the same number $m$ of elements. Let $a_{1},\ldots,a_{m}$ and $b_{1},\ldots,b_{m}$ be the elements of the first and second block, respectively. Define $\sigma_m$ as the permutation that swaps the elements of the first and second block,
\begin{equation}\label{EQ: two block swap}
\sigma_m=(a_{1}\,b_1)\cdots(a_{m}\,b_{m}).
\end{equation}
Let also $\rho_m\in \ES_n\times\ES_3$ be the order two element defined as
\begin{equation}\label{EQ: two block swap and conj}
\rho_m=(\sigma_m,\kappa).
\end{equation}
where $\kappa=(1\,2)\in\ES_3$. Let $\mu_m\in \ES_n$ be the order three elements defined as follows. Partition $n$ into $r\geq 3$ blocks such that the first three blocks have each the same number $m$ of elements. Let $a_{1},\ldots,a_{m}$, $b_{1},\ldots,b_{m}$, and $c_{1},\ldots,c_{m}$ be the element of the first, second, and third block, respectively. Define $\mu_m$ as the permutation that cycles forward the elements of the first three blocks,
\begin{equation}\label{EQ: three block cycle}
\mu_m=(a_{1}\,b_{1}\,c_{1})\cdot\ldots\cdot(a_{m}\,b_{m}\,c_{m}).
\end{equation}
Let $\nu_m\in \ES_n\times \ES_3$ be the order three element defined as
\begin{equation}\label{EQ: three block swap and cycle}
\nu_m=(\mu_m,\theta),
\end{equation}
where $\theta=(1\,2\,3)\in\ES_3$. The following theorem fully characterizes axial subgroups of $\ES_n\times\ES_3$ acting on $V_n\otimes V_3\cong\R^{n-1}\otimes\R^{2}$. Its proof is provided in Section~\ref{SupSEC: proofs}.

\begin{theorem}\label{THM: D3 axials}
	The conjugacy classes of axial subgroups of $\ES_n\times\ES_3$ acting on $V_n\otimes V_3$ as specified in Table~\ref{TAB: dissensus data NEW} are given by:
	\begin{itemize}
		\item[\it I)]
		\begin{equation}\label{EQ: n-3 product axial}
		\Sigma^\times_m =
		\ES_m\times \ES_{n-m}\times\Z_2(\kappa),
		\end{equation}
		where $1\leq m \leq n/2$, with fixed point subspace
		\begin{equation}\label{EQ: n-3 product axial fix}
		\Fix(\Sigma^\times_m)=\R\Bigg\{
		\bigg(\underbrace{v_3,\ldots,v_3}_{m\text{ times}},\underbrace{{\textstyle\frac{m}{m-n}}v_3,\ldots,{\textstyle\frac{m}{m-n}}v_3}_{n-m\text{ times}}\bigg)
		\Bigg\},
		\end{equation}
		where $\kappa v_3=v_3$, that is, $v_3\in\R\left\{\left({\textstyle-\frac{1}{2}},{\textstyle-\frac{1}{2}},1\right)\right\}$.
		
		\item[\it II)]
		\begin{equation}\label{EQ: n-3 graph Z2 axial}
		\Sigma^{\Z_2}=\ES_{n/2}\times \ES_{n/2}\times \Z_2(\rho_{n/2}),
		\end{equation}
		with fixed point subspace
		\begin{equation}\label{EQ: n-3 graph Z2 axial fix}
		\Fix(\Sigma^{\Z_2})=\R\Bigg\{
		\bigg(\underbrace{v_0,\ldots,v_0}_{\frac{n}{2}\text{ times}},\underbrace{-v_0,\ldots,-v_0}_{\frac{n}{2}\text{ times}}\bigg)
		\Bigg\},
		\end{equation}
		where $\kappa v_0=-v_0$, that is, $v_{0}\in\R\{(1,-1,0)\}$.
		
		\item[\it III)]
		\begin{equation}\label{EQ: n-3 graph D3 axial}
		\Sigma^{\ES_3}_m=\ES_m\times \ES_m\times \ES_m\times \Z_2(\rho_m) \times \Z_3(\nu_m),
		\end{equation}
		where $1\leq m\leq n/3$, with fixed point subspace
		\begin{equation}\label{EQ: n-3 graph D3 axial fix}
		\begin{array}{c}
		\Fix(\Sigma^{\ES_3}_m)=\R\Bigg\{
		\bigg(\underbrace{v_1,\ldots,v_1}_{m\text{ times}},\underbrace{v_2,\ldots,v_2}_{m\text{ times}},\underbrace{v_3,\ldots,v_3}_{m\text{ times}},\underbrace{0,\ldots,0}_{n-3m\text{ times}}\bigg)
		\Bigg\},
		\end{array}
		\end{equation}
		where $\kappa v_1=v_2$, $\theta v_1=v_2$, $\theta^2v_1=v_3$, that is, $v_1\in\R\left\{\left(1,{\textstyle-\frac{1}{2}},{\textstyle-\frac{1}{2}}\right)\right\}, v_2\in\R\left\{\left({\textstyle-\frac{1}{2}},1,{\textstyle-\frac{1}{2}}\right)\right\}, v_3\in\R\left\{\left({\textstyle-\frac{1}{2}},{\textstyle-\frac{1}{2}},1\right)\right\}$. 
		
		\item[IV)] If $n\geq 3$, then all the axial branches are generically unstable.
	\end{itemize}
\end{theorem}

\begin{remark}
	Observe that for $n=2$ the list of axial subgroups in~Theorem~\ref{THM: D3 axials} reduces to the axial subgroups of the action of $\D_6$ on $\R^2$.
\end{remark}

We start by interpreting Theorem~\ref{THM: D3 axials} in terms of opinion formation and postpone its proof to the end of the section. Each axial subgroup constructed in Theorem~\ref{THM: D3 axials} admits two interpretations, depending on whether $\Na=3$ and $\No=n$, or $\No=3$ and $\Na=n$. When $\Na=3$ and $\No=n$, the interpretation of the axial subgroup $\Sigma^{\times}_m$ is sketched in Figure~\ref{FIG: 3 options disagreement 1}a. In this representation, $\kappa\in\ES_3$ acts by swapping agents $1$ and $2$. The factor $\Z_2(\kappa)$ therefore imposes that agent $1$'s and agent $2$'s opinions are equal. The fixed point subspace of this subgroup describes the situation in which agent $3$ strongly favors the first $m$ options and weakly disfavors the last $\No-m$ options or strongly disfavors the first $m$ options and weakly favors the last $\No-m$ option. Both agents $1$ and $2$ are more moderate. Their opinions are opposite to agent $3$'s opinion, but with the half of the strength. When $\No=3$ and $\No=n$, $\kappa\in\ES_3$ acts by swapping options $1$ and $2$ and the resulting interpretation is sketched in Figure~\ref{FIG: 3 options disagreement 1}b. A small group of agents (the extremists) strongly favor (disfavor) option $3$, while a larger group of agents (the moderates) weakly disfavor (favor) option $3$ and are unopinionated between options $1$ and $2$.

\begin{figure}[h!]
	\centering
	\includegraphics[width=\textwidth]{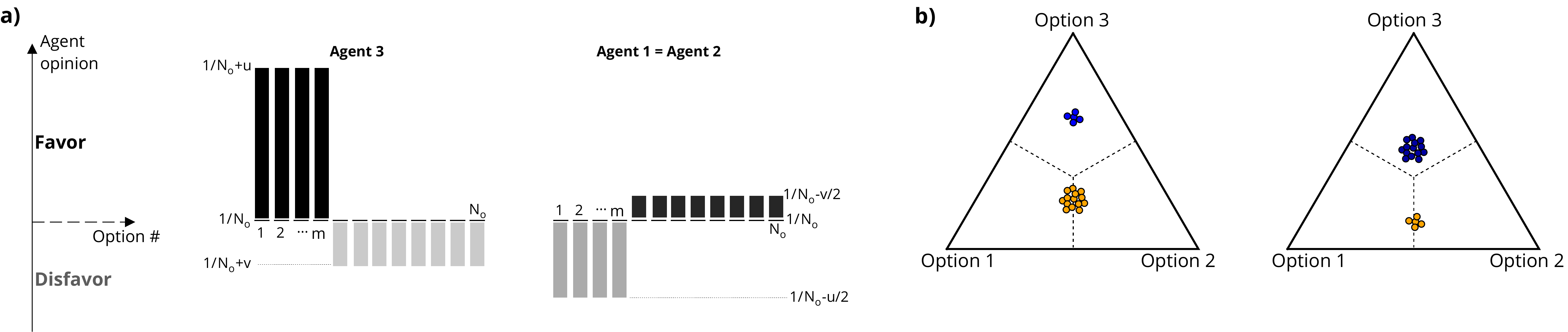}
	\caption{Interpretations of the $S_n\times\D_3$ axial $\Sigma^{\times}_m$ in terms of opinion formation. {\bf a)} $\Na=3$.  {\bf b)} $\No=3$.}\label{FIG: 3 options disagreement 1}
\end{figure}

For the axial  $\Sigma^{\Z_2}$, when $\Na=3$ and $\No=n$, the factor $\Z_2(\rho_{n/2})$ implies that agent $1$ and agent $2$ develop exactly opposite opinions about the $\No$ options. The axial fixed point subspace can then be interpreted as follows (Figure~\ref{FIG: 3 options disagreement 2}a). Agent $1$ favors the first $n/2$ options and disfavors the last $n/2$ options, while agent $2$ disfavors the first $n/2$ options and favors the last $n/2$ options. Agent $3$ remains unopinionated about all options. When $\No=3$ and $\Na=n$, the factor $\Z_2(\rho_{n/2})$ implies that the agent group splits in two equally sized groups with exactly opposite opinions. By analyzing the axial fixed point subspace, we conclude that one group favors option 1, strongly disfavors option 2, weakly disfavors option 3, while the other group strongly favors option 2, strongly disfavor option 1, and weakly disfavors option 3  (Figure~\ref{FIG: 3 options disagreement 2}b).

\begin{figure}[h!]
	\centering
	\includegraphics[width=\textwidth]{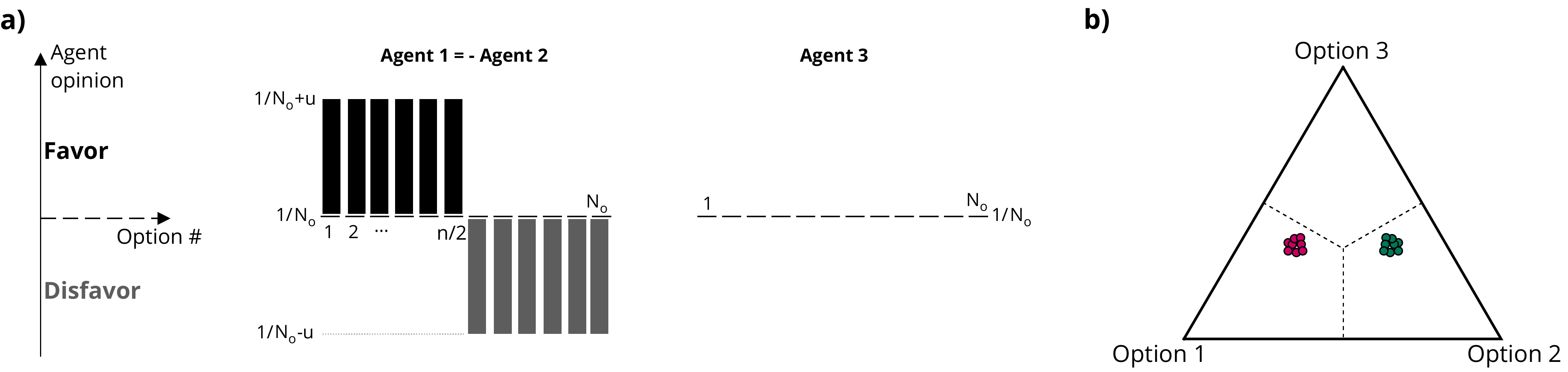}
	\caption{Interpretations of the $S_n\times\D_3$ axial $\Sigma^{\Z_2}$ in terms of opinion formation. {\bf a)} $\Na=3$.  {\bf b)} $\No=3$.}\label{FIG: 3 options disagreement 2}
\end{figure}

We finally interpret the axial $\Sigma^{\ES_3}_m$. When $\Na=3$ and $\No=n$, the factor $\Z_2(\rho_m)$ implies that the opinion that agent $1$ has about the first $m$ options is the same as the opinion that agent $2$ has about the second $m$ options, and vice-versa. The factor $\Z_3(\varrho_m)$ implies that the opinion that agent $1$ has about the first (resp. second, third) $m$ options is the same as the opinion that the second agent has about the third (resp. first, second) $m$ options and the same as the opinion that the third agent has about the second (resp. third, first) $m$ options. All agents are neutral about the last $\No-3m$ options. This interpretation is sketched in Figure~\ref{FIG: 3 options disagreement 3}a. It shows that this axial corresponds to each agent possessing an ensemble of $m$ strongly favored (resp. disfavored - not shown in the figure) options, while equally disfavoring (resp. favoring) $2m$ of the remaining options and remaining neutral about the remaining $\No-3m$ options.  When $\No=3$ and $\Na=n$ the interpretation is straightforward. $\Na-3m$ agent are unopinionated and the remaining $3m$ agent opinions are distributed in the 2 simplex with $\D_3\cong\ES_3$ (i.e., equilateral triangle) symmetry (Figure~\ref{FIG: 3 options disagreement 3}b). There are two possible configurations. In one, each group of $m$ agent has its favorite option. In the other, each group has its disfavored option and is conflicted between the remaining two options. 

\begin{figure}[h!]
	\centering
	\includegraphics[width=\textwidth]{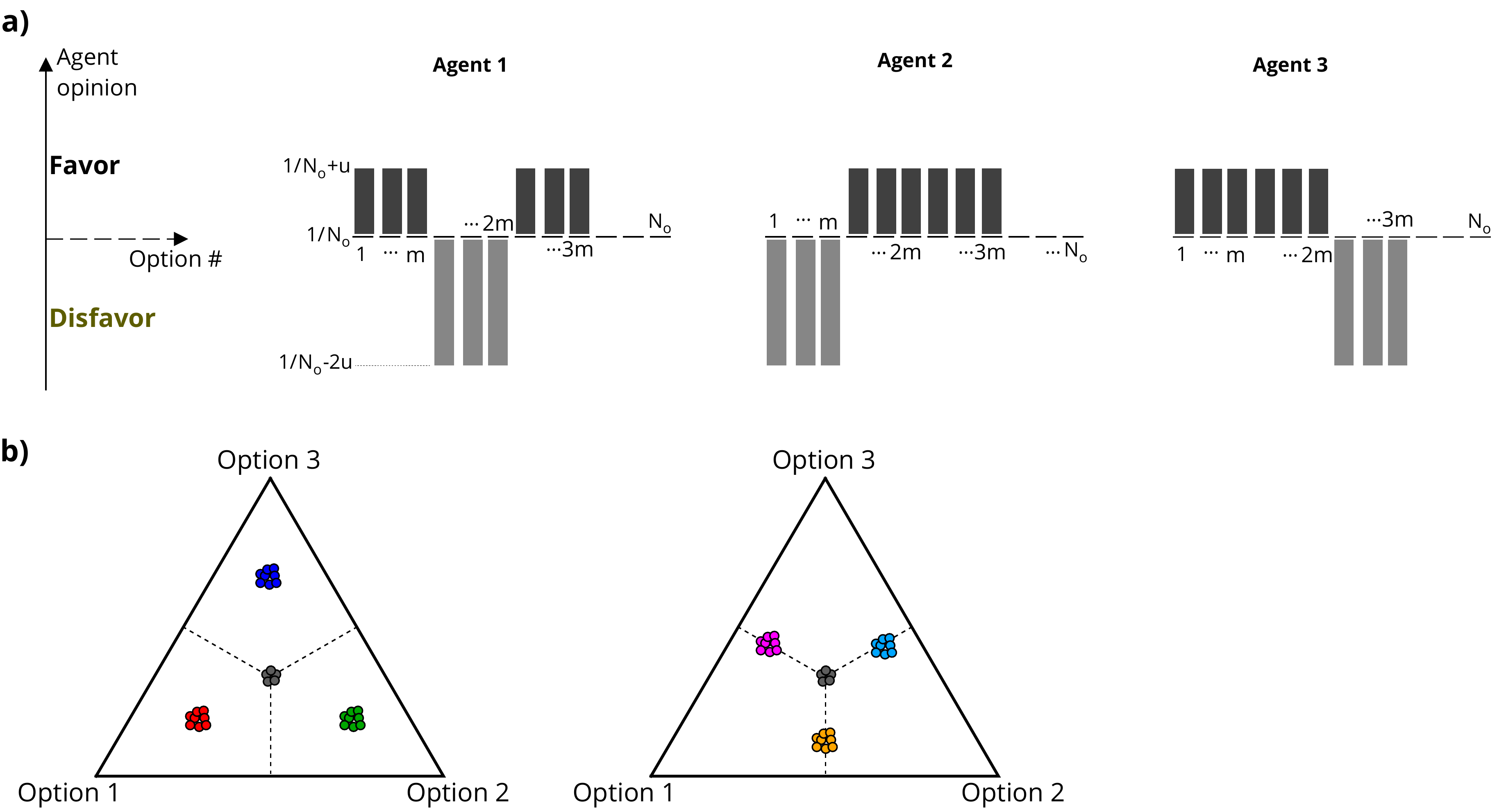}
	\caption{Interpretations of the $S_n\times\D_3$ axial $\Sigma^{\D_3}_m$ in terms of opinion formation. {\bf a)} $\Na=3$.  {\bf b)} $\No=3$.}\label{FIG: 3 options disagreement 3}
\end{figure}

\subsection{A first result for general $\bs{\Na}\geq 3$ and $\bs{\No}\geq 3$ disagreement}
\label{SSEC: n-k diss}

In the sequel we provide a list of axial subgroups for the action of $\Gamma=\ES_n\times\ES_k$ on the dissensus space $W_d=V_n\otimes V_k\cong\R^{n-1}\otimes\R^{k-1}$ as defined in Table~\ref{TAB: dissensus data NEW}.

We need some preliminary notions and definitions. The normalizer of a subgroup $\Sigma\subset\Gamma$ is
\[
N_\Gamma(\Sigma) = \{\gamma\in\Gamma: \gamma\Sigma = \Sigma\gamma\}
\]  
A well known result is $\gamma\in N_\Gamma(\Sigma)$ if and only if $\gamma (\Fix(\Sigma)) = \Fix(\Sigma)$. Indeed, if $\sigma\in\Sigma$, $\gamma\in N_\Gamma(\Sigma)$, and $x\in\Fix(\Sigma) $, then $\sigma(\gamma x)=\gamma(\sigma x)=\gamma x$, and therefore $\gamma x\in \Fix(\Sigma)$.
Define
\begin{equation}\label{EQ: normal direct product}
A\dot\times B = \{\gamma\in N_\Gamma(A\times B) : \gamma |_{\Fix(A\times B)} = \ONE_{\Fix(A\times B)}\}. 
\end{equation}
Note that $A\dot\times B$ is axial if and only if $\dim\Fix(A\times B) = 1$. Indeed, by definition, if $\gamma\in A\dot\times B$ and $z \in \Fix(A\times B)$ then $\gamma z = z$.  It follows that $\Fix(A\times B) \subset \Fix(A\dot\times B)$ and $\dim \Fix(A\times B) \leq \dim\Fix(A\dot\times B)$. However, $A\times B\subset A\dot\times B$ and therefore $\dim \Fix(A\times B) \geq \dim\Fix(A\dot\times B)$. It follows that, in general, $\dim \Fix(A\times B) = \dim\Fix(A\dot\times B)$. If $A\dot\times B$ is axial, then $\dim \Fix(A\times B) = \dim\Fix(A\dot\times B)=1$. Conversely, if $\dim \Fix(A\times B)=1$, then $\dim\Fix(A\dot\times B)=1$ and it remains to prove that $A\dot\times B$ is an isotropy subgroup. Let $\Fix(A \times B)  = \R\{z\}$ and suppose $\gamma \in \Sigma_z$.  Then, $\gamma: \Fix(A \times B) \to \Fix(A \times B)$ and therefore $\gamma\in N_\Gamma(A \times B)$.  Moreover, $\gamma$ fixes $z$ and so, by definition, $\gamma \in A \dot\times B$.  It follows that $\Sigma_z$ is a subset of $A \dot\times B$ and equallity follows because the reverse inclusion is straightforward.



Partition $\{1,\ldots,n\}$ into $r\geq 2$ blocks such that the first $s\geq 2$ blocks have each the same number $m$ of elements. Let the order 2 element $\rho_m\in\ES_n\times\ES_k$ be defined as
$$\rho_m=(\sigma_m,\sigma_{(12)}),$$
where $\sigma_m\in\ES_n$ swaps the first two blocks of vectors and $\sigma_{(12)}\in\ES_k$ swaps the first two elements of each vector. In other words, the element $\rho_m$ simultaneously swaps the first two blocks of vectors, $(\Zz_1,\ldots,\Zz_m)\,\leftrightarrow\,(\Zz_{m+1},\ldots,\Zz_{2m})$, and the first and second element of each vector. Let also $\nu_m^{(s)}\in\ES_n\times\ES_k$ be the order $s$ element defined as
$$\nu_m^{s}=(\mu_m^{(s)},\sigma_{(12\cdots s)})$$
where $\mu_m^{(s)}\in\ES_n$ cycles forward the first $s$ blocks of vectors and $\sigma_{(12\cdots s)}\in\ES_k$ cycles forward the first $s$ components of each vector. In other words, the element $\nu_m^{(s)}$ simultaneously cycles forward the first $s$ blocks of vectors and the first $s$ components of each vector. The proof of the following theorem is provided in Section~\ref{SupSEC: proofs}.

\begin{theorem}\label{THM:general axials}
	Consider the action of $\Gamma=\ES_n\times\ES_k$ on $W_d=V_n\otimes V_k$ as specified above. Then the following hold.
	\begin{itemize}
		\item[\it I)]
		$$\Sigma_{m,l}^\times = A\dot\times B$$
		is axial if and only if $A\subset \ES_n$ is an axial subgroup of the action of $\ES_n$ on $\R^{n-1}$, that is, $A=\ES_m\times\ES_{n-m}$, $1\leq m\leq \frac{n}{2}$, and $B\subset \ES_k$ is an axial subgroup of the action of $\ES_k$ on $\R^{k-1}$, that is, $B=\ES_l\times\ES_{k-l}$, $1\leq l\leq \frac{k}{2}$. The associated fixed point subspace is
	\end{itemize}
	\begin{align*}
	\Fix(A\dot\times B)=&\R\Bigg\{
	\bigg(\underbrace{v_l,\ldots,v_l}_{m\text{ times}},\underbrace{{\textstyle\frac{m}{m-n}}v_l,\ldots,{\textstyle\frac{m}{m-n}}v_l}_{n-m\text{ times}}\bigg)
	\Bigg\},
	\end{align*}
	where $v_l\in\Fix(B)$, that is, $v_l\in\R\Bigg\{\bigg(\underbrace{1,\ldots,1}_{l\text{ times}},\underbrace{{\textstyle\frac{l}{l-k}},\ldots,{\textstyle\frac{l}{l-k}}}_{k-l\text{ times}}\bigg)\Bigg\}$.\footnote{A case in which $A\dot\times B\supsetneq A\times B$ is for $n=k=4$, $m=l=2$, in which case the order-two element that simultaneously swaps the two blocks of agents and the two blocks of options is in $A\dot\times B$ (and clearly not in $A\times B$). }\vspace{4mm}
	
	Suppose that $\Sigma\in\ES_n\times\ES_k$ is axial and that
	\begin{equation}\label{EQ:graph representation}
	\Sigma=\{(\tau,\varphi(\tau)),\,\tau\in H\} 
	\end{equation}
	where $H\subset\ES_n$ is a subgroup and $\varphi:\ES_n\to\ES_k$ is a group homomorphism.
	\begin{itemize}
		\item[\it II)]
		If $\varphi(H)=\ES_s$, $2\leq s<k$, then necessarily $s=\frac{n}{m}$ and
		$$\Sigma^{\ES_s}=\underbrace{\ES_{m}\times\cdots\times \ES_{m}}_{s\text{ times}}\times\Z_2(\rho_m)\times\Z_s(\nu_m^{(s)}) $$
		and its fixed point subspace is
		$$\Fix(\Sigma)=\R\Bigg\{ 
		\Big(\underbrace{v_0,\ldots,v_0}_{m\text{ times}},\, \sigma_{(12\cdots s)}(\underbrace{v_0,\ldots,v_0}_{m\text{ times}}),\ldots,\,(\sigma_{(12\cdots s)})^{s-1}(\underbrace{v_0,\ldots,v_0}_{m\text{ times}})\Big)
		\Bigg\} $$
		where $v_0=\big( -(s-1),\underbrace{1,\ldots,1}_{s-1\text{ times}},\underbrace{0,\ldots,0}_{k-s\text{ times}}\big)$.
		
		\item[\it III)]
		If $\varphi(H)=\ES_k$, then
		$$\Sigma^{\ES_k}_m=\underbrace{\ES_m\times\cdots\times \ES_m}_{k\text{ times}}\times\Z_2(\rho_m)\times\Z_k(\nu_m^{(k)}) $$
		and its fixed point subspace is
		$$\Fix(\Sigma)=\R\Bigg\{ \hspace{-1mm}
		\bigg(\underbrace{v_1,\ldots,v_1}_{m\text{ times}},\sigma_{(12\cdots k)}(\underbrace{v_1,\ldots,v_1}_{m\text{ times}}),\ldots,(\sigma_{(12\cdots s)})^{k-1}(\underbrace{v_1,\ldots,v_1}_{m\text{ times}}),\hspace{-2mm}\underbrace{0,\ldots,0}_{n-km\text{ times}}\hspace{-2mm}\bigg)
		\hspace{-1mm}\Bigg\} $$
		where $v_1=(-(k-1),1,\ldots,1)\in\R^k$.
		
		\item[IV)] If $n,k\geq 3$ then all axial branches are unstable.
		
	\end{itemize}
\end{theorem}

Theorem~\ref{THM:general axials} provides a first list of axial subgroups for $\Gamma=\ES_n\times\ES_k$ acting on $W_d$. All axial subgroup listed in Theorems~\ref{THM: Aronson} and~\ref{THM: D3 axials} are special cases of those listed in Theorem~\ref{THM:general axials}. It remains an open question whether the list of axial subgroups in Theorem~\ref{THM:general axials} is complete. The key point in proving exhaustiveness of the axial list for $k=3$ was Lemma~\ref{lem: non product axials proj ker}, which shows that the graph representation condition~\eqref{EQ:graph representation} is satisfied. We still do not know whether the same condition is true for a general $k$.

\section{Conclusions and future directions}
\label{sec:conclusions}

We developed a new theoretical framework to describe, analyze, and predict the dynamics of opinion formation for an arbitrary number of agents and an arbitrary number of options. Our approach is model-independent and thus suitable to embrace a rich variety of biological, social, and artificial systems under the mild empirical hypothesis that they approximately satisfy suitable symmetry assumptions. The symmetry assumptions we rely on capture the situation in which agents are equal in the decision process and options are {\it a priori} equally valuable. This idealized situation is maximally symmetric and reveals all the richness of the possible opinion-formation behaviors: consensus, dissensus, their co-existence, and the emergence of prototypical types of dissensus, namely, uniform and moderate/extremist. It further reveals that
transitions between these behaviors are ruled by a small number of parameters. The cooperative or competitive nature of agent interactions is particularly important and tiny changes in the balance between agents' cooperativity and competitivity can lead to dramatic changes in their collective opinion behavior.

The key to developing our theory is equivariant bifurcation theory, which provides the model-independent tools to analyze the full richness of opinion-formation dynamics without fixing any specific set of equations. At the same time, our analysis of the specific symmetries of opinion-forming dynamics provides novel results in equivariant bifurcation theory, namely the classifications of all the axial subgroups of the action of $\ES_n \times \ES_3$ on $\R^{n-1}\otimes\R^{2}$ and the full classification of dissensus behaviors in the case of three options and arbitrary number of agents. Moreover, the vector field realization of the model-independent theory, presented in the companion paper~\cite{Bizyaeva2020}, was guided by the handful of cases in which equivariant singularity theory, and thus equivariant normal forms, have been developed, which further stress the importance of these tools for modeling high-dimensional complex behaviors like collective opinion formation.

\appendix

\section{Proofs of the main results}\label{SupSEC: proofs}

\begin{proof}[Proof of Theorem~\ref{THM: D3 axials}]
	Notice that all the subgroups listed in the statement of the theorem are axial subgroups for the given action, as evident from the associated fixed point subspace. Item {\it IV)} is a direct consequence of the existence of a dissensus quadratic equivariant~\cite[Theorem~2.14]{Golubitsky2002book} (see also~\cite[Theorem~XIII.4, page~90]{Golubitsky1985-2}). The rest of the proof aims at showing that no other axial subgroups exist.

	{\it I)}
	We begin by proving that, modulo conjugacy, if an axial subgroup $\Sigma$ satisfies $\Sigma=A\times B$ with $A\subset \ES_n$ and $B\subset \ES_3$, then necessarily $\Sigma=\Sigma^{\times}_m$ for some $m$. Up to conjugacy, $B$ is either $\ONE$, $\Z_2(\kappa)$, $\Z_3(\theta)$, or $\ES_3$. If $B=\ONE$, then it follows by \cite[Lemma~3.1]{Dionne1996} that $\Fix(\Sigma)=\Fix_{V_\Na}(A)\otimes V_3$ is even dimensional; and $\Sigma$ is not axial. If $B=\Z_3(\theta)$ or $B=\ES_3$, then $(\ONE,\theta)\in\Sigma$ and, invoking again \cite[Lemma~3.1]{Dionne1996}, 
	\begin{equation} \label{e:fix_theta}
	\Fix(\Sigma)\subset \Fix(\Z_3(\ONE,\theta))= V_\Na\otimes\Fix_{V_3}(\Z_3(\theta))=V_\Na\otimes\{0\}=\{0\}. 
	\end{equation}
	Hence, $\Sigma$ is not axial. If $B=\Z_2(\kappa)$, then
	$$\Fix_{V_3}(B)=\{y\in V_{3}:\, y_1=y_2 \}\cong\R$$
	and thus $\Fix(\Sigma)=\Fix_{V_\Na}(A)\otimes\R\cong \Fix_{V_\Na}(A)$, $A\subset\ES_n$. So $\Sigma$ is axial if $A$ is an axial subgroup of  $\ES_n$ acting on $V_\Na\cong\R^{n-1}$. Axial subgroups of $\ES_n$ acting on $V_\Na\cong\R^{n-1}$ are known~\cite{Field1989,Aronson1991} and are exactly of the form $A=\ES_m\times \ES_{n-m}$, for $1\leq m\leq n/2$.
	
	{\it II), III)} We now prove that any other axial subgroup is either of the form $\Sigma^{\Z_2}$ or $\Sigma^{\ES_3}$.
	Let $\Pi: \ES_n\times\ES_3\to \ES_n$ be the projection homomorphism. Then $H= \Pi(\Sigma)$ is a subgroup of $\ES_n$ for any axial subgroup $\Sigma$. We need two lemmas, whose proof is provided below.
	
	\begin{lemma}\label{lem: non product axials proj ker}
		Suppose $\Sigma\subset \ES_n\times \ES_3$ is axial with respect to the given action and that, up to conjugacy, $\Sigma\neq \Sigma^\times_m$, for any $1\leq m\leq n/2$. Then $\ker(\Pi|_\Sigma)=(\ONE,\ONE)$,
		and 
		\begin{equation}\label{EQ: general non product axial}
		\Sigma = \{(\tau,\varphi(\tau)) : \tau\in H\}
		\end{equation}
		where $\varphi:H\to\ES_3$ is a group homomorphism. Morever, up to conjugacy, $\varphi(H)$ can neither be $\ONE$ nor $\Z_3(\theta)$.
	\end{lemma}

	Let $L=\ker\,\varphi\subset H\subset \ES_n$. Then
	\begin{equation}\label{EQ: isotropy quotients}
	\Sigma/(L,\ONE)\cong H/L\cong \varphi(H).
	\end{equation}
	\begin{lemma}\label{LEM: ker is isotropy}
		If $\Sigma$ is an isotropy subgroup then $L$ is an isotropy subgroup.
	\end{lemma}
	
	It follows by Lemma~\ref{LEM: ker is isotropy} that $L$ is a an isotropy subgroup of $\ES_n$ acting on $\R^{n-1}$, that is~\cite{Field1989,Aronson1991}, $L=\ES_{k_1}\times\cdots\times \ES_{k_s}$ and thus
	\begin{equation}\label{EQ: Fix Ker}
	\Fix((L,\ONE))=\{(\bar v_1,\ldots,\bar v_1,\ldots,\bar v_s,\ldots,\bar v_s),\, \bar v_i\in V_3\},
	\end{equation}
	where the $i$-th block has dimension $k_i$. Notice also that $L$ is a normal subgroup of $H$, that is, $hLh^{-1}=L$ for all $h\in H$, because kernels of group homomorphisms are always normal\footnote{Suppose $G$ is a group and $\varphi:G\to G'$ is a group homomorphism. Suppose $x\in\ker, \varphi$, that is $\varphi(x)=e$. Then, given $y\in G$, we have $\varphi(yxy^{-1})=\varphi(y)\varphi(x)\varphi(y^{-1})=\varphi(y)\varphi(y)^{-1}=e$, that is, $yxy^{-1}\in\ker\,\varphi$.}.\\
	
	We now conclude the proof in for case II., i.e., $\varphi(H)=\Z_2(\kappa)$, which will provide the axial conjugacy class $\Sigma^{\Z_2}$. If $\varphi(H)=\Z_2(\kappa)$, then~\eqref{EQ: isotropy quotients} implies that $L$ has index two (i.e., the order of $\varphi(H)$) in $H$ and $(L,\ONE)$ has index two in $\Sigma$. It follows that there must exist $h\in H$, with $h^{2}\in L$, such that 
	$$\Sigma=(L,\ONE)\cup (hL,\kappa).$$
	The condition $h^2\in L$ follows by the fact that $\kappa^2=\ONE$ and thus, if $a$ is in the coset $hL$, then $(a,\kappa)^2=(a^2,\ONE)$ and $a^2$ must be in the coset $L$. Because $h$ commutes with $L$, it must swap equally sized pairs of blocks in~\eqref{EQ: Fix Ker} and up to conjugacy by elements of $L$ we can take $h$ to be of order two\footnote{Indeed, let $B_i$ be the set of coordinate indexes in the $i$-th block. Then $L(B_i)=B_i$ for all $i$, and because the permutation action is transitive no other $L$-invariant coordinate index set exists except the $B_i$'s. It follows that $L(h(B_i))=h(L(B_i))=h(B_i)$, that is $h(B_i)=B_j$ for some $j$. Moreover $h^2(B_i)=B_i$, because otherwise $h^2\notin L$}. The resulting isotropy conditions read
	$$(\bar v_{i1},\bar v_{i2},-(\bar v_{i1}+\bar v_{i2}))=(\bar v_{i2},\bar v_{i1},-(\bar v_{i1}+\bar v_{i2}))$$
	for each block that is not swapped by $h$ and 
	$$(\bar v_{i1},\bar v_{i2},-(\bar v_{i1}+\bar v_{i2}))=(\bar v_{j1},\bar v_{j2},-(\bar v_{j1}+\bar v_{j2}))$$
	for each pair of blocks that are swapped by $h$. Suppose there are $r\leq s/2$ pairs of blocks which are swapped by $h$ and $(s-2r)$ blocks which are fixed by $h$. Each block that is not swapped count for one degree of freedom. Each pair of swapped blocks count for two degrees of freedom. Furthermore, we must impose the disagreement condition condition $k_1\bar v_1+\cdots+k_s \bar v_s=0$, which only eliminates one degree of freedom because of the isotropy conditions. It follows that $\dim\Fix(\Sigma)=(s-2r)+2r-1=s-1$ and thus $s=2$, and $r=0$ or $r=1$. If $r=0$, then $h=\ONE$ and we go back to the conjugacy class of product axials $\Sigma^\times_{\frac{n}{2}}$. If $r=1$ we obtain the conjugacy class $\Sigma^{\Z_2}$.
	
	To conclude the proof for case III., i.e., $\varphi(H)=\ES_3$, we proceed similarly. By~\eqref{EQ: isotropy quotients}, there exist $h_1,h_2\in H$ such that
	$$\Sigma=(L,\ONE)\cup (h_1L,\kappa)\cup (h_2L,\theta),$$
	with $h_1^{2}\in L$, $h_2^2\in H-L$, and $h_2^3\in  L$. That is, $h_1$ swaps pairs of blocks and $h_2$ cycles triplets of blocks. Up to conjugacy by elements of $L$ we can take $h_1$ to be of order two and $h_2$ to be of order three. Any block which is not cycled by $h_2$ is zero because the isotropy condition
	$$(\bar v_{i1},\bar v_{i2},-(\bar v_{i1}+\bar v_{i2}))=(-(\bar v_{i1}+\bar v_{i2}),\bar v_{i1},\bar v_{i2})$$
	implies $\bar v_{ip}=0$, $p=1,2$. Imposing the isotropy condition associated to the element $(h_2,\theta)$ on a block triplet defined by $\bar v_i,\bar v_j,\bar v_l$ that is cycled by $h_2$ implies
	\begin{align}\label{EQ: h2 theta isotropy}
	\bar v_{i1}&=-(\bar v_{l1}+\bar v_{l2})=\bar v_{j2}\nonumber\\
	\bar v_{i2}&=\bar v_{l1}=-(\bar v_{j1}+\bar v_{j2})\\
	-(\bar v_{i1}+\bar v_{i2})&=\bar v_{l2}=\bar v_{j1}\nonumber
	\end{align}
	where, of course, the three blocks have the same size. If $h_1$ does not swap at least two of the three blocks cycled by $h_2$, then the isotropy condition associated to $\kappa$, i.e., $\bar v_{p1}=\bar v_{p2}$, $p=i,j,l$, implies that the whole block triplet is zero. Any block triplet which is cycled by $h_2$ and in which $h_1$ swaps two blocks  (up to conjugacy we can take them to be the last two), contribute one degree of freedom to the fixed point subspace. Indeed, the isotropy condition associated to $(h_1,\kappa)$ reads $\bar v_{i1}=\bar v_{j2},\bar v_{i2}=\bar v_{j2},\bar v_{l1}=\bar v_{l2}$, which, together with~\eqref{EQ: h2 theta isotropy} implies $v_i=c\left(1,{\textstyle-\frac{1}{2}},{\textstyle-\frac{1}{2}}\right), v_j=c\left({\textstyle-\frac{1}{2}},1,{\textstyle-\frac{1}{2}}\right), v_l=c\left({\textstyle-\frac{1}{2}},{\textstyle-\frac{1}{2}},1\right)$, $c\in \R$. If $\Sigma$ is axial, then there exists only one such block triplet and the conjugacy class $\Sigma^{\ES_3}$ follows.
\end{proof} 

\begin{proof}[Proof of Lemma~\ref{lem: non product axials proj ker}]
	The elements of $\ker(\Pi|_\Sigma)$ are of the form $(\ONE,\alpha)$ with $\alpha\in\ES_3$. Using the fact that conjugation leaves the kernel invariant, to prove the first part of the statement it suffices to exclude that $(\ONE,\theta)\in\Sigma$ and $(\ONE,\kappa)\in\Sigma$, because all non-trivial elements of $\ES_3$ are either conjugate to $\theta$ or $\kappa$.
	
	If $(\ONE,\theta)\in\Sigma$ then it follows by~\eqref{e:fix_theta} that $\Sigma$ is not axial. 
	
	If $(\ONE,\kappa)\in\Sigma$ two cases are possible: either $(\tau,\theta)\not\in\Sigma$ for any $\tau\in \ES_n$ or $(\tau,\theta)\in\Sigma$ for some $\tau\in \ES_n$. The former can be excluded because in that case $\Sigma = A\times\Z_2(\kappa)$ for some $A\subset \ES_n$, and it follows as in {\it I)} that $\Sigma=\Sigma^\times_m$ for some $1\leq m< n/2$. The second case can also be excluded because, if $(\ONE,\kappa)\in\Sigma$, then
	$$\Fix(\Sigma)\subset V_\Na\otimes \{y\in V_{3}:\, y_1=y_2 \},$$
	and the generic isotropy condition imposed by the element $(\tau,\theta)$ reads $(z_{i1},z_{i1},-2z_{i1})=(-2z_{\tau(i)1},z_{\tau(i)1},z_{\tau(i)1})$, which implies $z_{i1}=z_{\tau(i)1}=0$ and thus $\Zz_j=0$ for all $j$. Therefore $\Sigma$ is not axial.
	
	It follows by the First Isomorphism Theorem that $\Sigma/\ker(\Pi|_\Sigma)=\Sigma$ is isomorphic to $H$, and therefore~\eqref{EQ: general non product axial} holds. Up to conjugacy, the image $\varphi(H)$ can be either $\ONE$, $\Z_2(\kappa)$, $\Z_3(\theta)$, or $\ES_3$. If $\varphi(H)=\ONE$, then $\Sigma=A\times \ONE$, with $A\subset \ES_n$, a case that was already excluded. If $\varphi(H)=\Z_3(\theta)$, then $\Sigma$ only contains elements of $\ES_n$ and elements of the form $(\tau,\theta)$, for some $\tau\in \ES_n$. The associated isotropy conditions take the form $(z_{i1},z_{i2},-(z_{i1}+z_{i2}))=(-(z_{\tau(i)1}+z_{\tau(i)2}),z_{\tau(i)1},z_{\tau(i)2})$ whose solution space in even dimensional
	and therefore $\Sigma$ cannot be axial.
\end{proof}

\begin{proof}[Proof of Lemma~\ref{LEM: ker is isotropy}]
	Let $x$ be a vector which is fixed by $\Sigma$. Then, in particular, $x$ is fixed by $L$, because $(L,\ONE)\subset\Sigma$. Let $\gamma$ be an element of $\ES_n$ which also fixes $x$. Then $(\gamma,\ONE)$ must be in $\Sigma$ because $\Sigma$ is an isotropy subgroup, and therefore $\gamma$ is in $\ker\,\varphi=L$.
\end{proof}

\begin{proof}[Proof of Theorem~\ref{THM:general axials}]
	We first prove Item~I. of the statement and then sketch the proof for  Items~II. and~III.
	
	{\it I)} The following lemma follows from Lemma~3.1 in~\cite{Dionne1996}
	\begin{lemma} \label{P:AtimesB}
		Let $A\subset\ES_n$ and $B\subset\ES_k$ be subgroups.  Then 
		\[
		\dim_\WW(\Fix(A\times B)) = 1 
		\]
		if and only if
		\[
		\dim_{\R^{n-1}}(\Fix(A)) = 1 \qquad \mbox{and} \qquad  \dim_{\R^{k-1}}(\Fix(B)) = 1
		\]
	\end{lemma}
	If $A$ and $B$ are axial subgroups, then $\dim\Fix_\WW(A\times B)=1$ and Proposition~\ref{P:AtimesB} implies that $A\dot\times B$ is axial.  Conversely, if $A\dot\times B$ is axial, then 
	$\dim\Fix_\WW(A\dot\times B)= \dim\Fix_\WW(A\times B)$ and Lemma~\ref{P:AtimesB} implies that $A$ and $B$ are axial. 
	
	{\it II), III)}
	Under the hypothesis that the axial subgroup can be written as the graph of a homomorphism~\eqref{EQ:graph representation}, we can find its general form by following the same ideas as the case $k=3$ above. If  $\varphi(H)=\ES_s$, $2\leq s\leq k$, then $L=\ker\varphi$ has index $s!$ (i.e., the order of $\varphi(H)$) in $H$ and $(L,\ONE)$ has index $s!$ in $\Sigma$. It follows that, up to conjugacy, there must exist and order two element $h_1\in H$ and an order $s$ element $h_2\in H$ such that
	$$\Sigma=(L,\ONE)\cup (h_1L,\sigma_{(12)})\cup (h_2L,\sigma_{(12\cdots s)}).$$
	Moreover, invoking Lemma~\ref{LEM: ker is isotropy}, $L=\ES_{k_1}\times\cdots\times \ES_{k_r}$ with $$ \Fix(L)=\{(c_1,\ldots,c_1,\ldots,c_r,\ldots,c_r),\, c_i\in V_k\},$$
	and $h_1$ and $h_2$ must permute (equally-sized) blocks of vectors in $\Fix(L)$ because, since $L$ is normal, they both commute with $L$. By counting dimensions in the resulting isotropy conditions the result follows.
\end{proof}

\section{Algebraic for axial subgroup computation}

\afterpage{%
	
	\clearpage
	
	\thispagestyle{empty}
	
	\begin{landscape}

		\begin{table}
			
			\small
			\centering
			
			\begin{tabular}{l l l l c c c}
				
				\hline
				
				
				$\Na$ & $\No$ & $W_c$ & Action of $\Gamma$ & Axial $\Sigma$ & $\Fix(\Sigma)$ & Reference \Bstrut\Tstrut \\

				\hline
				
				3 & 2 & $\R\{(1,1,1)\}$ & $\Gamma = \Z_2(\tau)$ & $\ONE$ & $W_c$ & \cite[Chapter~VI]{Golubitsky1985} \Tstrut \\
				
				& & & $\tau \Zz=-\Zz$ & & \Bstrut\\

				\hline
				
				2 & 3 & $\R\{u_1,u_2\}$ & $\Gamma = \ES_3(\kappa,\rho)$ & $\Z_2(\kappa)$  & \eqref{EQ: 2a 3o cons fixed} & \cite[Section~XIII.5]{Golubitsky1985-2} \Tstrut \\
				
				& & $u_i = (v_i,v_i)\in (\R^3)^2 $ & $\kappa\Zz=(\kappa\Zz_1,\kappa\Zz_2)$ & &\\
				
				& & & $\rho\Zz=(\rho\Zz_1,\rho\Zz_2)$ & & \\
				
				& & &  $\kappa\Zz_i=(z_{i2},z_{i1},z_{i3})$  \\

				& & & $\rho\Zz_i=(z_{i3},z_{i1},z_{i2})$

				\Bstrut\\

				\hline
				
				$n$ & $k$ & $\R\{u_1,\ldots,u_{k-1}\}$ & $\Gamma = \ES_k(\kappa,\rho)$ & $\Sigma_p=\ES_p\times\ES_{k-p}$ & \eqref{EQ: na ko cons fixed} & \cite{Elmhirst2004} \Tstrut \\
				
				& & $u_i=(v_i,\ldots,v_i)\in (\R^k)^n$ & $\kappa\Zz=(\kappa\Zz_1,\ldots,\kappa\Zz_n)$ &  $1\leq p\leq\lfloor\frac{k}{2}\rfloor$ \\
				
				& &  & $\rho\Zz=(\rho\Zz_1,\ldots,\rho\Zz_n)$  &          \eqref{EQ: generic consensus axial}  \\
				
				& &   &  $\kappa\Zz_i=(z_{i2},z_{i1},z_{i3},\ldots,z_{ik})$  &  \\       
				
				& & & $\rho\Zz_i=(z_{ik},z_{i1},z_{i2},\ldots,z_{i(k-1)})$ & \Bstrut \\
				
				\hline\vspace{-3mm}
				
			\end{tabular}
			
			\caption{Algebraic data for consensus opinion formation. In the table, $v_1,\ldots,v_{k-1}\in\R^k $ denotes a basis of $V_k$, as defined in~\eqref{EQ: zero sum subspace}. Algebraic data whose expression was too big to be put in the table is defined in the text at the given reference. }\label{TAB: consensus data}
			
		\end{table}

	\end{landscape}
	
	\clearpage
	
}

\afterpage{%
	\clearpage
	\thispagestyle{empty}
	\begin{landscape}
		
		\begin{table}
			
			\small
			
			\centering
			
			\begin{tabular}{l l l l c c c}
				\hline
				
				$\Na$ & $\No$ & $W_d$ & Action of $\Gamma$ & Axial $\Sigma$ & $\Fix(\Sigma)$ & Reference \Bstrut\Tstrut \\
				
				\hline
				
				3 & 2 & $\R\{v_1,v_2\}$ & $\Z_2(\tau)\times \ES_3(\kappa,\rho)\cong\D_6$ & $\Z_2(\kappa)$ & (\ref{EQ: 3a 2o dis fixed}a) & \cite[Section~XIII.5]{Golubitsky1985-2} \Tstrut \\
				
				& & & $\tau \Zz=-\Zz$
				& $\Z_2(\tau\kappa)$ & (\ref{EQ: 3a 2o dis fixed}b) \\
				
				& & & $\kappa(\Zz_1,\Zz_2,\Zz_3)=(\Zz_{2},\Zz_{1},\Zz_{3})$ & & \\
				
				& & & $\rho(\Zz_1,\Zz_2,\Zz_3)=(\Zz_{3},\Zz_{1},\Zz_{2})$ & & \Bstrut \\
				
				\hline
				
				2 & 3 & $\R\{u_1,u_2\}$ & $\Z_2(\tau)\times \ES_3(\kappa,\rho)\cong\D_6$ & $\Z_2(\kappa)$ & (\ref{EQ: 2a 3o dis fixed}a)  & \cite[Section~XIII.5]{Golubitsky1985-2} \Tstrut \\
				
				& & $u_i=(v_i,-v_i)\in (\R^3)^2$ & $\tau(\Zz_1,\Zz_2)=(\Zz_2,\Zz_1)=-\Zz$ & $\Z_2(\tau\kappa)$ &  (\ref{EQ: 2a 3o dis fixed}b)  \\
				
				& & & $\kappa\Zz=(\kappa\Zz_1,\kappa\Zz_2)$ & & \\
				
				& & & $\rho\Zz=(\rho\Zz_1,\rho\Zz_2)$ & & \\
				
				& & & $\kappa\Zz_i=(z_{i2},z_{i1},z_{i3})$ \\
				
				& & & $\rho\Zz_i=(z_{i3},z_{i1},z_{i2})$ 
				\Bstrut \\

				\hline
				
				$n$ & 2 & $\R\{v_1,\ldots,v_{n-1}\}$ & $ \Z_2(\tau)\times \ES_n(\kappa,\rho)$ & $\Sigma_k$ \eqref{EQ: n 2 diss modext axial} & \eqref{EQ: n 2 diss modext axial fix} & \cite{Aronson1991} \Tstrut \\
				
				& & & $\tau\Zz=-\Zz$ &  $T_l$  \eqref{EQ: n 2 diss uniform axial} & \eqref{EQ: n 2 diss uniform axial fix} \\
				
				& & & $\kappa\Zz=(\Zz_2,\Zz_1,\Zz_3,\ldots,\Zz_n)$ & \\
				
				& & & $\rho\Zz=(\Zz_n,\Zz_1,\Zz_2,\ldots,\Zz_{n-1})$ & \Bstrut\\
				
				\hline
				
				2 & $n$ & $\R\{u_1,\ldots,u_{n-1}\}$ & $\Z_2(\tau)\times \ES_n(\kappa,\rho)$   &$\Sigma_k$ \eqref{EQ: n 2 diss modext axial} & \eqref{EQ: n 2 diss modext axial fix} & \cite{Aronson1991}  \Tstrut\\ 
				
				& & $u_i=(v_i,-v_i)\in(\R^n)^2$ & $\tau(\Zz_1,\Zz_2)=(\Zz_2,\Zz_1)=-\Zz$ & $T_l$  \eqref{EQ: n 2 diss uniform axial} & \eqref{EQ: n 2 diss uniform axial fix} \\
				
				& & & $\kappa\Zz=(\kappa\Zz_1,\ldots,\kappa\Zz_n)$ &  \\
				
				& &  & $\rho\Zz=(\rho\Zz_1,\ldots,\rho\Zz_n)$ & \\
				
				& & & $\kappa\Zz_i=(z_{i2},z_{i1},z_{i3},\ldots,z_{in})$ \\
				
				& & & $\rho\Zz_i=(z_{in},z_{i1},z_{i2},\ldots,z_{i(n-1)})$ \\
				
				\hline\vspace{-3mm}
			\end{tabular}
			\caption{Algebraic data for dissensus opinion formation. In the table, $v_1,\ldots,v_{n-1}\in\R^n $ denotes a basis of $V_n$, as defined in~\eqref{EQ: zero sum subspace}. Algebraic data whose expression was too big to be put in the table is defined in the text at the given reference. }\label{TAB: dissensus data}
		\end{table}
		
	\end{landscape}
	\clearpage
}

\afterpage{%
	\clearpage
	\thispagestyle{empty}
	\begin{landscape}
		
		\begin{table}
			
			\small
			
			\centering
			
			\begin{tabular}{l l l l c c }
				\hline
				$\Na$ & $\No$ & $W_d$ & Action of $\Gamma$ & Axial $\Sigma$ & $\Fix(\Sigma)$  \Bstrut\Tstrut \\ 
				
				\hline
				
				3 & 3 & $\R\{u_{11},u_{12},u_{21},u_{22}\}$ & $\Gamma=\ES_3(\kappa,\theta)\times\ES_3(\bar\kappa,\rho)$ & $\Z_2\times\Z_2$ \eqref{EQ: 3-3 product} &  \eqref{EQ: 3-3 product fix} \Tstrut \\
				
				& & $u_{ij}=v_j\otimes w_i\in\R^3\otimes\R^3\cong (\R^3)^3$ & $\kappa(\Zz_1,\Zz_2,\Zz_3)=(\kappa\Zz_1,\kappa\Zz_2,\kappa\Zz_3)$
				& $\Z_3\times\Z_2$ \eqref{EQ: 3-3 S3} & \eqref{EQ: 3-3 S3 fix} \\
				
				& & & $\theta(\Zz_1,\Zz_2,\Zz_3)=(\theta\Zz_1,\theta\Zz_2,\theta\Zz_3)$ & & \\
				
				& & & $\kappa\Zz_i=(z_{i2},z_{i1},z_{i3})$ & &  \\
				
				& & & $\theta\Zz_i=(z_{i3},z_{i1},z_{i2})$ & &  \\
				
				& & & $\bar\kappa(\Zz_1,\Zz_2,\Zz_3)=(\Zz_{2},\Zz_{1},\Zz_{3})$ & & \\
				
				& & & $\rho(\Zz_1,\Zz_2,\Zz_3)=(\Zz_{3},\Zz_{1},\Zz_{2})$ & & \Bstrut \\
				
				\hline
				
				$n$ & 3 & $\R\{u_{11},u_{12},u_{21},u_{22},\ldots, u_{(n-1)1},u_{(n-1)2}\}$ & $\Gamma=\ES_3(\kappa,\theta)\times\ES_n(\bar\kappa,\rho)$ & $\Sigma_m^{\times}$ \eqref{EQ: n-3 product axial} &   \eqref{EQ: n-3 product axial fix} \Tstrut \\
				
				& & $u_{ij}=v_j\otimes w_i\in\R^3\otimes\R^n\cong(\R^3)^n$ &  $\kappa(\Zz_1,\ldots,\Zz_n)=(\kappa\Zz_1,\ldots,\kappa\Zz_n)$
				& $\Sigma^{\Z_2}$ \eqref{EQ: n-3 graph Z2 axial} & \eqref{EQ: n-3 graph Z2 axial fix} \\
				
				& & & $\theta(\Zz_1,\ldots,\Zz_n)=(\theta\Zz_1,\ldots,\theta\Zz_n)$ & $\Sigma_m^{\ES_3}$ \eqref{EQ: n-3 graph D3 axial} & \eqref{EQ: n-3 graph D3 axial fix} \\
				
				& & & $\kappa\Zz_i=(z_{i2},z_{i1},z_{i3})$ &  &  \\
				
				& & & $\theta\Zz_i=(z_{i3},z_{i1},z_{i2})$ & &   \\
				
				& & & $\bar\kappa(\Zz_1,\ldots,\Zz_n)=(\Zz_{2},\Zz_{1},\Zz_{3},\ldots,\Zz_n)$ &   \\
				
				& & & $\rho(\Zz_1,\ldots,\Zz_n)=(\Zz_{n},\Zz_{1},\Zz_{2},\ldots,\Zz_{n-1})$ & \Bstrut \\

				\hline
				
				3 & $n$ & $\R\{u_{11},\ldots, u_{1(n-1)}, u_{21},\ldots, u_{2(n-1)}\}$ & $\Gamma=\ES_3(\kappa,\theta)\times\ES_n(\bar\kappa,\rho)$ & $\Sigma_m^{\times}$ \eqref{EQ: n-3 product axial} &   \eqref{EQ: n-3 product axial fix} \Tstrut \\
				
				& & $u_{ij}=w_j\otimes v_i\in\R^n\otimes\R^3\cong(\R^n)^3$ &  $\kappa(\Zz_1,\Zz_2,\Zz_3)=(\Zz_2,\Zz_1,\Zz_3)$
				&  $\Sigma^{\Z_2}$ \eqref{EQ: n-3 graph Z2 axial} & \eqref{EQ: n-3 graph Z2 axial fix}  \\
				
				& & & $\theta(\Zz_1,\Zz_2,\Zz_3)=(\Zz_3,\Zz_1,\Zz_2)$ & $\Sigma_m^{\ES_3}$ \eqref{EQ: n-3 graph D3 axial} & \eqref{EQ: n-3 graph D3 axial fix}   \\
				
				& & & $\bar\kappa(\Zz_1,\Zz_2,\Zz_3)=(\bar\kappa\Zz_1,\bar\kappa\Zz_2,\bar\kappa\Zz_3)$ & & \\
				
				& & & $\rho(\Zz_1,\Zz_2,\Zz_3)=(\rho\Zz_1,\rho\Zz_2,\rho\Zz_3)$ & &  \\
				
				& & & $\kappa\Zz_i=(z_{i2},z_{i1},z_{i3},\ldots,z_{in})$ &  &   \\
				
				& & & $\rho\Zz_i=(z_{in},z_{i1},z_{i2},\ldots,z_{i(n-1)})$ &   & 
				
				\Bstrut \\
				
				\hline
				
				$n$ & $k$ & $\R\{u_{11},\ldots, u_{1(k-1)}, u_{21},\ldots, u_{2(k-1)},\ldots,u_{n1},\ldots, u_{n(k-1)}\}$ & $\Gamma=\ES_n(\kappa,\theta)\times\ES_n(\bar\kappa,\rho)$ & ${\displaystyle\star}$ &   \Tstrut \\
				
				& & $u_{ij}=v_j\otimes w_i\in\R^k\otimes\R^{n}\cong(\R^k)^n$ &  $\kappa(\Zz_1,\ldots,\Zz_n)=(\Zz_2,\Zz_1,\Zz_3,\ldots,\Zz_n)$
				& $\Sigma_m^{\times}$ \eqref{EQ: n-3 product axial}  &  \eqref{EQ: n-3 product axial fix} \\
				
				& & & $\theta(\Zz_1,\ldots,\Zz_n)=(\Zz_n,\Zz_1,\Zz_2,\ldots,\Zz_{n-1})$ & $\Sigma^{\Z_2}$ \eqref{EQ: n-3 graph Z2 axial}  &\eqref{EQ: n-3 graph Z2 axial fix}   \\
				
				& & & $\bar\kappa(\Zz_1,\ldots,\Zz_n)=(\bar\kappa\Zz_1,\ldots,\bar\kappa\Zz_n)$ & $\Sigma_m^{\ES_3}$ \eqref{EQ: n-3 graph D3 axial} &  \eqref{EQ: n-3 graph D3 axial fix}  \\
				
				& & & $\rho(\Zz_1,\ldots,\Zz_n)=(\rho\Zz_1,\ldots,\rho\Zz_n)$ & &  \\
				
				& & & $\kappa\Zz_i=(z_{i2},z_{i1},z_{i3},\ldots,z_{ik})$ &  &   \\
				
				& & & $\rho\Zz_i=(z_{ik},z_{i1},z_{i2},\ldots,z_{i(k-1)})$ &   & 
				
				\Bstrut \\
				
				\hline\vspace{-3mm}
			\end{tabular}
			\caption{Algebraic data for dissensus opinion formation. In the table, $v_1,\ldots,v_k\in\R^k $ denotes a basis of $V_k$ and $w_1,\ldots,w_{n-1}\in\R^n $ denotes a basis of $V_n$, where $V_k$ and $V_n$ are defined as in~\eqref{EQ: zero sum subspace}. Given two matrices $A$ and $B$ of arbitrary sizes, $A\otimes B$ denotes the Kronecker product between them. $\star$ the list of axial might not be complete.  Algebraic data whose expression was too big to be put in the table is defined in the text at the given reference.}\label{TAB: dissensus data NEW}
		\end{table}
		
	\end{landscape}
	\clearpage
}

\bibliographystyle{siamplain}
\bibliography{references}
\end{document}